\documentclass[11pt]{article}
\usepackage[small,compact]{titlesec}
\usepackage{authblk}
\usepackage{amsmath,amssymb,amsfonts,mathabx,setspace,bm,bbm}
\usepackage{graphicx,caption,epsfig,subfigure,epsfig,wrapfig}
\usepackage{url,color,verbatim,algorithmic}
\usepackage[dvipsnames]{xcolor}
\usepackage[ruled,boxed]{algorithm}
\usepackage{cases} 
\usepackage[overload]{empheq}
\usepackage{mathtools} 
\usepackage{enumitem}
\usepackage{tabulary}
\usepackage{booktabs}
\usepackage[scaled]{helvet}
\usepackage{cancel}
\usepackage[T1]{fontenc}
\usepackage[bookmarks=true, bookmarksnumbered=true, colorlinks=true,   pdfstartview=FitV,
linkcolor=blue, citecolor=blue, urlcolor=blue]{hyperref}
\usepackage{soul}
\usepackage[sort,compress]{cite}
\usepackage{tikz}
\usetikzlibrary{positioning}
\usetikzlibrary{arrows,shapes}
\usetikzlibrary{decorations.markings}
\usetikzlibrary{arrows.meta}
\usetikzlibrary{shapes.multipart}
\usetikzlibrary{shapes.geometric}
\usepackage{adjustbox}
\usepackage{multirow}

\def\R{\mathbb{R}}                            
\def\P{\mathbb{P}}
\def\E{\mathbb{E}}

\def\calL{\mathcal{L}}

\def\br{\textbf{r}}

\newcommand{\rbracket}[1]{\left(#1\right)}
\newcommand{\sbracket}[1]{\left[#1\right]}
\newcommand{\norm}[1]{\left\|#1\right\|}
\newcommand{\normv}[1]{\left| #1\right|}
\newcommand{\innerp}[1]{\langle{#1}\rangle}

\newcommand{\floor}[1]{\lfloor{#1}\rfloor}

\def\btheta{\boldsymbol{\theta}}
\def\bPhi{\boldsymbol{\Phi}}
\def\boldeta{\boldsymbol{\eta}}
\def\gphi{g}

\newtheorem{theorem}{Theorem}

\newtheorem{assumption}[theorem]{Assumption}

\newtheorem{definition}[theorem]{Definition}

\newtheorem{lemma}[theorem]{Lemma}
\newtheorem{proposition}[theorem]{Proposition}
\newtheorem{remark}[theorem]{Remark}
\newenvironment{proof}[1][Proof]{\noindent\textbf{#1.} }{\ \rule{0.5em}{0.5em}}

\def\tcb{\textcolor{blue}}

\renewcommand{\Delta}{\triangle}
\DeclareMathOperator*{\argmin}{arg\,min}

\def\calC{\mathcal{C}}
\def\calE{\mathcal{E}}
\def\calK{\mathcal{K}}

\def\calX{\mathcal{X}}
\def\calH{\mathcal{H}}
\def\b1{\mathbf{1}}

\def\LGbar{ {\mathcal{L}_{\overline{G}}}  }

\def\Lbar{ {\bar{\mathcal{L}}}  }

\def\supp{\mathrm{supp}}

\def\R{\mathbb{R}}                            
\def\P{\mathbb{P}}

\def\bfA{\mathbf{A}}

\def\bfb{\mathbf{b}}

\def\bfr{\mathbf{r}}

\def\KL{\textup{KL}}

\definecolor{darkblue}{rgb}{0,0,0.7}

\definecolor{darkgreen}{rgb}{0.01,0.75,0.24}

\let\underbrace\LaTeXunderbrace

\numberwithin{theorem}{section}
\numberwithin{equation}{section}

\allowdisplaybreaks[3]

\title{Optimal minimax rate of learning nonlocal interaction kernels}
\author[1]{Xiong Wang} 
\author[2]{Inbar Seroussi}
\author[1]{Fei Lu \footnote{Emails:\, xiongwang@jhu.edu,\, inbarser@gmail.com,\, feilu@math.jhu.edu}}
\affil[1]{Department of Mathematics, Johns Hopkins University, Baltimore, USA. }
\affil[2]{Department of Applied Mathematics, Tel Aviv University, Tel Aviv, Israel. }
\date{}
\begin{document}

\maketitle \vspace{-8mm}

\begin{abstract}
Nonparametric estimation of nonlocal interaction kernels is crucial in various applications involving interacting particle systems. The inference challenge, situated at the nexus of statistical learning and inverse problems, arises from the nonlocal dependency. A central question is whether the optimal minimax rate of convergence for this problem aligns with the rate of $M^{-\frac{2\beta}{2\beta+1}}$ in classical nonparametric regression, where $M$ is the sample size and $\beta$ represents the regularity index of the radial kernel. Our study confirms this alignment for systems with a finite number of particles.

We introduce a tamed least squares estimator (tLSE) that achieves the optimal convergence rate when $\beta\geq 1/4$ for a broad class of exchangeable distributions by leveraging random matrix theory and Sobolev embedding. The upper minimax rate relies on fourth-moment bounds for normal vectors and nonasymptotic bounds for the left tail probability of the smallest eigenvalue of the normal matrix. The lower minimax rate is derived using the Fano-Tsybakov hypothesis testing method. Our tLSE method offers a straightforward approach for establishing the optimal minimax rate for models with either local or nonlocal dependency. 
\end{abstract}


%

\tableofcontents

\section{Introduction}
Consider the nonparametric regression of the radial \emph{interaction kernel} $\phi: \R^+\to\R$ in the model 
\begin{align}\label{eq:model}
Y=R_\phi[X]+ \eta 
\end{align}
from data consisting of samples $\{ (X^m,Y^m)\}_{m=1}^M$ of the joint distribution of $(X,Y)$. Here, $Y$ and $X$ are ${\R}^{N\times d}$-valued random variables with $N\geq 3$, denoted by $Y=(Y_1,\ldots, Y_N)^\top$ and $X=(X_1,\ldots, X_N)^\top$. The operator $R_\phi[X]=(R_\phi[X]_1,\cdots,R_\phi[X]_N)^\top$ represents the interaction between particles through the kernel $\phi$, its entries are defined by
\begin{equation}\label{eq:opt_model_R}
R_\phi[X]_i = \frac{1}{N-1} \sum\nolimits_{j\neq i} \phi(\|X_i-X_j\|_{\R^d}) \frac{X_i-X_j}{\|X_i-X_j\|_{\R^d}}\in\R^d , \quad i= 1,\ldots\, N, 
\end{equation} 
where we write $\sum_{j\neq i}:=\sum_{j=1, j\neq i}^N$. 
The noise $\eta$ is independent of $X$ and not necessarily Gaussian.

Nonparametric regression is particularly suitable for estimating the kernel $\phi$, thanks to the linear dependence of $R_\phi$ on $\phi$. The estimator, which is solved by least squares, is the minimizer of an empirical mean-square loss function 
\begin{align}\label{eq:lossFn}
	 \calE_{M}(\phi) = \frac{1}{M} \sum\nolimits_{m=1}^M \frac{1}{N} \|Y^m-R_{\phi}[X^m]\|_{\R^{Nd}}^2\, 
\end{align}
over a hypothesis space that is adaptively chosen to avoid underfitting and overfitting. 

The above nonparametric regression problem arises in the inference for systems of interacting particles or agents. Such systems are prevalent in collective dynamics in various fields, including flocking \cite{CS07,AH10,cattiaux2018_StochasticCucker}, opinion dynamics \cite{MT14}, kinetic granular media \cite{carrillo2003_KineticEquilibration,cattiaux2007_ProbabilisticApproach}, to name just a few. Driven by the applications, the past decade has seen a burst of efforts in inferring the system from data, including parametric \cite{della2023lan,messenger2022learning,liu2020_ParameterEstimation}, semi-parametric \cite{belomestny2023semiparametric}, and nonparametric \cite{della2022nonparametric,yao2022mean,LZTM19pnas,LMT21_JMLR} approaches. Given often limited prior knowledge about the kernel in applications, a nonparametric approach is desirable. In particular, the studies \cite{LMT21_JMLR,LMT22,LZTM19pnas} consider nonparametric inference of radial interaction kernels for first-order stochastic differential equations in the form  
 \begin{align}\label{eq:model_SDEs}
     \mathrm{d} X(t)=R_\phi[X(t)] \mathrm{d}t+\sigma \mathrm{d}B(t)\,, 
     \end{align}
where $X(t)=(X_1(t),\cdots,X_N(t))$ represents the position of particles, $R_{\phi}$ is same as in \eqref{eq:opt_model_R} and $B(t)$ is a standard Brownian motion in $\R^{Nd}$ with $\sigma\geq 0$ representing the strength of the random noise.  The least squares estimator is demonstrated to exhibit a convergence rate of $\big(\frac{M}{\log M}\big)^{-\frac{2\beta}{2\beta+1}}$ for suitable large $\beta$, where $M$ is the number of independent trajectories and $\beta$ represents the regularity index of the true kernel. However, the optimal minimax rate, namely the best convergence rate in the worst case, remains open. 

This study aims to answer the question of the optimal minimax rate. We consider the simplified but generic statistical model \eqref{eq:model}, which excludes numerical errors arising from the discretization of differential equations and eliminates dependence among components in the trajectory data.

\subsection{Main results}
This study establishes that the rate of $M^{-\frac{2\beta}{2\beta+1}}$ is the optimal minimax convergence rate under a coercivity condition that ensures the well-posedness of the inverse problem in the large sample limit. Informally, we establish the following minimax rate for $\beta\geq 1/4$:  
\begin{align*}
    \inf_{\widehat \phi_M} \sup_{\phi_*\in \calH^\beta} \E\Big[\|\widehat \phi_M-\phi_*\|_{L^2_\rho}^2\Big] \approx M^{-\frac{2\beta}{2\beta+1}}\,, \quad \text{ as }M\to \infty\,, 
\end{align*}
where the infimum is among all estimators $\widehat \phi_M$ inferred from data $\{(X^m,Y^m)\}_{m=1}^M$, and  $L^2_\rho$ is the space of square-integrable functions under the weight $\rho$, the probability measure of pairwise distances. Here, the hypothesis space $\calH^\beta$ can be a fractional Sobolev class $H^{\beta}(L)$ or H\"older class $\calC^{\beta}(L)$ (in Definitions \ref{def:wSobolev}--\ref{def:H\"older}). 
In particular, the minimax rate holds for $\beta\in [1/4, 1/2]$, so the hypothesis space can contain discontinuous functions, which are widely used in interacting particle systems (IPS) \cite{LMT21_JMLR,LMT22,LZTM19pnas,MT14}. This extends the range for $\beta$ from the cornerstone technique of covering argument in \cite{CuckerSmale02,Gyorfi06a}, which often requires $\beta> 1/2$ as a sufficient condition for estimating the entropy numbers or covering number of Sobolev classes. This range is also larger than the ranges $\beta\geq 2$ in \cite[Theorem 3.2]{LMT22} for stochastic IPS and $\beta\geq 1$ in \cite[Theorem 6]{LMT21_JMLR} for deterministic IPS, both based on the covering argument.  

A major innovation of our study is a new approach to proving the upper minimax rate. We introduce a \emph{tamed least square estimator} (tLSE) in Definition \ref{eq:tame-lse} and show that it achieves the optimal rate with a straightforward proof bypassing the traditional covering argument techniques employed in \cite{CuckerSmale02,LMT21_JMLR,LMT22,LZTM19pnas}. Our proof is based on sharp variance estimates leveraging Sobolev embedding and non-asymptotic estimates of the left tail probability of the smallest eigenvalue of the normal matrix. 
Notably, the tLSE-based method applies to the hypothesis spaces $\mathcal{H}^\beta$ in two settings: one with uniformly bounded basis functions, for which we establish minimax rates for all $\beta \geq 1/4$ in Theorem \ref{cor:L2rho_upper}, and another with non-uniformly bounded basis functions exhibiting mild growth of order $n^\delta$ with $\delta\geq 0$, for which we achieve minimax rates for all $\beta \geq \delta + 1/2$ in Theorem \ref{cor:L2rho_upper_case2}. We anticipate that the tLSE will serve as a valuable theoretical tool for establishing upper minimax rates in general nonparametric regression contexts.

To affirm that the upper minimax rate is optimal, we prove in Theorem \ref{thm:L2rho_lower} that the rate is also the lower minimax rate. We accomplish this by applying the Fano-Tsybakov method in \cite{tsybakov2008introduction}, which we generalize to include the weighted measure $\rho$. This involves careful construction of hypothesis functions for hypothesis testing in Section \ref{sec:lowerBd} and investigation into the relationship between fractional Sobolev classes and H\"older classes in Appendix \ref{sec:append_Sob}.

\subsection{Main difficulties and technical innovations}
The optimal minimax rate is well-established for classical nonparametric estimation (see, e.g., \cite{CuckerSmale02,Gyorfi06a,tsybakov2008introduction} and the references therein). In the classical setting, one estimates the function $\phi:\R\to \R$ in the model $Y=\phi(Z)+\eta$ from sample data $\{(Z^m,Y^m)\}_{m=1}^M$, where the data $Y$ depends \emph{locally} on a single value of $\phi$.  A critical fact in this setting is that the conditional expectation $\phi(z)= \E[Y | Z=z]$ is the unique minimizer of the mean-square loss function in the large sample limit, which is a well-posed inverse problem. Notable estimators achieving the minimax rate include the projection estimator for deterministic $Z$ (e.g., \cite{tsybakov2008introduction}) and the least squares estimator for random $Z$ using tools from the empirical process theory, which are based on covering arguments with the chaining technique (e.g., \cite{van2000asymptotic} and \cite[Chapters 11, 19]{Gyorfi06a}). 

However, \emph{nonlocal dependence} presents a new challenge in estimating the interaction kernel. The nonlocal dependence means that the operator $R_\phi[X]$ depends on the kernel $\phi$ non-locally through the weighted sum of multiple values of $\phi$, similar to a convolution. Thus, this intersection of statistical learning and deconvolution-type inverse problems raises significant hurdles in both well-posedness and the construction of estimators to achieve the minimax rate. 
To address these challenges, we show first that the inverse problem in the large sample limit is well-posed for a large class of distributions of $X$ satisfying Assumption \ref{assump:noise-X}. A key condition for well-posedness is the \emph{coercivity condition} studied in \cite{LiLu20,LLMTZ21,LMT22,LZTM19pnas}, and we examine it in Lemma \ref{lemma:coercive}. Due to this condition, a universal convergence rate for all distributions may not be feasible. Importantly, the coercivity condition also ensures that the nonlocal dependence does not affect the minimax rate, as discussed after Lemma \ref{lemma:projEst}.

Our major technical innovation lies in developing the tLSE with a straightforward proof for the upper minimax convergence rate. The tLSE is zero when the smallest eigenvalue of the normal matrix is below a threshold related to the coercivity constant, and it is the least squares estimator otherwise. That is, 
\begin{equation}\label{eq:tlse_intro}
\begin{aligned}
\widehat\phi_{n,M}^{\textup{tlse}} = \sum\nolimits_{k=1}^n \theta_k^{\textup{tlse}} \psi_k,\   \text{ where }
	(\theta_1^{\textup{tlse}},\ldots, \theta_n^{\textup{tlse}})^\top&= [\widebar{\bfA}_{n}^M]^{-1} \widebar{\bfb}_{n}^M \mathbf{1}_{\{ \lambda_{\min}(\widebar{\bfA}_{n}^M)> \frac{1}{4} c_{\bar\calL}\} } \,,  
\end{aligned}
\end{equation}
where $c_{\bar\calL}$ is the coercivity constant in Definition \ref{def:coercivity}. Here $\widebar{\bfA}_{n}^M$ and $\widebar{\bfb}_{n}^M$ are the normal matrix and normal vector for the regression over the hypothesis space $\calH_n= \mathrm{span}\{\psi_k\}_{k=1}^n$ with orthonormal basis functions $\{\psi_k\}_{k= 1}^{\infty}$ of $L^2_{\rho}$. Note that only in the set $\{ \lambda_{\min}(\widebar{\bfA}_{n}^M)\leq \frac{1}{4} c_{\bar\calL}\}$, the tLSE differs from the least squares estimator 
$
(\theta_1^{lse},\ldots, \theta_n^{lse})^\top= [\widebar{\bfA}_{n}^M]^{\dagger} \widebar{\bfb}_{n}^M
$, 
where $[\widebar{\bfA}_{n}^M]^{\dagger}$ denotes the Moore-Penrose inverse of $\widebar{\bfA}_{n}^M$. 
A crucial observation in our proof is that the optimal minimax rate is achieved if the probability of the set $\{ \lambda_{\min}(\widebar{\bfA}_{n}^M)\leq \frac{1}{4} c_{\bar\calL}\}$ does not affect the bias-variance tradeoff. This leads to the study of the left tail probability of $\lambda_{\min}(\widebar{\bfA}_{n}^M)$ with the dimension $n\approx M^{\frac{1}{2\beta+1}}$ chosen from the tradeoff, aiming for a non-asymptotic bound exponentially decaying in the sample size $M$.

We prove the following left tail probability using the matrix Chernoff inequality, 
\begin{align*}
	\P\left\{ \lambda_{\min}(\bar{\bfA}_{n}^{M}) \leq (1-\varepsilon) c_{\bar\calL} \right\}\leq 2n \exp\rbracket{-\frac{c \varepsilon M}{n}}\,, \quad \forall \varepsilon \in (0,1)\,,
\end{align*} 
where the constant $c>0$ is independent of $n,M$ and $\varepsilon$. 
 Importantly, this left tail probability bound is negligible in the bias-variance tradeoff as long as $n$ increases slower than $M$, a condition satisfied when $n\approx M^{\frac{1}{2\beta+1}}$ for all $\beta> 0$. The bound enables our tLSE to achieve the optimal upper minimax rate for all $\beta\geq 1/4$, which is required for the Sobolev embedding condition (Assumption {\rm\ref{assum:SoboEmb}}) to hold, as detailed in Lemma \ref{lemma:SobEmbed}.  We note that 
similar left tail probability bounds can be obtained using other matrix concentration inequalities, such as the widely-used matrix Bernstein inequality or the PAC-Bayes inequality introduced in \cite{Oliveira2016,Mourtada2022}. However, the Bernstein approach requires $\beta> 1/2$ to be negligible in the bias-variance tradeoff, and the PAC-Bayesian left tail probability demands additional constraints on the distribution of $X$ in order to achieve the upper minimax rate for all $\beta\geq 1/4$. More details can be found in Lemma \ref{thm:min-eigen_B} and Lemma \ref{thm:min-eigen}.

\subsection{Summary of the tLSE method and main insights}\label{sec:tsle_method} 
The tLSE method provides a new and efficient method for proving the upper minimax rate in nonparametric regression, applicable to models with either local or nonlocal dependency. As long as the coercivity condition holds, the proof is largely the same for both types of models. The process involves decomposing the $L^2_\rho$ error of the estimator $\widehat \phi_{n,M}^{\textup{tlse}}$ in \eqref{eq:tlse_intro} into bias and variance terms with proper upper bounds, then seeking a bias-variance tradeoff with optimal dimension $n=n_M\approx  M^{\frac{1}{2\beta +1}}$ to get the optimal rate. 
 \begin{enumerate}[leftmargin=5mm]
     \item[$\bullet$] \emph{Bias Term:}  The bias term is of order $n^{-2\beta}$ for functions in the fractional Sobolev class $H_{\rho}^{\beta}(L)$. 
     \item[$\bullet$] \emph{Variance Term:} For the nonlocal model, the variance term is further decomposed into a tamed variance term, a left tail probability bound, and a nonlocal bias term. In contrast, the local model does not have the nonlocal bias term. 
     \item[$\bullet$] \emph{Tamed Variance Term:} The tamed variance term is the variance of the well-conditioned parts of the tLSE, and it is of order $\frac{n}{M}$. 
    \item[$\bullet$] \emph{Left Tail Probability Bound:} The left tail probability controls the cutoff error of the tLSE, and it decays much faster than the bias term or the tamed variance term in the tradeoff.   
    \item[$\bullet$] \emph{Nonlocal Bias Term:} The nonlocal bias term arises due to the nonlocal interaction, which leads to a non-identity normal operator. It is of order $n^{-2\beta}$ under the coercivity condition.
\end{enumerate}
We summarize the dominating orders of these terms in Figure \ref{fig:diagram-tLSE}.

\usetikzlibrary{positioning}
\usetikzlibrary{arrows,shapes}
\usetikzlibrary{decorations.markings}

\tikzset{unode/.style = {
    circle, 
    draw=cyan!30!black, 
    thick,
    fill=cyan!80!black,
    inner sep=2.3pt,
    minimum size=2.3pt }}
    
\tikzset{roundnode/.style={
	circle, 
	draw=black!80, 
	fill=gray!10, 
	very thick, 
	minimum size=7mm}}

\tikzset{squarednode/.style={
	rectangle, 
	draw=black!80, 
	fill=gray!10, 
	very thick, 
	minimum size=5mm}}

\tikzset{squarednode2/.style={
	rectangle, 
	draw=orange!80, 
	fill=yellow!10, 
	very thick, 
	minimum size=5mm}}
	
\tikzset{squarednode3/.style={
	rectangle, 
	draw=black!80, 
	fill=cyan!20, 
	very thick, 
	minimum size=5mm}}
	
\tikzset{ellipsednode/.style={
	ellipse, 
	draw=black!80, 
	fill=cyan!20, 
	very thick, 
	minimum size=5mm}}
\tikzset{ellipsednode2/.style={
	ellipse, 
	draw=black!80, 
	fill=orange!20, 
	very thick, 
	minimum size=5mm}}

\begin{figure}[hbt]
\centering
\begin{tikzpicture}
\begin{scope}[xshift=2cm,yshift=2.5cm]
    \node[squarednode2, align=center] (NLE) at (0,0){\text{Total Error}\\ \textbf{Nonlocal  Model}};
    \node[align=center] (EP1) at (2,0){$=$};
    \node[squarednode2, align=center] (VN) at (3.8,0){\text{Bias:} $O(n^{-2\beta})$};
    \node[align=center] (PN1) at (5.7,0){$+$};
    \node[squarednode2, align=center] (BN) at (7,0){\text{Variance}};

    \node[squarednode, align=center] (LE) at (0,2.5){\text{Total error}\\ \textbf{Local Model}};
    \node[align=center] (EP1) at (2,2.5){$=$};
    \node[squarednode, align=center] (VL) at (3.8,2.5){\text{Bias:} $O(n^{-2\beta})$};
    \node[align=center] (PL1) at (5.7,2.5){$+$};
    \node[squarednode, align=center] (BL) at (7,2.5){\text{Variance}};

	\node[squarednode3, align=center] (LPV) at (10.91,2.5){\text{Tamed Variance}: \\ $O\left(n/M\right)$};
    \node[squarednode3, align=center] (LPL) at (11.3,1.25){\text{Left Tail Probability:}\\ $O\big(n\exp\left(-cM/n\right)\big)$};
    \node[squarednode3, align=center] (LPN) at (10.73,0){\text{Nonlocal Bias}: \\ $O(n^{-2\beta})$};

	\draw[very thick, orange, -stealth, shorten >=2pt] (BN.east) to node[auto] {} (LPN.west);
	\draw[very thick, orange, -stealth, shorten >=2pt] (BN.east) to node[auto] {} (LPL.west);
	\draw[very thick, orange, -stealth, shorten >=2pt] (BN.east) to node[auto] {} (LPV.west);
	
	\draw[very thick, gray, -stealth, shorten >=2pt] (BL.east) to node[auto] {} (LPL.west);
	\draw[very thick, gray, -stealth, shorten >=2pt] (BL.east) to node[auto] {} (LPV.west);
\end{scope}
\end{tikzpicture}
\caption{The bias-variance tradeoff in the tLSE approach for local and nonlocal models. The left tail probability and the nonlocal bias do not affect the bias-variance tradeoff.}
\label{fig:diagram-tLSE}	
\end{figure}

We note that many other methods achieve optimal minimax rates using more delicate tools and assumptions. The LSE using pseudo-inverse $[\widebar{\bfA}_{n}^M]^{\dag} \widebar{\bfb}_{n}^M$ has to deal with the negative moments of the normal matrices or, equivalently, the small ball probability of the smallest eigenvalue. Regularized LSEs in the form [$\widebar{\bfA}_{n}^M+\lambda \mathbf{B_n}]^{-1} \widebar{\bfb}_{n}^M$ with $\mathbf{B_n}$ being a positive definite matrix (see e.g., \cite[Chapter III.2]{CuckerSmale02} and \cite[Chapter 13.4]{Wainwright2019}) must be defined with a delicate regularization along with the bias-variance tradeoff. Another commonly used approach based on empirical process theory {\rm\cite{CuckerSmale02,LMT21_JMLR,LMT22}} bounds the variance term uniformly on the function spaces using the defect function and the covering techniques. On the contrary, the tLSE straightforwardly achieves the optimal rate without using the covering technique and thus establishes the optimal upper minimax rate even when $\beta\leq 1/2$.

This study has two main contributions:
\begin{enumerate}
\item \emph{Optimal minimax convergence rate for interaction kernel learning.} We establish the optimal minimax convergence rate $M^{-\frac{2\beta}{2\beta+1}}$ for learning the interaction kernel in Model \eqref{eq:model} with a wide range of distributions. It coincides with the minimax rate in the classical regression setting without local dependence. 
Moreover, this optimal rate applies to Sobolev classes with $\frac14\leq \beta\leq \frac12$. These classes encompass widely used discontinuous functions such as piecewise constant functions in opinion dynamics {\rm(see, e.g., \cite{MT14})}.

\item \emph{Introduction of the Tamed Least Square Estimator (tLSE).} The tLSE approach represents a new and efficient method for proving the minimax rate in nonparametric regression. A key insight is that the optimal minimax rate depends on whether the bias-variance tradeoff can remain unaffected by the left tail probability of the smallest eigenvalue of the normal matrix. This approach establishes a connection between the minimax rate, the left tail probability of the smallest eigenvalue of the random normal matrix, and fractional Sobolev embedding.
\end{enumerate}

The insights gained from this study pave the way for future research on the minimax rate for nonparametric regression regarding models with nonlocal dependence. The inverse problem in the large sample limit plays a fundamental role. We have focused on scenarios where the inverse problem is well-posed, finding that the optimal minimax rate $M^{-\frac{2\beta}{2\beta+1}}$ is consistent for regression with both local and nonlocal dependencies. In contrast, when the inverse problem is ill-posed, i.e., with a zero coercivity constant, the optimal rate, if it exists, is expected to be slower than $M^{-\frac{2\beta}{2\beta+1}}$ since the current rate bears a constant depending on the reciprocal of the coercivity constant. A subsequent study \cite{ZhangWangLu2025} further investigates the minimax rate in the absence of the coercivity condition, instead imposing polynomial or exponential decay conditions on the spectrum of the normal operator.

Additionally, our study has focused on the convergence in the sample size $M$ while keeping the number of particles $N$ finite. An intriguing direction for future research lies in examining the convergence rate as $N$ increases. Given that the inverse problem in the limit of $N=\infty$ becomes an ill-posed deconvolution, the convergence rate will likely depend on the spectral properties of the normal operator.

\subsection{Related work}
\paragraph*{Minimax rate for nonparametric regression} 
The study of the minimax rate in nonparametric regression is a well-established and extensively explored topic within inference and learning. Due to the vastness of the literature, we direct readers to \cite{Gyorfi06a,tsybakov2008introduction,CuckerSmale02,nickl2019_NonparametricStatistical}, among others, for comprehensive reviews. For lower minimax rates, this study utilizes the Fano-Tsybakov hypothesis testing method \cite{tsybakov2008introduction}. The Assouad method \cite{assouad1983deux,CaiYuan2012,yu97Assouad} and van Trees method \cite{gill1995applications} are viable alternatives.

For the upper minimax rate, notable estimators achieving the optimal rate without the logarithmic term include the projection estimator for deterministic input data (see, e.g., \cite{tsybakov2008introduction}), and the least squares estimator whose rate is proved by using the empirical process theory with covering arguments and chaining technique (see, e.g., \cite{van2000asymptotic} and \cite[Chapter 19]{Gyorfi06a}). Additionally, we note that the empirical process theory with a covering argument is widely used, and it applies to the interaction kernel estimation in \cite{LZTM19pnas,LMT22,LMT21_JMLR} and many others.  However, it leads to a sub-optimal rate with a logarithmic factor when using a fixed cover. The chaining technique may remove the logarithmic factor by constructing a sequence of covers, but the nonlocal dependence will further complicate the proof. The tamed least square estimator (tLSE) in the present paper stands out for its simplicity and broad applicability to nonparametric regression problems with either local or nonlocal dependencies.

\paragraph*{Inference for systems of interacting particles} A large amount of literature has been devoted to the inference for systems of interacting particles, and we can only sample a few here. Parametric inference has been studied in  \cite{amorino2023parameter,chen2021_MaximumLikelihooda,della2023lan,kasonga1990_MaximumLikelihooda,liu2020_ParameterEstimation,sharrock2021_ParameterEstimation} for the drift term and in \cite{huang2019learning} for the diffusion term. Nonparametric inference on estimating the drift $R_\phi$, but not the kernel $\phi$, has been studied in \cite{della2022nonparametric,yao2022mean}. The semi-parametric inference in \cite{belomestny2023semiparametric} estimates the interaction kernel. All these studies consider the case when $N\to \infty$ from a single long trajectory of the system. Inference of the mean-field equations has also been studied in \cite{della2022nonparametric,LangLu22,mavridis2022learning,messenger2022learning}. The closest to this study are \cite{LMT21_JMLR,LMT22,LZTM19pnas}, where the rate for learning the interaction kernels from multiple trajectories is $\big(\frac{M}{\log M}\big)^{-\frac{2\beta}{2\beta+1}}$, is suboptimal due to the use of supremum norm in the covering number argument. 
Building on these results, our study achieves the optimal rate in a simplified static model \eqref{eq:model}, advancing the understanding of the inference problem.

\paragraph*{Nonparametric deconvolution} 
 Nonlocal dependence is a key feature in nonparametric deconvolution, particularly in estimating probability densities as studied in \cite{fan1991optimal,Meister2009deconvolution}, among others. In such contexts, the underlying inverse problem in the large sample limit typically manifests as an ill-posed deconvolution challenge. The established optimal rate for these scenarios is $M^{-\frac{2\beta}{2\beta+2\alpha+1}}$, where $\alpha$ is the decay rate of the Fourier transform of the convolution kernel. In contrast, our study navigates a well-posed inverse problem made possible through the coercivity condition, differentiating it from the typical deconvolution framework.

\paragraph*{Linear regression for parametric inference and random matrices}
The normal matrix $\widebar{\bfA}_{n}^M$ in our study resembles the sample covariance matrix $\frac{1}{M}\sum_{m=1}^M x^{m}(x^{m})^\top$ in linear regression $y\approx \btheta^\top x $ from samples $\{(x^m,y^m)\}$ of a distribution on $\R^n \times \R$. Therefore, the analysis of this normal matrix can draw parallels from the study of sample covariance matrices with independent columns or entries, as explored in \cite{KM2015,LivshytsTikhomirovVershynin2021,MeiWangYao2023,MendelsonPaouris2014,Mourtada2022,Tikhomirov2018,Vershynin2018,Wainwright2019,Yaskov2015}.  
With notation $\bPhi^m =(R_{\psi_1}[X^{m}],\ldots,R_{\psi_n}[X^{m}])\in \R^{Nd\times n}=$ for each sample $X^m$, our least squares estimator can be analogized to a linear regression estimator for $Y \approx \btheta^\top \bPhi$ with a normal matrix $\bar{\bfA}_{n}^{M}=\frac{1}{MN}\sum_{m=1}^M [\bPhi^m]^\top \bPhi^m =\frac{1}{MN}\sum_{m=1}^M \sum_{i=1}^N [\bPhi_{i,\cdot}^m]^\top \bPhi_{i,\cdot}^m$. Because of the dependence between $\{\bPhi_{i,\cdot}^m\}_{i=1}^N$, $\widebar{\bfA}_{n}^M$ can not be viewed as an example of a sample covariance matrix with independent columns or entries. It is worth noting that there is ongoing interest in the study of sample covariance matrix with dependence, see, e.g., \cite{BVZ2021,Manrique2022,Oliveira2016,ShuNan2019,Vershynin2020}. 
Compared to these studies, the random matrices in nonparametric regression are normal matrices depending on the basis functions.

\section*{Notation} Throughout the paper, we use $C$ to denote universal constants independent of the sample size $M$ and the dimension $n$. The notations $C_\beta$ or $C_{\beta,L}$ denote constants depending on the subscripts. We use $\E_{\phi_*}$ to denote the expectation with respect to the joint distribution of $(X,Y)$ in Model \eqref{eq:model} where $Y$ depends on both $X$, $\eta$ and the true interaction kernel $\phi_*$. We omit the dependence on $\phi_*$, i.e., $\E=\E_{\phi_*}$, if the random variable only relies on $(X,\eta)$. We introduce the $L^2_{\rho}$ inner product as $\innerp{f,g}_{L^2_{\rho}}=\int f(r)g(r)\rho (dr)$ and denote $L^2_\rho$ norm by $\|f\|_{L_\rho^2}^2 = \int |f(r)|^2\rho (dr)$. For any operator $\calL: L_\rho^2\to L_\rho^2$, the operator norm is given by $\norm{\calL}_{\text{op}}=\sup_{\norm{f}_{L_\rho^2}=1}\|\calL f\|_{L^2_{\rho}}$. In addition, we define the $L^p_\rho$ norm by $\|f\|_{L_\rho^p}^p = \int |f(r)|^p\rho (dr)$ for all $p\geq 1$. 
Table \ref{tab:TableOfNotation2} summarizes the main notations. 

\smallskip

\begin{table}[htbp]
\centering
\vspace{-4mm}
\caption{Some notations}
\begin{tabular}{c l}
\toprule
\textbf{Notations} & \textbf{Description}   \\
\midrule
$M$, $N$ & Sample size and number of particles \\
$\rho$ (and $\rho'$)   & Exploration measure (and its density function) of pairwise distances \\
$\bar\calL$,\, $c_{\bar\calL}$   & The normal operator in \eqref{eq:operatorL} and its coercivity constant in Lemma {\rm\ref{lemma:coercive}} \\
$H_{\rho}^{\beta}(L)$ ,\, $\calC^{\beta}(L)$ & Sobolev and H\"older classes in Definitions \ref{def:wSobolev}-\ref{def:H\"older}\\
$\calH_n= \mathrm{span}\{\psi_k\}_{k=1}^n$ & Hypothesis space spanned by $n$ basis functions $\{\psi_k\}_{k=1}^n$ \\
$\bfA_{n}^M$, $\bfb_{n}^M$  & Normal matrix and normal vector in \eqref{eq:Ab} \\
\bottomrule
\end{tabular}
\label{tab:TableOfNotation2}
\vspace{-4mm}
\end{table}

\smallskip

The rest of the paper is organized as follows. We study the inverse problem in the large sample limit in Section \ref{sec:setting_IP}. In the process, we introduce assumptions and function spaces. Section \ref{sec:upperBd} introduces the tLSE and proves that the tLSE achieves the optimal rate, establishing an upper minimax rate. Section \ref{sec:lowerBd} proves the lower minimax rate via the hypothesis testing scheme. We present the technical results and proofs in the Appendix.

\section{Settings and inverse problem in large sample limit}\label{sec:setting_IP}
 
The well-posedness of the inverse problem in the large-sample limit lies at the heart of statistical inference. This section builds the foundation by imposing constraints on the distributions of $X$ and the noise in Section \ref{sec:settings}, setting a weighted function space in Section \ref{sec:rho}, and showing that the inverse problem is well-posed (see Section \ref{sec:normalOperator}). Finally, Section \ref{sec:SobolevClass} introduces the Sobolev and H\"older classes. Further details regarding the fractional Sobolev and H\"older classes (spaces) are provided in Appendix \ref{sec:append_Sob}.

\subsection{Assumptions on distributions}\label{sec:settings} 
Recall that the data $\{(X^m,Y^m)\}_{m=1}^M$ are i.i.d.~samples of $(X,Y)$ satisfying the model in \eqref{eq:model}. The joint distribution depends on the distributions of $X$, the noise $\eta$, and the interaction kernel $\phi$. We make some subsequent assumptions on the distributions of $X$ and $\eta$. Recall that a random vector $X=(X_1,\ldots,X_N)$ has an \emph{exchangeable} distribution if the joint distributions of $(X_i)_{i\in \mathcal{I}}$ and $(X_i)_{i\in \mathcal{I_\pi}}$ are identical, where $\mathcal{I}\subset \{1,\ldots,N\}$ and $\mathcal{I_\pi}$ is a permutation of $\mathcal{I}$.  
\begin{assumption}[Distribution of $X$]\label{assump:noise-X}
We assume the entries of the $(\R^d)^{\otimes N}$-valued random variable $X=(X_1,\ldots,X_N)$ satisfy the following conditions:
\begin{enumerate}[label=$(\mathrm{A\arabic*})$]
\item\label{Assump:exchangeable} The components of the random vector $X = (X_1, \ldots, X_N)$ are exchangeable and have finite mean. 
\item\label{Assump:CondIndpt} For each pair $\{X_i-X_j, X_i-X_{j'}\}$ with $j\neq j'$ and $j,j' \neq i$, there exists a $\sigma$-algebra $\calX_i$ such that the pair are conditionally independent. 
\item \label{Assump:cts}  For each pair $\{X_i-X_j, X_i-X_{j'}\}$ with $j\neq j'$ and $j,j' \neq i$, the joint probability density function is continuous. 
 \end{enumerate}
\end{assumption}

Assumptions \ref{assump:noise-X} \ref{Assump:exchangeable}--\ref{Assump:cts} are mild conditions to simplify the inverse problem of estimating the kernel $\phi$, and weaker constraints may replace them with more careful arguments as in \cite{LLMTZ21,LMT22}. The exchangeability in \ref{assump:noise-X} \ref{Assump:exchangeable} simplifies the exploration measure in Lemma \ref{lemma:rho}. The conditional independence in \ref{assump:noise-X} \ref{Assump:CondIndpt}, together with the exchangeability, enables the coercivity condition for the inverse problem to be well-posed, as detailed in Lemma \ref{lemma:coercive}. The continuity in Assumption \ref{assump:noise-X} \ref{Assump:cts} ensures that the exploration measure has a continuous density, which is used in proving the lower minimax rate in Section \ref{sec:lowerBd}.

A sufficient condition for Assumptions \ref{assump:noise-X} \ref{Assump:exchangeable}--\ref{Assump:CondIndpt} is that 
$(X_1,\ldots,X_N)$ are conditionally i.i.d.~in the sense that there exists a $\sigma$-algebra $\calX$ such that $\{X_i\}_{i=1}^N$ are i.i.d.  given $\calX$. The exchangeability follows from the fact that $\P\{ \bigcap_{i=1}^N X_i\in A_i \}=\E [\prod_{i=1}^N \E[\b1_{\{X_i\in A_i\}}\mid \calX]]=\E [\prod_{i=1}^N\E[\b1_{\{X_{\pi(i)}\in A_i\}}\mid \calX]]=\P\{\bigcap_{i=1}^N  X_{\pi(i)}\in A_i \}$ for any permutation $\pi$. Also, the random variables $X_i-X_j$ and $X_i-X_{j'}$ are conditionally independent given the $\sigma$-algebra $\calX_i=\sigma(X_i,\calX)$. We note that exchangeability has a long history in probability, statistics, and interacting particle systems. For example,   \cite{de1929funzione,DiaconisFreedman1980,Hoff2009,Kallenberg2005,LindleyNovick1981,LMT22} and references therein. Random variables in an exchangeable infinite sequence are conditionally i.i.d. by the well-known De Finetti theorem (e.g., \cite[Theorem 1.1]{Kallenberg2005}).

Examples of $X=(X_1,\ldots,X_N)$ satisfying \ref{assump:noise-X} \ref{Assump:exchangeable}--\ref{Assump:cts} are prevalent in applications. A convenient example is $X$ with i.i.d.~components. In particular, when $X$ has i.i.d.~components being uniformly distributed on $[0,1]$, we can compute the joint distribution explicitly further to analyze the inverse problem in the large sample limit as explained in Remark \ref{rmk:compact} and Appendix \ref{sec:unif-X}.

Another example is the Euler-Maruyama approximation of the interacting particle system \eqref{eq:model_SDEs} at the initial time. Specifically, consider the one-step transition 
$$X(t_1)\approx X(t_0)+R_\phi[X(t_0)] (t_1- t_0)+ \sigma (  W(t_1)-W(t_0) )$$ 
where $X(t_0)$ has an exchangeable distribution.  
The above equation gives our regression model \eqref{eq:model} with $Y=\frac{X(t_1)- X(t_0)}{t_1- t_0} $,  $X= X(t_0)$, and $\eta=\sigma \frac{ W(t_1)-W(t_0)}{t_1- t_0}$.   
Assumption \ref{assump:noise-X} \ref{Assump:exchangeable} is fulfilled since $X= X(t_0)$ forms an exchangeable random vector. Additionally, given $\calX_i=\mathcal{F}\{X(t_0),W_i(t_1)\}$, the pairs $X_j(t_1)-X_i(t_1)$ and $X_{j'}(t_1)-X_i(t_1)$ are independent, satisfying Assumption \ref{assump:noise-X} \ref{Assump:CondIndpt}. Clearly, Assumption \ref{assump:noise-X} \ref{Assump:cts} also holds within this context.

\begin{assumption}[Distribution of noise]\label{assump:noise}
The noise $\eta$ is independent of the random array $X$. Moreover, we assume the following conditions:
\begin{enumerate}[label=$(\mathrm{B\arabic*})$]
    \item\label{Assump:Fourth_Noise} The noise vector $\eta = (\eta_1, \ldots, \eta_N)$ consists of i.i.d. centered entries with finite variance $\sigma_\eta^2$ and bounded fourth moment.
    \item\label{Assump:Density_Noise} The density $p_\eta$ of $\eta$ satisfy that 
$\exists \ c_\eta > 0$: 
\[ \int_{\R^{Nd}} p_{\eta}(u) \log \frac{p_{\eta}(u)} {p_{\eta}(u + v)}du \leq c_\eta \Vert v\Vert ^2,\quad \forall\ v\in\R^{Nd}. 
\] 
\end{enumerate}
\end{assumption}

Assumption \ref{assump:noise} \ref{Assump:Fourth_Noise} on the noise is mild. The density Assumption \ref{assump:noise} \ref{Assump:Density_Noise} on the noise, as used in {\rm\cite[page 91]{tsybakov2008introduction}}, is more restrictive. For example, when $\eta\sim \mathcal{N} 
(0,\sigma_\eta^2I_d)^{\otimes N}$, both assumptions hold and the inequality in Assumption \ref{assump:noise} \ref{Assump:Density_Noise} holds with $c_\eta=Nd/(2\sigma_\eta^2)$. We also note that the noise need not be Gaussian; an example is provided in \cite{ZhangWangLu2025}. 
Importantly, the lower minimax bound in Theorem \ref{thm:L2rho_lower} requires Assumption \ref{assump:noise} \ref{Assump:Density_Noise}, whereas our upper minimax bound in Theorem \ref{cor:L2rho_upper} requires only Assumption \ref{assump:noise} \ref{Assump:Fourth_Noise}.

\subsection{Exploration measure}\label{sec:rho}
The first step in the regression is to set a function space of learning. We set the default function space of learning to be $L^2_\rho$ by defining a measure $\rho$ that quantifies the exploration of the interaction kernel by the data. 
The exploration measure serves as the analog of the probability measure of the independent variable in classical statistical learning.   
\begin{definition}[Exploration measure] \label{def:rho}
The exploration measure $\rho$ associated with independent variable of the interaction kernel in \eqref{eq:opt_model_R} is the large sample limit of the empirical measure $\rho_M$ of data $\{X^m\}_{m=1}^M$:  
\begin{equation}\label{eq:rho_all}
\begin{aligned}
\rho(A) & =\lim_{M\to \infty}\rho_{M}(A) = \frac{1}{N(N-1)} \sum_{i=1}^N\sum\nolimits_{j\neq i} \P{(\|X_i-X_j\|_{\R^d}\in A)},
\end{aligned}
\end{equation} 
where  $\rho_{M}(A)  = \sum\nolimits_{m=1}^M \sum\nolimits_{i=1}^N\sum\nolimits_{j\neq i}\frac{\mathbf{1}_{A}(\|X^m_i-X^m_j\|_{\R^d})}{MN(N-1)}$ for any Lebesgue measurable set  $A\subset \R^+$ and $\sum\nolimits_{j\neq i}$ means $\sum\nolimits_{j=1, j\neq i}^{N}$. 
\end{definition}

\begin{lemma}[Exploration measure under exchangeability]\label{lemma:rho}
 Under Assumption {\rm\ref{assump:noise-X}}, the measure $\rho$ is the distribution of $\|X_1-X_2\|_{\R^d}$ and has a continuous density.
\end{lemma}
\begin{proof}
The exchangeability in Assumption \ref{assump:noise-X} \ref{Assump:exchangeable} implies that the distributions of $X_i-X_j$ and $X_1-X_2$ are the same for any $i\neq j$. Hence, by definition, the exploration measure is the distribution of the random variable $\|X_1-X_2\|_{\R^d}$:
\begin{equation*}
	\rho(A) = \P{(\|X_1-X_2\|_{\R^d}\in A)}
\end{equation*}
which has a continuous density by Assumption \ref{assump:noise-X} \ref{Assump:cts}.  
\end{proof}

The exploration measure $\rho$ can be supported on either $[0,\infty)$ or bounded intervals. For example, the exploration measure has a density $\rho'(r)=\frac{r^{d-1}}{C_{\lambda}}e^{-\frac{r^2}{4\lambda}}\b1_{[0,\infty)}(r)$ with $C_\lambda = \frac{(4\lambda)^{\frac{d}{2}}}{2}\Gamma(d/2)$ when each triplet $(X_i,X_j,X_k)$ is a Gaussian vector in $\R^{3d}$ with an exchangeable joint distribution and $\textrm{cov}(X_i,X_j)=\lambda I_d$ with some $\lambda>0$; see \cite[Lemma 3.2]{LLMTZ21} for the details. The exploration measure has a bounded support when the $X_i$'s are independent and uniformly distributed on $[0,1]$; see Appendix \ref{sec:unif-X}.

\subsection{Inverse problem in the large sample limit}
\label{sec:normalOperator}
We show that the inference via minimizing the loss function in the large sample limit is a deterministic inverse problem. Importantly, the inverse problem is well-posed under Assumption \ref{assump:noise-X}.  

\begin{definition}[Normal Operator] \label{def:inv-op}
The \emph{normal operator} $\bar \calL: L^2_{\rho}\to L^2_{\rho}$ associated with the Model \eqref{eq:model} is 
 \begin{equation}\label{eq:operatorL}
	\begin{aligned}
		\langle \bar \calL \phi,\psi \rangle_{L^2_{\rho}} 
  &= \frac{1}{N}\sum\nolimits_{i=1}^N\E [\langle R_\phi[X]_i, R_\psi[X]_i\rangle_{\R^d} ],  \quad  \forall \phi,\psi\in L^2_\rho. 
	\end{aligned}
	\end{equation}
\end{definition}

\begin{definition}[Coercivity condition] \label{def:coercivity}
A self-adjoint linear operator $\bar\calL:L^2_\rho\to L^2_\rho$ is \emph{coercive} on $L^2_\rho$ with a constant $c_{\bar\calL}>0$ if 
\[ 
\innerp{\bar\calL\phi,\phi}_{L^2_\rho} \geq c_{\bar\calL} \|\phi\|_{L^2_\rho}^2, \quad \forall \phi\in L^2_\rho.
\] 
In other words, $\E[\|R_{\phi}[X]\|_{\R^{Nd}}^2]\geq N c_{\bar\calL} \|\phi\|_{L^2_\rho}^2$ for all $\phi\in L^2_\rho$.
\end{definition}

\begin{proposition}[Inverse problem in the large sample limit] \label{prop:invP}
Under Assumption {\rm\ref{assump:noise-X}} \ref{Assump:exchangeable}, the large sample limit of the empirical mean square loss function $\calE_{M}(\phi)$ in \eqref{eq:lossFn} is 
\[
\calE_\infty(\phi) = \E \Big[ \frac{1}{N} \|Y-R_{\phi}[X]\|_{\R^{Nd}}^2 \Big] = \innerp{\bar\calL\phi,\phi}_{L^2_\rho} - 2\innerp{\bar\calL\phi_{*},\phi}_{L^2_\rho} +\innerp{\bar\calL\phi_*,\phi_*}_{L^2_\rho} +\sigma_\eta^2d. 
\]
Moreover,  the expected loss function $\calE_\infty(\phi)$ is uniformly convex in $L^2_\rho$ if and only if $\bar\calL$ is coercive, and the ground truth function $\phi_{*}$ is the unique minimizer when $\bar\calL$ is coercive. 
\end{proposition}

\begin{proof} Recall that $Y=R_{\phi_{*}}[X] + \eta$ and $\eta$ is centered. By the definition of $\bar\calL$, we have 
\begin{align*}
	\E \Big[ \frac{1}{N} \|Y-R_{\phi}[X]\|_{\R^{Nd}}^2 \Big] 
	&=\frac{1}{N}\E \left[ \|R_{\phi}[X]-R_{\phi_*}[X]\|_{\R^{Nd}}^2  + \|\eta\|_{\R^{Nd}}^2\right] \\
	 & = \innerp{\bar\calL\phi,\phi}_{L^2_\rho} - 2\innerp{\bar\calL\phi_{*},\phi}_{L^2_\rho} + \innerp{\bar\calL\phi_*,\phi_*}_{L^2_\rho} + \sigma_\eta^2d. 
\end{align*}
The Hessian of $\calE_\infty$ is $\nabla^2 \calE_\infty= 2\bar\calL$, where $\nabla^2$ denotes the second-order Fr\'echet derivative in $L^2_\rho$. Hence, the loss function is uniformly convex if and only if $\bar\calL$ is coercive. 
Additionally, the minimizer of $\calE_\infty$ is a solution to $0= \nabla \calE_\infty(\phi) = 2\bar\calL  \phi - 2 \bar\calL  \phi_{*}$. Thus, if $\bar\calL$ is coercive, the unique minimizer is $\phi_{*}$. 
\end{proof}

Proposition \ref{prop:invP} implies that the inverse problem of minimizing the loss function is well-posed if and only if $\bar\calL$ is coercive. 
The next lemma shows that $\bar\calL$ is coercive with a constant $c_{\bar\calL} \geq \frac{N-1}{N^2}$. 
For simplicity, we denote 
\begin{equation}\label{Def:r+br}
	r_{ij}= \|X_i-X_j\|_{\R^d},\quad  \br_{ij}= \frac{X_i-X_j}{r_{ij}}=\frac{X_i-X_j}{\|X_i-X_j\|_{\R^d}}\,. 
\end{equation}
We define a operator $\calL_G: L^2_{\rho} \to L^2_{\rho}$ to be
\begin{equation}\label{eq:LG}
	\innerp{\calL_G\phi, \psi}_{L^2_\rho} = \E[\phi(r_{12})\psi(r_{13})\innerp{\br_{12},\br_{13}}_{\R^d}]. 
\end{equation}

\begin{lemma}
\label{lemma:coercive}
	Under Assumptions \ref{assump:noise-X} \ref{Assump:exchangeable}--\ref{Assump:CondIndpt}, the operator $\bar \calL$ in Definition {\rm \ref{def:inv-op}} is self-adjoint and has a decomposition
	\begin{equation} \label{eq:Lbar}
		\bar \calL = \frac{(N-1)(N-2)}{N^2}  \calL_G + \frac{N-1}{N^2} I, 
	\end{equation}
where the operator $\calL_G$ in \eqref{eq:LG} is non-negative. Also, $\bar \calL$ is coercive with a coercivity constant
	\begin{equation}\label{eq:lambda-min}
		c_{\bar\calL}\geq \frac{N-1}{N^2}\,. 
	\end{equation}
Moreover, under Assumptions \ref{assump:noise-X} \ref{Assump:exchangeable}--\ref{Assump:CondIndpt}, we have $\norm{\Lbar}_{\textup{op}}:=\sup_{\|\phi\|_{L^2_\rho}=1} \innerp{\Lbar \phi,\phi}\leq 1$.
\end{lemma}

\begin{proof}
By Assumption \ref{assump:noise-X} \ref{Assump:exchangeable}, the components of $X$ are exchangeable, $\br_{ij}$ and $\br_{ij'}$ are independent conditional on a $\sigma$-algebra $\calX_i$ for any $i=1,\cdots,N$. By exchangeability, $\E[\phi(r_{ij}) \phi(r_{ij'})\innerp{\br_{ij},\br_{ij'}}_{\R^d}] = \E[\phi(r_{12})\phi(r_{13})\innerp{\br_{12},\br_{13}}_{\R^d}]$ for all $j\neq j'$ and $\E[\|\phi(r_{ij})\bfr_{ij}\|_{\R^d}^2]=\E[\|\phi(r_{12})\bfr_{12}\|^2_{\R^d}]=\E[|\phi(r_{12})|^2]$ for all $i\neq j$. 
By the conditionally independence Assumption \ref{assump:noise-X} \ref{Assump:CondIndpt}, 
we have  
\begin{align*}
	\E[\phi(r_{12})\phi(r_{13})\innerp{\br_{12},\br_{13}}_{\R^d}\vert \calX_1] = \innerp{\E[\phi(r_{12})\br_{12} \vert \calX_1], \E[\phi(r_{13})\br_{13} \vert \calX_1]}_{\R^d} 
\end{align*}
and $\E[\phi(r_{12})\br_{12} \vert \calX_1]= \E[\phi(r_{13})\br_{13} \vert \calX_1]$ by exchangeability. Hence, 
 \begin{align*}
    \E[\phi(r_{12})\phi(r_{13})\innerp{\br_{12},\br_{13}}_{\R^d}]  &= \E[ \E[\phi(r_{12})\phi(r_{13})\innerp{\br_{12},\br_{13}}_{\R^d}\vert \calX_1]  ]\\
  &= \E[\innerp{\E[\phi(r_{12})\br_{12} \vert \calX_1], \E[\phi(r_{13})\br_{13} \vert \calX_1]}_{\R^d} ]  \\
  &=  \E[\big \| \E[\phi(r_{12})\br_{12} \vert \calX_1]\big \|^2_{\R^d} ] \geq 0.  
 \end{align*}
This means $\innerp{\calL_G\phi, \phi} \geq 0 $ for any $\phi\in L^2_\rho$. In other words, $\calL_G$ is non-negative.

Moreover, the decomposition \eqref{eq:Lbar} follows from  
\begin{align*}
\innerp{\bar\calL\phi,\phi}_{L^2_\rho} = & \frac{1}{N}\E[\innerp{R_\phi[X],R_\phi[X]}_{\R^{Nd}}]
 =  \frac{1}{N^3} \sum_{i=1}^N\sum_{j\neq i} \sum_{j'\neq i} \E[\phi(r_{ij}) \phi(r_{ij'})\innerp{\br_{ij},\br_{ij'}}_{\R^d} ] \\ 
= & \frac{N-1}{N^2} \E[\phi(r_{12})^2] + \frac{(N-1)(N-2)}{N^2} \E[\phi(r_{12})\phi(r_{13})\innerp{\br_{12},\br_{13}}_{\R^d}] \\ 
= & \bigg\langle\left(\frac{N-1}{N^2} I + \frac{(N-1)(N-2)}{N^2}  \calL_G \right) \phi, \phi\bigg\rangle_{L^2_\rho}\,.
\end{align*}
Thus, $\bar\calL$ is coercive with the constant in \eqref{eq:lambda-min} because $\calL_G$ is positive. 
Additionally, the second equality above also implies $\norm{\Lbar}_{\textup{op}}\leq 1$.   	
\end{proof}

\begin{remark}[Sharp coercivity constant]\label{rmk:compact}
 The lower bound for the coercivity constant in \eqref{eq:lambda-min} is sharp. It is achieved when the operator $\calL_G$ is compact, which is true under relatively weak constraints on $X$ by noticing that it is an integral operator (see e.g.,{\rm\cite{LLMTZ21,LMT22,LZTM19pnas}}). For example, $\calL_G$ is a compact integral operator when $X$ is uniformly distributed on $[0,1]^3$ illustrated in Appendix \ref{sec:unif-X}. 
\end{remark}

\begin{remark}[Differences from classical nonparametric estimation.]\label{rmk:loc-nonlocal}
A key distinction between the nonparametric estimation of the function $\phi$ in the classical model $Y=\phi(X)+\eta$, and our model, $Y=R_\phi[X]+\eta$, lies in the nature of the normal operator in the large sample limit. For the classical model, the normal operator is the identity operator (which follows by replacing the model in Definition {\rm\ref{def:inv-op}}), and the inverse problem in the large sample limit is always well-posed. Conversely, in our model, the normal operator may lack coercivity. This difference stems from the nonlocal dependence of $R_\phi$ on $\phi$, where $R_\phi[X]$ depends on a convolution of multiple values of $\phi$. Although Assumption \ref{assump:noise-X} \ref{Assump:CondIndpt} guarantees the coercivity of the normal operator, this nonlocal dependence introduces an additional bias term in the analysis of the least squares estimator, and we control this bias by relying on the coercivity condition, see \eqref{eq:btilde_bd} in Lemma {\rm \ref{lemma:projEst}}. 
\end{remark}

\subsection{Fractional Sobolev class and H\"older class}\label{sec:SobolevClass}

Function classes play a pivotal role in nonparametric regression by quantifying important properties of the underlying functions, such as smoothness, boundedness, and integrability. 
In this section, we introduce weighted Sobolev classes where the ground truth $\phi_*$ belongs, connect them with classical H\"older classes, and introduce two key assumptions on the function classes that connect the Sobolev embedding.  

Without loss of generality, we assume that $[0, 1] \subseteq \operatorname{supp}(\rho')$ and focus our kernel estimation on $[0, 1]$. In practice, this interval can be replaced by a compact subset of $\operatorname{supp}(\rho')$, for instance, the set $\{\rho' \ge \tau\}$ for some threshold $\tau > 0$ as in \cite{LMT21_JMLR, LMT22, LZTM19pnas}. Consequently, the function space for learning becomes $L^2_{\rho} = L^2_{\rho}([0, 1])$, with Sobolev and H\"{o}lder classes being subsets of this space. For brevity, we will omit the explicit domain of the functions whenever the context is clear.

Our primary focus is on the scenario where $\rho'$ is bounded below by a positive constant on $[0, 1]$, which yields uniformly bounded basis functions such as $\{\psi_k(x) =\sqrt{2}\sin(k\pi x)/ \sqrt{\rho'(x)}\}_{k=1}^{\infty}$ on $[0,1]$. Thus, we make the following assumption on the basis functions. 
\begin{assumption}[Uniformly bounded basis functions]\label{assum:bd-eigenfn} 
The orthonormal basis functions $\{\psi_k\}_{k= 1}^{\infty}$ of $L_{\rho}^2([0,1])$ are uniformly bounded with $C_{\max}:=\sup_{k\geq 1}\|\psi_k\|_{\infty}<\infty$.    
\end{assumption}

We also incorporate the case of \textbf{non-uniformly bounded} basis functions, which happens if $\rho'$ does not have a positive lower bound on $[0,1]$. We will replace the uniform boundedness by a mild growth condition: $\sup_{1\leq k\leq n} \| \psi_k\|_{\infty} \leq C_{\delta} n^{\delta}$ 
with $\delta\geq 0$ and $C_{\delta}>0$ depends only on $\delta$; see Assumption \ref{assum:L4-eigenfn}.

We first define two fractional (weighted) Sobolev classes $H^\beta_\rho(L)$ and $H^\beta_{0,\rho}(L)$ and two H\"older classes $\mathcal{C}^\beta(L)$ and $\mathcal{C}_0^\beta(L)$. The Sobolev classes are generalizations of the conventional unweighted Sobolev class (see, e.g., \cite[Definition 1.12]{tsybakov2008introduction}), which will be used for controlling the bias in the bias-variance tradeoff in the proof of the upper minimax rate. The H\"older classes will be used to prove the lower minimax rate. The fractional (weighted) Sobolev classes are defined in terms of the orthonormal basis functions $\{\psi_k\}_{k= 1}^{\infty}$. We impose the following condition on this basis.

\begin{definition}[Fractional Sobolev class]\label{def:wSobolev} 
Let $\{\psi_k\}_{k=1}^\infty$ be a complete orthonormal basis.  
For $\beta>0$ and $L>0$, define the \emph{fractional Sobolev class} $H_{\rho}^{\beta}(L)\subseteq L^2_\rho$ as
\begin{align}
	H_{\rho}^{\beta}(L) := \left\{\phi = \sum\nolimits_{k=1}^\infty \theta_k\psi_k \in L^2_\rho([0,1]): 
	{ \|\phi\|_{H_{\rho}^{\beta}}^2:=} \sum\nolimits_{k=1}^\infty k^{2\beta}\theta_k^2 \leq L^2/\pi^{2\beta} \right\}. 
\end{align}
In particular, we use the notation $H_{0,\rho}^{\beta}(L)$ as a specific instance of $H_{\rho}^{\beta}(L)$, where the basis functions are chosen as $\psi_k(x) = \sqrt{2}\sin(k\pi x)/\sqrt{\rho'(x)}$ for $k \geq 1$.
\end{definition}

\begin{definition}[H\"older classes]\label{def:H\"older}
For $\beta,L>0$, the H\"older class $\calC^{\beta}(L)$ on $[0,1]$ and the H\"older class with zero boundary conditions $\calC_0^{\beta}(L)$  are 
 \begin{equation}\label{eq:class-H\"older}
 \begin{aligned}
 	\calC^{\beta}(L)&=\Big\{f:[0,1] \to \R : | f^{(l)}(x)-f^{(l)}(y)|\leq L| x -y |^{\beta-l}, \forall x,y \in [0,1] \Big\},\\
 	\calC_0^{\beta}(L)&=\{f\in \calC^{\beta}(L): f^{(j)}(0)=f^{(j)}(1)=0, j=0,\cdots,l-1 \}\,.
 \end{aligned}
\end{equation}
where $f^{(j)}$ denotes the $j$-th order derivative of functions $f$ and $l =\lfloor\beta\rfloor$. 
\end{definition}

The Sobolev classes, similar to the H\"older classes, quantify the ``smoothness'' of a function in terms of the function's coefficient decay, as the following lemma shows.   
\begin{lemma}\label{lemma:wSobolev}
	Let $\phi= \sum_{k=1}^\infty \theta_k\psi_k\in H_{\rho}^{\beta}(L) \text{ or } H_{0,\rho}^{\beta}(L) $. Then  $\sum_{k=n+1}^\infty |\theta_k|^2 \leq L^2 n^{-2\beta} {/ \pi^{2\beta}}$ for all $n$. In particular, $\|\btheta\|_{\ell^2}^2\leq L^2 {/\pi^{2\beta}}$ and $\sup_k |\theta_k|^2 \leq L^2 {/ \pi^{2\beta}}$. 
\end{lemma} 
\begin{proof} It follows directly from the definition of the Sobolev class that 
\begin{align*}
    \sum\nolimits_{k=n+1}^\infty |\theta_k|^2 \leq n^{-2\beta} \sum\nolimits_{k=n+1}^\infty k^{2\beta}|\theta_k|^2 \leq L^2 n^{-2\beta} {/ \pi^{2\beta}}\,.
\end{align*} 
The last two statements also follow the definition directly. 
\end{proof}

\begin{remark}[H\"older class being a subset of the Sobolev class.]
\label{rmk:Soblev_Holder} 
 When $\beta$ is an integer, we have $\calC_0^{\beta}(L)\subseteq H_{0,\rho}^{\beta}(L)$.  In fact, note that the H\"older class is a subset of the conventional weighted Sobolev class $W_{0,\rho}^{\beta}(L)$ defined as 
 \begin{equation}\label{Eq:Sobolev}
\begin{aligned}
	W_{\rho}^{\beta}(L)&:=\Big\{f\in L^2_\rho([0,1]): f^{(\beta-1)} \text{ is a.c.}, \int_0^1 |f^{(\beta)}(x)|^2\rho(dx) \leq L^2\Big\}; \\
	W_{0,\rho}^{\beta}(L)&:=\Big\{ f\in W_{\rho}^{\beta}(L): f^{(j)}(0)=f^{(j)}(1)=0 \text{ for } j=0,\cdots,\beta-1 \Big\}\,, 
\end{aligned}
\end{equation}
where a.c.~means absolutely continuous; see e.g., {\rm \cite[Defintion 1.11]{tsybakov2008introduction}}. 
Since $\beta$ is an integer and $\rho'$ is continuous with positive lower and upper bounds, the (weighted) fractional Sobolev class $H_{0,\rho}^{\beta}(L)$ is equivalent to $W_{0,\rho}^{\beta}(L)$ by the same proof for {\rm\cite[Proposition 1.14]{tsybakov2008introduction}} when the basis functions are $\psi_k(x)=\sqrt{2}\sin\rbracket{k\pi x}/\sqrt{\rho'(x)}$, $x\in[0,1]$. Combining these two facts, we obtain that $\calC_0^{\beta}(L)\subset H_{0,\rho}^{\beta}(L)$. 
When $\beta$ is not an integer, the case is more complicated, and we discuss this case in Section \ref{sec:lowerBd} and Appendix \ref{sec:append_Sob}.
\end{remark}

One of the innovations in this study is to apply the following Sobolev embedding condition to obtain the optimal fourth-moment bounds for the normal vectors in Lemma \ref{lemma:mmtBd_v2} when $\beta\leq \frac 12$. These moment bounds are crucial for establishing variance estimates in Lemma \ref{Lem:Ab_conv}.

\begin{assumption}[Sobolev embedding condition]\label{assum:SoboEmb} 
Assume there exists a positive finite constant $\calK_{\beta}$ only depends on $\beta$ such that 
\begin{equation}\label{Assu:suff_L4}
\sup_{\phi\in H_{\rho}^{\beta}(L)}  \norm{\phi}_{L^{4}_{\rho}}^4 \leq \calK_{\beta} L^4\,.  
\end{equation}
\end{assumption}
In other words, Assumption \ref{assum:SoboEmb} means that the $L^4_{\rho}([0,1])$-norm is controlled by the $H_{\rho}^{\beta}(L)$ bound, similar to the \emph{fractional Sobolev embedding theorem} (\cite[Theorem 6.7]{Hitchhiker2012}).
We show next that it holds for $H_{\rho}^\beta(L)$ with $\beta> 1/2$ when the basis functions are uniformly bounded. In particular, condition \eqref{Assu:suff_L4}  holds for $H_{0,\rho}^\beta(L)$ with $\beta\geq 1/4$ by the fractional Sobolev embedding theorem.

\begin{lemma}
\label{lemma:SobEmbed}
The Sobolev embedding condition \eqref{Assu:suff_L4} holds for $H_{\rho}^\beta(L)$ with $\beta> 1/2$ under Assumption {\rm \ref{assum:bd-eigenfn}}. Furthermore, it holds for $H_{0,\rho}^\beta(L)$ with $\beta\geq 1/4$. 	
\end{lemma}
\begin{proof}
When $\beta>1/2$, Assumption {\rm \ref{assum:bd-eigenfn}} immediately implies the uniform boundedness of any function $\phi=\sum_{k=1}^\infty \theta_k\psi_k\in  H_{\rho}^{\beta}(L)$: 
\[
\|\phi\|_{\infty}\leq \sum\nolimits_{k}|\theta_k|\leq \Big(\sum\nolimits_{k}k^{-2\beta}\Big)^{1/2} \Big(\sum\nolimits_{k}k^{2\beta}|\theta_k|^2\Big)^{1/2} \leq C_{\beta}L {/ \pi^{\beta}}\,,
\]
where $C_{\beta}:=(\sum\nolimits_{k}k^{-2\beta})^{1/2}\leq \frac{2\beta}{2\beta-1}$. 
Consequently, Eq. \eqref{Assu:suff_L4} holds with $\calK_\beta = C_\beta^4$. 

Next, we show that the Sobolev embedding condition \eqref{Assu:suff_L4} holds for fractional Sobolev class $H_{0,\rho}^\beta(L)$ with $\beta\geq 1/4$. First, when $\beta>1/2$, note that the basis functions $\psi_k(x) =\sqrt{2}\sin(k\pi x)/ \sqrt{\rho'(x)}$ of $L^2_{\rho}=L^2_{\rho}([0,1])$ are uniformly bounded, i.e., they satisfy Assumption {\rm \ref{assum:bd-eigenfn}}. Hence, the Sobolev embedding condition holds.  
When $\beta\leq \frac 1 2$, we resort to the fractional Sobolev embedding theorems. Since $\rho'$ has positive lower and upper bounds, we have $H_{0,\rho}^{\beta} = W_{0,\rho}^{\beta}([0,1]) = W_{0}^{\beta}([0,1])$, where $W_{0,\rho}^{\beta}([0,1])$ and $W_{0}^{\beta}([0,1])$ are the classical weighted and unweighted fractional Sobolev spaces defined using singular integrals, respectively; see Appendix \ref{sec:append_Sob}. Then,  when $\beta<\frac{1}{2}$, we have, by the fractional Sobolev embedding theorem in \cite[Theorem 6.7]{Hitchhiker2012},  
\begin{align}\label{Embed:Sob}
	\norm{\phi}_{L^q}\leq C_{\beta,q} \norm{\phi}_{W^{\beta}}\,, \text{ for any }q\in[1,\frac{2}{1-2\beta}]
\end{align}
for a constant $C_{\beta,q}>0$ uniformly for all $\phi\in H_{0,\rho}^{\beta}= W^{\beta}_0([0,1])$. 
When $\beta=\frac 12$, by  \cite[Theorem 6.10]{Hitchhiker2012}, we have \eqref{Embed:Sob} holds for any $q\in[1,\infty)$. Thus, applying these embedding inequalities with $q=4$ in \eqref{Embed:Sob} to bound $\norm{\phi}_{L^{4}}$ as in \eqref{Assu:suff_L4} by $\norm{\phi}_{W_{\rho}^{\beta}}$ and then employing the connection between Gagliardo fractional Sobolev class and spectral Sobolev class established in Lemma \ref{Lem:spectral<Gagliardo}, we obtain \eqref{Assu:suff_L4} with $\tilde{C}_{\beta}=2^{2\beta} \int_0^{\infty} \frac{1-\cos(h)}{h^{1+2\beta}}dh$ and $\calK_\beta  = [2\tilde{C}_{\beta}C_{\beta,4}{/ \pi^{\beta}}] ^4$  provided that $\frac{2}{1-2\beta}\geq 4$, equivalently,  $\beta \geq \frac 14$.
\end{proof}

\begin{remark}[The case $\beta<1/4$]\label{rmk:SobEmbed}
To ensure the Sobolev embedding condition, additional constraints on the exploration measure and the basis functions are necessary. The Sobolev class $H_{0,\rho}^{\beta}(L)$ provides such an example, and it suggests that $\beta\geq 1/4$ is necessary for the Sobolev embedding condition with the $L^4_\rho$-norm. Thus, to study the case $\beta<1/4$, one has to replace the $L^4_\rho$ norm with an $L^q_\rho$-norm with $q<4$, and correspondingly, replace the assumption on the fourth moment with a higher-order moment.  
\end{remark}


\smallskip

\section{Upper minimax rate}\label{sec:upperBd}

In this section, we establish an upper minimax rate of $M^{-\frac{2\beta}{2\beta + 1}}$ by introducing the \emph{tamed least squares estimator} (tLSE), as detailed in Theorem \ref{thm:L2rho_upper}. The tLSE not only achieves this rate efficiently but also provides a relatively simple proof. Its efficacy extends beyond the scope of this study, marking the tLSE a valuable tool for proving upper minimax rates for general nonparametric regression.

\subsection{A tamed least squares estimator}\label{sec:tLSE}
Given the data $\{(X^m,Y^m)\}_{m=1}^M$, we construct an estimator based on the loss function of the empirical mean square error in \eqref{eq:lossFn}, optimized over a hypothesis space $\calH_n=\mathrm{span}\{\psi_k\}_{k=1}^n$. 
Since $R_\phi$ is linear in $\phi$, the loss function is quadratic in $\phi$, and one can solve for the minimizer by least squares. 
We introduce the forthcoming tamed least squares estimator. 
\begin{definition}[Tamed least squares estimator (tLSE)]\label{def:tlse}
The {tamed least squares estimator} in $\calH_n= \mathrm{span}\{\psi_k\}_{k=1}^n$ is $	\widehat \phi_{n,M} = \sum_{k=1}^{n} \widehat \theta_k \psi_k$  with $\widehat \btheta_{n,M} = (\widehat \theta_1,\ldots,\widehat \theta_n)^\top$ solving by  
\begin{equation}\label{eq:tame-lse}
\begin{aligned}
	\widehat\btheta_{n,M}&= [\widebar{\bfA}_{n}^M]^{-1} \widebar{\bfb}_{n}^M \mathbf{1}_{\{ \lambda_{\min}(\widebar{\bfA}_{n}^M)> \frac{1}{4} c_{\bar\calL}\} } = \begin{cases}
		 0\,,&  \text{if }\lambda_{\min}(\widebar{\bfA}_{n}^M)\leq \frac{1}{4} c_{\bar\calL}\,; \\
		 [\widebar{\bfA}_{n}^M]^{-1} \widebar{\bfb}_{n}^M\,,& \text{if }\lambda_{\min}(\widebar{\bfA}_{n}^M)> \frac{1}{4} c_{\bar\calL}\,.
	\end{cases}
\end{aligned}
\end{equation}
where $\widebar{\bfA}_{n}^{M}$ and $\widebar{\bfb}_{n}^{M}$ are the \emph{normal matrix} and \emph{normal vector}, respectively
	\begin{subnumcases}{\label{eq:Ab}}
		\widebar{\bfA}_{n}^{M}(k,l)  = \frac{1}{MN}\sum\nolimits_{m=1}^M\innerp{R_{\psi_k}[X^{m}], R_{\psi_l}[X^{m}]}_{\R^{Nd}}, \label{eq:Ab_A}\\ 
		\widebar{\bfb}_{n}^M(k) = \frac{1}{MN}\sum\nolimits_{m=1}^M\innerp{R_{\psi_k}[X^{m}], Y^{m}}_{\R^{Nd}},\label{eq:Ab_b}
	\end{subnumcases}  
	and the constant $c_{\bar\calL}$ is the coercivity constant in \eqref{eq:lambda-min}.  
\end{definition}
The threshold in \eqref{eq:tame-lse}, denoted as $\frac{1}{4}c_{\bar\calL}$, can be eased to $\frac{1-\epsilon}{2}c_{\bar\calL}$ for any $\epsilon\in (0,1)$, as demonstrated in Lemma \ref{thm:min-eigen}.

We emphasize that the tLSE is not the widely used least squares estimator (LSE): 
\begin{equation}\label{eq:lse}
	\widehat\btheta_{n,M}^{lse}= [\widebar{\bfA}_{n}^M]^{\dag} \widebar{\bfb}_{n}^M
\end{equation}
where $A^\dagger$ of a matrix $A$ denotes its Moore-Penrose inverse satisfying $A^\dag A = A A^\dag  = I_{\text{rank}(A)}$. The tLSE differs from the LSE in the random set $\{\lambda_{\min}(\widebar{\bfA}_{n}^M)\leq \frac{1}{4}c_{\bar\calL}\}$: in this set, the tLSE simply is zero while the LSE retrieves information from data by pseudo-inverse. The probability of this set decays exponentially as $M$ increases (see Section \ref{sec:prob_eig_bd}), making the tLSE and LSE the same with a high probability. However, this probability is non-negligible, as we show in Remark \ref{rmk:singularA} below that the normal matrix may be singular with a positive probability. 
\begin{remark}[Positive probability of a singular normal matrix]\label{rmk:singularA}
We construct an example showing that the normal matrix $\widebar{\bfA}_{n}^M$ can be singular with a positive probability.  Consider three particles independently and uniformly distributed over $[0,1]$, namely, $X_1, X_2, X_3 \overset{i.i.d.}{\sim}U([0,1])$. We have $r_{12}=|X_1-X_2|\sim \rho'(r)=2(1-r)\b1_{[0,1]}(r)$ and $R_{\phi}[X]=\frac 12 \phi(|X_1-X_2|)\frac{X_1-X_2}{|X_1-X_2|}+\frac 12 \phi(|X_1-X_3|)\frac{X_1-X_3}{|X_1-X_3|}$. Let $\phi(r)=2 \b1_{[1/2,1]}(r)$. Note that $\|\phi\|_{L^2_{\rho}}^2 = 
		\int_0^1 |\phi(r)|^2 \rho'(r)dr=\int_{1/2}^1 4\cdot 2(1-r)dr=1$\,.
Thus,  if $\phi$ is one of the basis functions in the definition of $\bar{\bfA}_{n}^{M}$, we have 
$
\lambda_{\min}(\bar{\bfA}_{n}^{M}) \leq \frac{1}{MN}\sum_{m=1}^M \|R_{\phi}[X^m]\|_{\R^{N}}^2  
$.  
As a result, 
	\begin{align*}
      &  \P \big\{	\lambda_{\min}(\bar{\bfA}_{n}^{M}) =0 \big\} \geq \P \Big\{\frac {1}{MN} \sum\nolimits_{m=1}^M \|R_{\phi}[X^m]\|_{\R^{N}}^2= 0 \Big\} \geq \left(\P \big\{\|R_{\phi}[X]\|_{\R^{N}}^2= 0 \big\}\right)^M \\
	\geq &  \Big( \P\Big\{ \bigcap_{i\neq j} \{|\phi(|X_i-X_j|)|=0\} \Big\} \Big)^M  \geq  \Big( \P\Big\{ \bigcap_{i=1}^3 \{ X_i\in[0,1/8]\} \Big\}  \Big)^M  \geq 8^{-3M}\,.
	\end{align*}
For any $N$, we have similarly :
	$\P\big\{ \lambda_{\min}(\bar{\bfA}_{n}^{M}) =0 \big\} \geq   \left( \P\Big\{ \bigcap_{i=1}^N \{ X_i\in[0,1/8]\} \Big\}  \right)^M \geq \frac{1}{8^{NM}}\,.$
\end{remark}

A major advantage of the tLSE over the regular LSE is its appealing effectiveness in proving the optimal upper minimax rate. The main challenge in proving the convergence rate for the LSE is to control the variance term $\E[\|\widehat\btheta_{n,M}^{lse}-\btheta^* \|^2_{\ell^2}]$ uniformly in $n$, where $\btheta^*$ denotes the true parameter. Since the LSE uses the pseudo-inverse, one has to study the negative moments of the normal matrices. However, as Remark \ref{rmk:singularA} shows, the normal matrix can be singular with a positive probability; hence the negative moments are unbounded. Thus, one has to either study additional conditions for the negative moments to be uniformly bounded for all $n$, or properly study regularized least squares \cite{Gyorfi06a,tsybakov2008introduction}. 

In contrast, the tLSE achieves the minimax rate with a notably simpler proof, requiring only the coercivity condition. The key component is that the left tail probability of $\{\lambda_{\min}(\widebar{\bfA}_{n}^M) \leq \frac {c_{\bar\calL}}{4}\}$ is negligible in the bias-variance tradeoff. 

We will primarily utilize the matrix Chernoff inequality to estimate the left tail probability. Additionally, the matrix Bernstein inequality and the PAC-Bayes inequality will be employed to derive analogous results, facilitating comparisons with the Chernoff left tail probability.

The forthcoming lemma shows that in the large sample limit, the normal matrix is invertible. Then, the tLSE is the same as the LSE, and it recovers the projection of the true function in the hypothesis space with a controlled error. 

\begin{lemma}\label{lemma:projEst} 
Set $\phi_*=\sum_{k=1}^\infty \theta_k^* \psi_k$ be the true kernel. Under Assumption {\rm\ref{assump:noise-X}},  Assumption {\rm\ref{assump:noise} \ref{Assump:Fourth_Noise}} and Assumption {\rm \ref{assum:bd-eigenfn}}, for all $1\leq k,l\leq n$, the limits $\widebar{\bfA}_{n}^\infty(k,l)=\lim_{M\to \infty} \widebar{\bfA}_{n}^{M}(k,l)$ and $\widebar{\bfb}_{n}^\infty(k)=\lim_{M\to \infty} \widebar{\bfb}_{n}^M(k)$ exist almost surely and satisfy
\begin{subnumcases}{\label{Eq:AbEE}}
	\widebar{\bfA}_{n}^{\infty}(k,l) = \frac 1N \E[\innerp{R_{\psi_k}[X],R_{\psi_l}[X]}_{\R^{Nd}}] =\innerp{\bar\calL\psi_k,\psi_l}_{L^2_\rho} \,,\quad 1\leq k,l\leq n\,;  \label{Eq:AbEE_A} \\
 \widebar{\bfb}_{n}^{\infty}(k)  = \frac 1N \E[\innerp{R_{\psi_k}[X],Y}_{\R^{Nd}}] =\innerp{\bar\calL\psi_k,\phi_*}_{L^2_\rho} ,\quad 1\leq k \leq n \,,\label{Eq:AbEE_b}
\end{subnumcases}
and the smallest eigenvalue of the deterministic norm matrix $\widebar{\bfA}_{n}^\infty$ satisfies $\lambda_{\min}(\widebar{\bfA}_{n}^\infty) \geq c_{\bar\calL}>0$.  Importantly, we have
\begin{equation}\label{Eq:Asyp_theta}
	\btheta_n^*=(\theta_1^*,\theta_2^*,\cdots,\theta_n^*)^\top=[\widebar{\bfA}_{n}^\infty]^{-1}\widebar{\bfb}_{n}^\infty - [\widebar{\bfA}_{n}^\infty]^{-1}\widetilde{\bfb}_{n}^\infty, 
\end{equation}
where $\widetilde{\bfb}_{n}^{\infty}(k)  :=\innerp{\bar\calL\psi_k,\phi_{*,n}^\perp}_{L^2_\rho}$ for $1\leq k \leq n$ with $\phi_{*,n}^\perp := \sum_{l= n+1}^\infty \theta_l^*\psi_l$, and 
\begin{equation}\label{eq:btilde_bd}
	 \| [\widebar{\bfA}_{n}^\infty]^{-1}\widetilde{\bfb}_{n}^\infty\|_{\R^n}^2 =
	 \|\btheta_n^* -[\widebar{\bfA}_{n}^\infty]^{-1}\widebar{\bfb}_{n}^\infty \|_{\R^n}^2
	  \leq c_{\bar\calL}^{-2}  \sum\nolimits_{l= n+1}^\infty |\theta_l^*|^2=:\epsilon_n^*\,. 
\end{equation}
\end{lemma}

\begin{proof} 
The existence of the limits $\widebar{\bfA}_{n}^\infty(k,l)=\lim\limits_{M\to \infty} \widebar{\bfA}_{n}^{M}(k,l)$ and $\widebar{\bfb}_{n}^\infty(k)=\lim\limits_{M\to \infty} \widebar{\bfb}_{n}^M(k)$ follows from the strong law of large numbers. 
The equations in \eqref{Eq:AbEE} follow directly from the definitions of the operator and $Y=R_{\phi_*}[X]+\eta$. 
To show the bound for the smallest eigenvalue of the expectation of the normal matrix, note that for any $\btheta = (\theta_1,\ldots,\theta_n)\in \R^n$, Eq. \eqref{Eq:AbEE_A} and properties of norm operator in Lemma \ref{lemma:coercive} imply that 
\[
\btheta^\top \widebar{\bfA}_{n}^\infty \btheta = \sum\nolimits_{k,l=1}^n \theta_k\theta_l\innerp{\bar\calL\psi_k,\psi_l}_{L^2_\rho}
= \innerp{\bar\calL\sum\nolimits_{k=1}^n\theta_k\psi_k,\sum\nolimits_{l=1}^n\theta_l\psi_l}_{L^2_\rho} \geq c_{\bar\calL} \|\sum\nolimits_{k=1}^n\theta_k\psi_k\|^2_{L^2_\rho} =  c_{\bar\calL}\|\btheta\|_{\R^n}^2\,.
\]
Also, Eq.\eqref{Eq:Asyp_theta} follows from the fact that for any $k=1,\cdots,n$
\begin{align*}
     \widebar{\bfb}_{n}^{\infty}(k) & = \innerp{\bar\calL\psi_k,(\sum\nolimits_{l= 1}^n + \sum\nolimits_{l= n+1}^\infty) \theta_l^*\psi_l}_{L^2_\rho} \\
     & = [\widebar{\bfA}_{n}^\infty \btheta_n^*](k) + \innerp{\bar\calL\psi_k, \phi_{*,n}^\perp}_{L^2_\rho}\\
     &= [ \widebar{\bfA}_{n}^\infty \btheta_n^*](k) +\widetilde{\bfb}_{n}^{\infty}(k).
\end{align*}
We proceed to prove Eq. \eqref{eq:btilde_bd}. Since $\bar\calL$ is self-adjoint, by Parseval's identity and definition of operator norm, we have that 
\[
\| \widetilde{\bfb}_{n}^{\infty}\|_{\R^n}^2 = \sum\nolimits_{k=1}^n | \innerp{\psi_k, \bar\calL\phi_{*,n}^\perp}_{L^2_\rho}|^2 \leq \sum\nolimits_{k=1}^\infty | \innerp{\psi_k, \bar\calL\phi_{*,n}^\perp}_{L^2_\rho}|^2=  \| \bar\calL\phi_{*,n}^\perp\|_{L^2_\rho}^2 \leq  \|\bar\calL \|^2_{\text{op}} \|\phi_{*,n}^\perp\|_{L^2_\rho}^2\,. 
\]
Hence, applying $ \| [\widebar{\bfA}_{n}^\infty]^{-1}\|_{\text{op}}^2 = \lambda_{\min}(\widebar{\bfA}_{n}^\infty)^{-2} \leq  c_{\bar\calL}^{-2} $, contraction inequality $\|\bar\calL \|^2_{\text{op}}\leq 1$ from  Lemma \ref{lemma:coercive}, and $\|\phi_{*,n}^\perp\|_{L^2_\rho}^2 = \sum_{l= n+1}^\infty |\theta_l^*|^2$, we obtain 
\begin{align*}
     \| [\widebar{\bfA}_{n}^\infty]^{-1}\widetilde{\bfb}_{n}^\infty\|_{\R^n}^2  
	 \leq \| [\widebar{\bfA}_{n}^\infty]^{-1}\|_{\text{op}}^2 \|\widetilde{\bfb}_{n}^\infty\|_{\R^n}^2 \leq c_{\bar\calL}^{-2} \|\bar\calL \|^2_{\text{op}} \|\phi_{*,n}^\perp\|_{L^2_\rho}^2 \leq c_{\bar\calL}^{-2} \sum\nolimits_{l= n+1}^\infty |\theta_l^*|^2\,. 
\end{align*}
Then, the inequality \eqref{eq:btilde_bd} follows by combining the above inequality with Eq.\eqref{Eq:Asyp_theta}. 
\end{proof}

\smallskip

The extra bias term $[\widebar{\bfA}_{n}^\infty]^{-1}\widetilde{\bfb}_{n}^\infty$, controlled by \eqref{eq:btilde_bd}, underscores a notable distinction between the classical local model and our nonlocal model in nonparametric regression. 
It is absent in the classical nonparametric estimation, where the normal matrix is the identity matrix and the normal vector is the projection of $\btheta^*$, since the normal operator is the identity operator. Therefore, this extra term is directly attributable to the nonlocal dependence, and we call it \emph{nonlocal bias}. It leads to an extra term (in the variance) in the bias-variance tradeoff, as we will show in Lemma \ref{Lem:Ab_conv}. Importantly, the coercivity constant plays a pivotal role in controlling the nonlocal bias term, as shown in \eqref{eq:btilde_bd}. Thus, as long as the coercivity condition holds, the nonlocal dependence does not affect the minimax rate resulting from the bias-variance tradeoff.

\begin{remark}\label{rmk:tlse-regu}
The tLSE also differs from commonly used regularized estimators in practice: the regularized LSE by truncated SVD or Tikhonov regularization (see, e.g., {\tcb{}\rm \cite{cucker2002_BestChoicesb,hansen1987_TruncatedSVD,LLA22}}), or the truncated LSE that uses a cutoff to make the estimator bounded. These three estimators retrieve information from data by tackling the challenge from an ill-conditioned or even singular normal matrix.  
In contrast, the tLSE is zero when the normal matrix has an eigenvalue smaller than the threshold, abandoning the estimation task without extracting information from data. The tLSE has a theoretical advantage over these practical estimators: it achieves the optimal minimax rate based on the coercivity condition, while the analysis of LSE typically involves bounding the negative moments of the small eigenvalues of the normal matrix. \end{remark}

\subsection{Upper minimax rate: uniformly bounded basis functions}\label{Sec:UpbdMini}
Our main result is the forthcoming theorem, which shows that the tamed LSE achieves the minimax convergence rate when the dimension of the hypothesis space is properly selected.

\begin{theorem}[Upper minimax rate]\label{cor:L2rho_upper}
Suppose Assumption {\rm\ref{assump:noise-X}} and Assumption {\rm\ref{assump:noise} \ref{Assump:Fourth_Noise}} on the model hold, and Assumption {\rm\ref{assum:bd-eigenfn}} on the basis functions hold. 
If  $\beta> \frac 12$, then
\begin{equation}\label{ineq:upper_main2}
\limsup_{M\to\infty} \inf_{\widehat{\phi}} \sup_{\phi_*\in H_{\rho}^{\beta}(L)}    \E_{\phi_*}\left[ M^{\frac{2\beta}{2\beta+1}} \| \widehat{\phi}-\phi_* \|_{L^2_\rho} ^2 \right] \leq C_{\textup{upper}}\,,
\end{equation}
where $C_{\textup{upper}}>0$ is a constant in Eq. \eqref{ineq:upper_main} and the infimum is among all estimators $\widehat \phi_M$ inferred from data $\{(X^m,Y^m)\}_{m=1}^M$. 
Furthermore, the upper bound \eqref{ineq:upper_main2} holds for all $\beta\geq \frac 14$ if the Sobolev embedding condition in Assumption {\rm\ref{assum:SoboEmb}} is also satisfied. 
\end{theorem}

The upper minimax rate follows immediately from Proposition \ref{thm:L2rho_upper}, which shows that the tamed LSE $\widehat{\phi}_{n_M,M}$ achieves the rate, since 
\begin{equation*}
	\begin{aligned}
		&\limsup_{M\to\infty} \inf_{\widehat{\phi}} \sup_{\phi_*\in H_{\rho}^{\beta}(L) }   \E_{\phi_*}[ M^{\frac{2\beta}{2\beta+1}} \| \widehat{\phi}-\phi_* \|_{L^2_\rho} ^2 ] \\
 \leq&  \limsup_{M\to\infty}  \sup_{\phi_*\in H_{\rho}^{\beta}(L)
 } 
   \E_{\phi_*}[ M^{\frac{2\beta}{2\beta+1}} \| \widehat{\phi}_{n_M,M}-\phi_* \|_{L^2_\rho} ^2 ]. 
	\end{aligned}
\end{equation*}
Thus, we focus on proving Proposition \ref{thm:L2rho_upper} in this section.

\begin{proposition}[Convergence rate for tLSE]\label{thm:L2rho_upper}
 Suppose Assumption {\rm\ref{assump:noise-X}} and Assumption {\rm\ref{assump:noise} \ref{Assump:Fourth_Noise}} on the model hold, and Assumption {\rm\ref{assum:bd-eigenfn}} on the basis functions hold. Then, the tLSE in \eqref{eq:tame-lse} with $n_M=\floor{( \frac{2\beta (Lc_{\bar\calL}^2{/ \pi^{\beta}}+ 2 )}{C_0 c_{\bar\calL}^4} M)^{\frac{1}{2\beta+1}}}$  converges at the rate $M^{-\frac{2\beta}{2\beta+1}}$ for any $\beta> \frac 12$, i.e., 
\begin{equation}\label{ineq:upper_main}
\limsup_{M\to\infty} \sup_{\phi_*\in H_{\rho}^{\beta}(L)}   \E_{\phi_*}\left[ M^{\frac{2\beta}{2\beta+1}} \| \widehat{\phi}_{n_M,M}-\phi_* \|_{L^2_\rho} ^2 \right] \leq C_{\textup{upper}} :=2C_{\beta,L}  (C_0 c_{\bar\calL}^{-2})^\frac{2\beta}{2\beta+1}\,. 
\end{equation}
Furthermore, the inequality \eqref{ineq:upper_main} holds for all $\beta\geq \frac 14$ if Assumption {\rm \ref{assum:SoboEmb}} is also satisfied, for example, when $H_{\rho}^\beta= H_{0,\rho}^\beta$. 
\end{proposition}

The constants in the theorem are as follows: $c_{\bar\calL}\geq \frac{N-1}{N^2}$ is the coercivity constant defined in \eqref{eq:lambda-min}, $C_{\beta,L}=\frac{2\beta+1}{2\beta} [2\beta (L{/ \pi^{\beta}}+ 2 c_{\bar\calL}^{-2} )]^\frac{1}{2\beta+1}$,  and $C_0$ can be found in Lemma \ref{Lem:Ab_conv}.  In particular, \eqref{ineq:upper_main} holds for $H_{0,\rho}^{\beta}(L)$ with $\beta\geq 1/4$ since Assumptions {\rm\ref{assump:noise-X}}, {\rm\ref{assump:noise} \ref{Assump:Fourth_Noise}}, {\rm\ref{assum:bd-eigenfn}} and {\rm \ref{assum:SoboEmb}} can be verified.

\begin{remark}[The case of $\beta\leq 1/2$ and Sobolev embedding in H\"older space]\label{Rmk:SobEmbH} 
    The case $\beta\leq 1/2$  holds practical significance, particularly because the fractional Sobolev class $H_{\rho}^{\beta}(L)$ can contain discontinuous interaction functions such as piecewise constants. On the other hand, when $\beta>1/2$, the functions in $H_{\rho}^{\beta}(L)$ are continuous when the density of $\rho$ is both lower-bounded away from zero and upper-bounded. This follows from the fact that $H_{\rho}^{\beta}(L)\simeq W^{\beta}_{\rho} \simeq W^{\beta}$ and by the Sobolev embedding that $W^{\beta}$ embeds continuously in $\calC^{\beta-\frac 12}$ (see e.g., {\rm \cite[Theorem 8.2]{Hitchhiker2012}}) as discussed in Appendix \ref{sec:append_Sob}. Therefore, to cover discontinuous functions, it is necessary to consider the case $\beta\leq 1/2$.  
\end{remark}

\begin{proof}[Proof of Proposition \ref{thm:L2rho_upper}]
The proof follows the standard technique of bias-variance tradeoff, except an extra term bounding the probability of the set where the tLSE is zero. The bound follows from the left tail probability $\P\left\{ \lambda_{\min}(\bar{\bfA}_{n}^{M}) \leq \frac 1 4 c_{\bar\calL} \right\} \leq G_{L,c_{\bar\calL}}(n,M) $, which we establish in Lemma \ref{Lem:Ab_conv}. The term $G_{L,c_{\bar\calL}}(n,M)$ enjoys an exponential decay since it comes from the concentration inequality of the smallest eigenvalue of the normal matrix. \\

Let $\phi_*=\sum_{k=1}^\infty \theta_k^* \psi_k$ and let $\btheta_{n}^*=(\theta_1^*,\ldots, \theta_n^*)$.  
We start from the bias-variance decomposition: 
\begin{align*}
	 \E_{\phi_*}[ \| \widehat{\phi}_{n,M}-\phi_* \|_{L^2_\rho}^2 ] 
	&= \underbrace{\E_{\phi_*}[\|\widehat \btheta_{n,M}- \btheta^*_{n}\|^2_{\R^n}]}_{\text{variance term}} + \underbrace{\sum\nolimits_{k={n}+1}^\infty |\theta^*_k|^2}_{\text{bias term}}\,.
\end{align*}
The variance term is controlled by, as we detailed in Lemma {\rm\ref{Lem:Ab_conv}} (see below), 
\begin{align*}
    \E_{\phi_*}[\|\widehat \btheta_{n,M}- \btheta^*_{n}\|^2_{\R^n}]  \leq& \underbrace{C_0  c_{\bar\calL}^{-2} \frac{n}{M}}_{\text{`tamed' variance term}}+ \underbrace{G_{L,c_{\bar\calL}}(n,M)}_{\text{concentration term}}  + \underbrace{2c_{\bar\calL}^{-2}  \sum\nolimits_{l= n+1}^\infty |\theta_l^*|^2}_{\text{nonlocal bias term}} \, ,
\end{align*}
where the `tamed' variance term of order $\frac{n}{M}$ comes from the well-conditioned parts of the tLSE, a concentration term $G_{L,c_{\bar\calL}}(n,M)$ comes from the left tail probability for tLSE to be zero, and the nonlocal bias term $2 c_{\bar\calL}^{-2}  \sum_{l= n+1}^\infty |\theta_l^*|^2 $ originates from the nonlocal dependence in Eq. \eqref{eq:btilde_bd}. Here, the universal positive constants are $C_0=\sqrt{798}C_{\max}^4(\frac{C_\eta {\pi^{\beta}}}{C_{\max}^4L N^2} + L{/ \pi^{\beta}})$, and $C_{\max} = \sup_{k\geq 1}\|\psi_k\|_\infty$.

The bias term is bounded above by the smoothness of the true kernel in $H_{\rho}^{\beta}(L)$. That is, by Lemma \ref{lemma:wSobolev} we have $\sum\nolimits_{k={n}+1}^\infty |\theta^*_k|^2\leq L n^{-2\beta}{/ \pi^{\beta}}$. 
Combining these three estimates, we have 
\begin{align}
  \E_{\phi_*}[ \| \widehat{\phi}_{n,M}-\phi_* \|_{L^2_\rho} ^2 ]
 &\leq (L+ 2 c_{\bar\calL}^{-2} ) n^{-2\beta}+  C_0  c_{\bar\calL}^{-2} \frac n M+ G_{L,c_{\bar\calL}}(n,M) \label{Eq:Trade-off}\\
	&=: g(n)+ G_{L,c_{\bar\calL}}(n,M) \,.\nonumber 
\end{align}

Minimizing the trade-off function $g(n) = \bar L n^{-2\beta}+ C_0  c_{\bar\calL}^{-2} n M^{-1}$ with $\bar L = L{/ \pi^{\beta}}+ 2 c_{\bar\calL}^{-2} $, we obtain the optimal dimension of hypothesis space  $n_M=\floor{( \frac{c_{\bar\calL}^2\tilde L}{C_0 } M)^{\frac{1}{2\beta+1}}}$, and $g(n_M) \leq 2{\bar  L}^{\frac{1}{2\beta+1}}  (C_0 c_{\bar\calL}^{-2})^\frac{2\beta}{2\beta+1}M^{-\frac{2\beta}{2\beta+1}} $. 

When $\beta> \frac 12$, we have $M\gg n_M^2$ as $M\to \infty$ and the following estimate for either choice of $G_{L,c_{\bar\calL}}(n,M)=\frac{L^2}{\pi^{2\beta}}\calE_{c_{\bar\calL}}^{{\rm Che}}(n,M)$ or $\frac{L^2}{\pi^{2\beta}}\calE_{c_{\bar\calL}}^{{\rm Ber}}(n,M)$ defined in \eqref{Def:G_Lc}: 
\begin{align}
	G_{L,c_{\bar\calL}}(n_M,M) \leq 2{\bar  L}^{\frac{1}{2\beta+1}}  (C_0 c_{\bar\calL}^{-2})^\frac{2\beta}{2\beta+1}M^{-\frac{2\beta}{2\beta+1}}. \label{Eq:Trade-off_G}
\end{align}

Moreover, if Assumption \ref{assum:SoboEmb} holds with $\beta\geq 1/4$, $G_{L,c_{\bar\calL}}(n,M)=\frac{L^2}{\pi^{2\beta}}\calE_{c_{\bar\calL}}^{{\rm Che}}(n,M)$, the inequalities \eqref{Eq:Trade-off} and  \eqref{Eq:Trade-off_G}  remain valid if $M$ is large enough.
Hence, in either case, with $C_{\beta,L,N, \bar\calL}=4{\bar  L}^{\frac{1}{2\beta+1}}  (C_0 c_{\bar\calL}^{-2})^\frac{2\beta}{2\beta+1}$, we have 
\begin{align*}
	\E_{\phi_*}[ \| \widehat{\phi}_{n,M}-\phi_*\|_{L^2_\rho} ^2 ] 
	&\leq C_{\beta,L,N, \bar\calL} M^{-\frac{2\beta}{2\beta+1}} 
\end{align*}
if $M$ is large enough, which implies \eqref{ineq:upper_main}.  
\end{proof}

\begin{remark} The left-tail probability bound via PAC-Bayes inequality can also yield the above rate for all $\beta\geq 1/4$ under an additional fourth-moment condition 
(Assumption {\rm \ref{assum:L2L4}}). The proof is the same as above except using  $G_{L,c_{\bar\calL}}(n,M)$ defined in \eqref{Def:G_Lc2}. 
\end{remark}

\smallskip

Recall $\widebar{\bfA}_{n}^{\infty} = \E[\widebar{\bfA}_{n}^{M}]=\lim_{M\to \infty}\widebar{\bfA}_{n}^{M}$ and $\widebar{\bfb}_{n}^{\infty} = \E[\widebar{\bfb}_{n}^{M}]=\lim_{M\to \infty}\widebar{\bfb}_{n}^{M}$ as defined in Lemma \ref{lemma:projEst}. Then, we can estimate the variance as follows.

\begin{lemma}[Bound for variance]\label{Lem:Ab_conv}
Under Assumption {\rm\ref{assum:bd-eigenfn}} on the basis functions with $C_{\max} = \sup\limits_{k\geq 1}\|\psi_k\|_\infty$, the following bound for the tamed LSE in Definition {\rm\ref{def:tlse}} satisfies  
\begin{equation}\label{eq:errorL2E}
	\E_{\phi_*}\Big[\Big\|\widehat \btheta_{n,M}- \btheta_{n}^*\Big\|^2_{\R^n}\Big] \leq C_0  c_{\bar\calL}^{-2} \frac n M + 2\epsilon_n^* +  G_{L,c_{\bar\calL}}(n,M) \,,
\end{equation}
where $C_0=2^6\sqrt{3}C_{\max}^4C_{\beta}^2 (\frac{C_\eta^{1/2} {\pi^{2\beta}}}{C_{\max}^4C_{\beta}^2L^2N} + L^2{/ \pi^{2\beta}})$, $\epsilon_n^*= c_{\bar\calL}^{-2}  \sum_{l= n+1}^\infty |\theta_l^*|^2$  and 
\begin{align}\label{Def:G_Lc}
	G_{L,c_{\bar\calL}}(n,M)= \frac{L^2}{\pi^{2\beta}}\calE_{c_{\bar\calL}}^{{\rm Che}}(n,M) \quad\text{ or }\quad \frac{L^2}{\pi^{2\beta}}\calE_{c_{\bar\calL}}^{{\rm Ber}}(n,M) 
\end{align}
with $\calE_{c_{\bar\calL}}^{{\rm Che}}(n,M)$ and $\calE_{c_{\bar\calL}}^{{\rm Ber}}(n,M)$ defined in \eqref{Eq:Low_Min_est_C2} and \eqref{Eq:Low_Min_est_B2}, respectively. Moreover, if {Assumption {\rm\ref{assum:L2L4}}} is also satisfied,  the inequality \eqref{eq:errorL2E} holds with 
\begin{align}\label{Def:G_Lc2}
	G_{L,c_{\bar\calL}}(n,M)=  \frac{L^2}{\pi^{2\beta}}\calE_{c_{\bar\calL}}^{{\rm PAC}}(n,M)\,,
\end{align}
where $\calE_{c_{\bar\calL}}^{{\rm PAC}}(n,M)$ is defined in \eqref{Eq:Low_Min_est2}.
\end{lemma}

\begin{proof}
Recall $\widehat \btheta_{n,M} = [\widebar{\bfA}_{n}^M]^{-1}\widebar{\bfb}_{n}^M\mathbf{1}_{\{ \lambda_{\min}(\widebar{\bfA}_{n}^M)> \frac{1}{4} c_{\bar\calL}\} }$ in Eq. \eqref{eq:tame-lse} and 
$$\btheta^*_n= [\widebar{\bfA}_{n}^\infty]^{-1}\widebar{\bfb}_{n}^\infty -\mathbf{v}$$
in Eq. \eqref{Eq:Asyp_theta} with $\mathbf{v} =  [\widebar{\bfA}_{n}^\infty]^{-1}\widetilde{\bfb}_{n}^\infty$ satisfying $\|\mathbf{v}\|^2\leq \epsilon_n^*= c_{\bar\calL}^{-2}  \sum_{l= n+1}^\infty |\theta_l^*|^2$ by \eqref{eq:btilde_bd}. 
For simplicity of notation, denote $\mathcal{A} := \{ \lambda_{\min}(\widebar{\bfA}_{n}^M)>  c_{\bar\calL}/4\}.$  
With $\mathcal{A}^c$ denoting the complement of the set $\mathcal{A}$, we have 
\begin{align*}
	\|\widehat \btheta_{n,M}- \btheta_{n}^*\|_{\R^{n}}^2 
	   &= \| [\widebar{\bfA}_{n}^M]^{-1} \widebar{\bfb}_{n}^M   - \btheta_{n}^*\|_{\R^{n}}^2\mathbf{1}_{\mathcal{A}} + \|\btheta_{n}^*\|_{\R^{n}}^2 \mathbf{1}_{\mathcal{A}^c} \\
	& \leq 2 \big( \| [\widebar{\bfA}_{n}^M]^{-1} \widebar{\bfb}_{n}^M  -[\widebar{\bfA}_{n}^\infty]^{-1}\widebar{\bfb}_{n}^\infty\|_{\R^{n}}^2  +\epsilon_n^* \big) \mathbf{1}_{\mathcal{A}} + \|\btheta_{n}^*\|_{\R^{n}}^2 \mathbf{1}_{\mathcal{A}^c} \\
	& \leq  4 \big(  \| [\widebar{\bfA}_{n}^{M}]^{-1} (\widebar{\bfb}_{n}^{M} - \widebar{\bfb}_{n}^{\infty})\|^2_{\R^n}  + \| ([\widebar{\bfA}_{n}^{M}]^{-1}- [\widebar{\bfA}_{n}^{\infty}]^{-1})\widebar{\bfb}_{n}^{\infty}\|^2_{\R^n} \big) \b1_{\mathcal{A}}  \\
	&\qquad +2 \epsilon_n^* +\|\btheta_{n}^*\|^2_{\R^n}  \mathbf{1}_{\mathcal{A}^c}. 
\end{align*}
Taking expectation, we get 
\begin{align*}
	\E_{\phi_*}\left[\|\widehat \btheta_{n}^M- \btheta_{n}^*\|^2_{\R^n}\right] &\leq 4\E_{\phi_*}\big[\| [\widebar{\bfA}_{n}^{M}]^{-1} (\widebar{\bfb}_{n}^{M} - \widebar{\bfb}_{n}^{\infty})\|^2_{\R^n} \b1_{\mathcal{A}} \big] \\
	&+ 4\E_{\phi_*}\big[\| ([\widebar{\bfA}_{n}^{M}]^{-1}- [\widebar{\bfA}_{n}^{\infty}]^{-1})\widebar{\bfb}_{n}^{\infty}\|^2_{\R^n} \b1_{\mathcal{A}} \big] +\|\btheta_{n}^*\|^2_{\R^n} \P\{\mathcal{A}^c\} + 2\epsilon_n^*\,. 
\end{align*}
The first three terms on the right-hand side are bounded as follows. 
Applying H\"older inequality and Lemma \ref{lemma:mmtBd_v2}, we have 
\begin{align*}
	\E_{\phi_*}\big[\| [\widebar{\bfA}_{n}^{M}]^{-1} (\widebar{\bfb}_{n}^{M} - \widebar{\bfb}_{n}^{\infty})\|_{\R^{n}}^2 \b1_{\mathcal{A}}\big] 
	\leq& (\E_{\phi_*}\|[\widebar{\bfA}_{n}^{M}]^{-1}\b1_{\mathcal{A}} \|^4)^{1/2} (\E_{\phi_*} [\|\widebar{\bfb}_{n}^{M} - \widebar{\bfb}_{n}^{\infty}\|^4_{\R^n}])^{1/2} \\
 	\leq& 16 c_{\bar\calL}^{-2} C_b\frac n M.  
\end{align*}
Similarly, using the facts $\big([\widebar{\bfA}_{n}^{\infty}]^{-1}-[\widebar{\bfA}_{n}^{M}]^{-1}\big)\widebar{\bfb}_{n}^{\infty} = [\widebar{\bfA}_{n}^{M}]^{-1} (\widebar{\bfA}_{n}^{M}- \widebar{\bfA}_{n}^{\infty})[\widebar{\bfA}_{n}^{\infty}]^{-1}\widebar{\bfb}_{n}^{\infty}$ on the set $\mathcal{A}$ and $\btheta^*_n= [\widebar{\bfA}_{n}^\infty]^{-1}\widebar{\bfb}_{n}^\infty$, we bound the second term as 
\begin{align*}
	 \E_{\phi_*} \big[\| ([\widebar{\bfA}_{n}^{M}]^{-1}-[\widebar{\bfA}_{n}^{\infty}]^{-1})\widebar{\bfb}_{n}^{\infty} \b1_{\mathcal{A}}\|_{\R^{n}}^2  \big] 
	\leq&  (\E_{\phi_*} \|[\widebar{\bfA}_{n}^{M}]^{-1} \b1_{\mathcal{A}} \|^4)^{\frac 12} (\E_{\phi_*}[\|(\widebar{\bfA}_{n}^{M}- \widebar{\bfA}_{n}^{\infty})\btheta^*_n \|_{\R^n}^4])^{\frac 12} \\
	 \leq& 16 c_{\bar\calL}^{-2} C_A\frac n M\,. 
\end{align*}
Following the Chernoff left tail probability \eqref{Eq:Low_Min_est_C2}, the Bernstein left tail probability \eqref{Eq:Low_Min_est_B2} and the fact that $\btheta_{n}^*$ satisfying $\sum\nolimits_{k=1}^\infty k^{2\beta}(\theta_k^*)^2 \leq \frac{L^2}{\pi^{2\beta}}$, we  have 
\begin{align}\label{Eq:inq_theta_star_A_c}
	\|\btheta_{n}^*\|^2_{\R^n} \P\{\mathcal{A}^c\}\leq \left( \sum\nolimits_{k=1}^\infty k^{2\beta}(\theta_k^*)^2\right) \P\{\mathcal{A}^c\} \leq   G_{L,c_{\bar\calL}}(n,M)\,.
\end{align} 
Combining the preceding three estimates, we have 
\begin{align*}
	\E_{\phi_*}\left[\|\widehat \btheta_{n,M}- \btheta_{n}^*\|^2_{\R^n}\right] 
	&\leq 64 c_{\bar\calL}^{-2} (C_A+C_b) \frac n M + 2\epsilon_n^*+G_{L,c_{\bar\calL}}(n,M)\,.
\end{align*}
Recall that $C_A$ and $C_b$ can be found in \eqref{Ineq:4MofA} and \eqref{Ineq:4Mofb}. 
Hence, {$2(C_A+C_b) \leq  2^{6}\sqrt{3}C_{\max}^4C_{\beta}^2 (L^2{/ \pi^{2\beta}}+\frac{C_\eta^{1/2}{\pi^{2\beta}}}{C_{\max}^4C_{\beta}^2L^2N})=:C_0$}, and we conclude the proof of \eqref{eq:errorL2E} by applying Lemmas \ref{thm:min-eigen_C} and \ref{thm:min-eigen_B}.  
The bound involving \eqref{Def:G_Lc2} follows similarly by employing \eqref{Eq:Low_Min_est2} from Lemma \ref{thm:min-eigen} to  \eqref{Eq:inq_theta_star_A_c}. 
\end{proof}

\smallskip

The succeeding lemma establishes the fourth-moment bounds for the normal vectors. Its proof is included in Appendix \ref{sec:proof_upperRate}.  
\begin{lemma}[Fourth-moment bounds for the normal vectors]\label{lemma:mmtBd_v2}
Let $\widebar{\bfA}_{n}^\infty = \E_{\phi_*}[\widebar{\bfA}_{n}^{M}] $ and $\widebar{\bfb}_{n}^\infty = \E_{\phi_*}[\widebar{\bfb}_{n}^{M}]$, where $\widebar{\bfA}_{n}^{M}$ and $\widebar{\bfb}_{n}^{M}$ are defined in \eqref{eq:Ab}. Let $\btheta_n^*=(\theta_1^*,\ldots,\theta_n^*)$ be the first $n$ coefficients of the true function $\phi_*$. 
\begin{itemize}
\item 
When $\beta>1/2$, under Assumption {\rm\ref{assum:bd-eigenfn}} on the basis functions, we have
\begin{subnumcases}{\label{Ineq:4MofAb}}
    \Big(\E_{\phi_*}\big[\| ( \widebar{\bfA}_{n}^{M}-\widebar{\bfA}_{n}^\infty)\btheta^*_n\|_{\R^n}^4\big]\Big)^{\frac 12}  \leq C_A \frac n M \,; \label{Ineq:4MofA}\\ 
		\Big(\E_{\phi_*}\big[\| \widebar{\bfb}_{n}^{M}-\widebar{\bfb}_{n}^\infty\|_{\R^n}^4\big]\Big)^{\frac 12} \leq C_b \frac n M \,,\label{Ineq:4Mofb}
\end{subnumcases}
where the constants $C_A = \sqrt{96}C_{\beta}^2C_{\max}^4 L^2{/ \pi^{2\beta}}$ and $C_b =  \sqrt{798}C_{\max}^2(C_{\beta}^4C_{\max}^4L^4{/ \pi^{4\beta}} +\frac{1}{N^2}C_{\eta} )^{1/2} $ with $C_{\max} = \sup_{k\geq 1}\|\psi_k\|_\infty$, $C_{\beta}=(\sum_{k=1}^\infty k^{-2\beta})^{1/2}$ are independent of $n$ and $M$. 

\item When $\beta\geq 1/4$, under Assumption {\rm\ref{assum:bd-eigenfn}} and Sobolev embedding Assumption {\rm\ref{Assu:suff_L4}}, we have the estimates in \eqref{Ineq:4MofAb} hold with $C_A = \sqrt{96\calK_{\beta}}C_{\max}^2 L^2$, $C_b = \sqrt{798}C_{\max}^2(\calK_{\beta} L^4 +\frac{1}{N^2}C_{\eta} )^{1/2}$ and $\calK_{\beta}$ is the Sobolev embedding constant in Eq. \eqref{Assu:suff_L4}.
\end{itemize}
\end{lemma}

\subsection{Upper minimax rate: non-uniformly bounded basis functions}\label{Sec:UpbdMini_NonUnif}

The assumption of a uniformly bounded basis may not hold in general. For example, the basis functions $\{ \psi_k(x) = \frac{\sqrt{2}\sin(k\pi x)} {\sqrt{\rho'(x)}} \}_{k= 1}^{\infty}$ are not uniformly bounded when the continuous density $\rho'$ does not have a positive lower bound, as in the example in Appendix \ref{sec:unif-X}.

To consider the case of a non-uniformly bounded basis, we introduce the following mild growth condition on the basis functions. 
\begin{assumption}[Mild growth condition]\label{assum:L4-eigenfn} 
The complete orthonormal basis functions $\{\psi_k\}_{k= 1}^{\infty}$ satisfy the uniform $L^4_{\rho}([0,1])$-bound condition:
\begin{equation}\label{cond:L4-eigenfn}
	\sup_{k\geq 1}\|\psi_k\|_{L^4_{\rho}}\leq \widetilde{L}<\infty \,, 
\end{equation} 
 and the $\delta$-power growth condition:
\begin{equation}\label{cond:sqrt_n}
	\sup_{1\leq k\leq n} \| \psi_k\|_{\infty} \leq C_{\delta} n^{\delta}, \forall n\geq 1,  
\end{equation}
where $\delta\geq 0$ and $C_{\delta}$ is a positive constant depending only on $\delta$. 
\end{assumption} 

Assumption \ref{assum:L4-eigenfn} holds for the basis functions $\psi_k(x)=e_k(x)/\sqrt{\rho'(x)}$ with $e_k(x)=\sqrt{2}\sin(k\pi x)$ and $\rho'(x)=\frac 32\sqrt{1-x}$ on $[0,1]$. Clearly, this basis $\{\psi_k(x)\}_{k=1}^{\infty}$ is orthonormal and complete in $L^2_\rho$. Also, the $\delta$-power growth condition \eqref{cond:sqrt_n} holds with $\delta=\frac{1}{4}$ since for each $k\geq 1$, the function $|\psi_k(x)|=\frac{2}{\sqrt{3}}\frac{|\sin(k\pi x)|}{(1-x)^{1/4}}$ takes the global maximum value at $x=1-\frac{1}{2k}$ or $x=1-\frac{3}{2k}$. Moreover, for each $k\geq 1$
\begin{align*}
	\|\psi_k\|_{L^4_{\rho}}^4 = \int_0^1 \bigg| \frac{\sqrt{2}\sin(k\pi x)}{\sqrt{\rho'(x)}}\bigg|^4 \rho'(x) dx \leq \frac 23 \int_0^1 \frac{1}{\sqrt{1-x}} dx =\frac 43\,.
\end{align*}

\begin{remark}
	The $\delta$-power growth condition {\rm \eqref{cond:sqrt_n}} is inspired by the following assumption (see, for example, {\rm \cite[Section 7.4.2]{Pascal2007}})
\begin{equation}\label{assum:sqrt_n_R1}
	\left\| \sum_{k=1}^{n} \theta_k \psi_k \right\|_{\infty} \leq \sqrt{n} \max_{k=1,...,n}| \theta_k|\quad \text{for all } \btheta=(\theta_k)\in \ell^2\,.
\end{equation} 
Clearly, this condition implies the $\delta$-growth condition with $\delta=1/2$, which follows by taking $\btheta=e_k:=(0,\cdots,1,\cdots,0)\in \R^n$ for all $1\leq k\leq n$ in Eq. {\rm \eqref{assum:sqrt_n_R1}}. 
\end{remark}

Our main result of this section is the forthcoming theorem, which shows that the tamed LSE estimator achieves the optimal upper minimax convergence rate for the case with non-uniformly bounded basis functions.
\begin{theorem}[upper minimax rate: non-uniformly bounded basis]\label{cor:L2rho_upper_case2}
Suppose that Assumption {\rm\ref{assump:noise-X}},  Assumption {\rm\ref{assump:noise} \ref{Assump:Fourth_Noise}} and Assumption {\rm \ref{assum:bd-eigenfn}} on the model hold. Let Assumption {\rm\ref{assum:L4-eigenfn} on the basis hold with some $\delta\geq 0$}. If  $\beta> \delta +\frac12$ then
\begin{equation}\label{ineq:upper_case2}
\limsup_{M\to\infty} \inf_{\widehat{\phi}} \sup_{\phi_*\in H_{\rho}^{\beta}(L)}    \E_{\phi_*}\left[ M^{\frac{2\beta}{2\beta+1}} \| \widehat{\phi}-\phi_* \|_{L^2_\rho} ^2 \right] \leq C_{\textup{upper}}\,,
\end{equation}
where $C_{\textup{upper}}>0$ is a constant independent of $M$ and the infimum is among all estimators $\widehat \phi_M$ inferred from data $\{(X^m,Y^m)\}_{m=1}^M$. 
\end{theorem}
Compared to the one in Eq. \eqref{ineq:upper_main2}, the positive constant $C_{\textup{upper}}$ in Eq. \eqref{ineq:upper_case2} additionally depends on $\delta$ and $\widetilde{L}$ introduced in Assumption \ref{assum:L4-eigenfn}.  
The proof of Theorem {\rm \ref{cor:L2rho_upper_case2}} is similar to Theorem {\rm \ref{cor:L2rho_upper}}. We postpone the detailed proof with the precise value of $C_{\textup{upper}}$ in Appendix {\rm \ref{subsec:case2}}.

\subsection{Left tail probability of the smallest eigenvalue}\label{sec:prob_eig_bd}

Recall that the smallest eigenvalue of $\bar{\bfA}_{n}^{M}$ is defined as
\begin{align*}
	\lambda_{\min}(\bar{\bfA}_{n}^{M})&=\inf_{\btheta\in S^{n-1}}\btheta^\top \bar{\bfA}_{n}^{M}\btheta
	 = \inf_{\btheta\in S^{n-1}}\frac{1}{MN}\sum\nolimits_{m=1}^M \|R_{\phi_{\btheta}}[X^m]\|_{\R^{Nd}}^2, 
\end{align*}
where $\phi_{\btheta}=\sum_{k=1}^n \theta_k\psi_k$ and $\btheta=(\theta_1,\cdots,\theta_n)\in S^{n-1}$, i.e., $\|\btheta\|_{\R^n}=1$. 

We characterize the left tail probability of $\lambda_{\min}(\bar{\bfA}_{n}^{M})$ in terms of its exponential decay in $M$ and increment in $n$ in Lemma \ref{thm:min-eigen_C}-- Lemma \ref{thm:min-eigen}, which are proved using the Chernoff, Bernstein, and PAC-Bayesian inequalities, respectively. 


\begin{lemma}[Chernoff left tail probability of the smallest eigenvalue]\label{thm:min-eigen_C}
	Consider $\bar{\bfA}_n^M$ as defined in \eqref{eq:Ab_A} associated with the basis functions $\{\psi_k\}_{k=1}^{\infty}$ satisfying Assumption {\rm\ref{assum:bd-eigenfn}}. Then, we have
\begin{equation}\label{Eq:Low_Min_est_C}
		\P\left\{ \lambda_{\min}(\bar{\bfA}_{n}^{M}) \leq (1-\varepsilon) c_{\bar\calL} \right\}\leq n  \sbracket{ \frac{e^{-\varepsilon}}{(1-\varepsilon)^{1-\varepsilon}}}^{\frac{c_{\bar\calL}  M}{nC_{\max}^2}}\,
	\end{equation}
	for any $\varepsilon\in (0,1)$. In particular, 
	\begin{equation}\label{Eq:Low_Min_est_C2}
		\P\left\{ \lambda_{\min}(\bar{\bfA}_{n}^{M}) \leq \frac{c_{\bar\calL}}{2}  \right\}\leq n \sbracket{\frac{e}{2}}^{-\frac{c_{\bar\calL}  M}{2nC_{\max}^2}}=:\calE_{c_{\bar\calL}}^{{\rm Che}}(n,M) \,.
	\end{equation}
\end{lemma}
\begin{proof}
    The proof follows from the matrix Chernoff inequality (Theorem 5.1.1 in \cite{Tropp2015}), which we recall in Appendix \ref{sec:append}. For $\bar{\bfA}_{n}^{\infty}=\E[\bar{\bfA}_{n}^{M}]$, note that Lemma \ref{lemma:projEst} implies
	\begin{equation}\label{Eq:Coer_Mat_Che}
		\lambda_{\min}(\bar{\bfA}_{n}^{\infty})=\inf_{\btheta\in S^{n-1}}\btheta^\top \bar{\bfA}_{n}^{\infty}\btheta=\inf_{\btheta\in S^{n-1}}\frac{1}{N}\E[\|R_{\phi_{\btheta}}[X]\|_{\R^{Nd}}^2]\geq c_{\bar \calL}>0\,.
	\end{equation} 	
	We denote $\bPhi^m =(R_{\psi_1}[X^{m}],\ldots,R_{\psi_n}[X^{m}]) \in \R^{Nd\times n}$ for each sample $X^m$ and write 
    \begin{align*}
        \bar{\bfA}_{n}^{M}=\frac{1}{MN}\sum_{m=1}^M [\bPhi^m]^\top \bPhi^m=:\sum_{m=1}^M Q^{m}, 
    \end{align*}
    where $\{Q^{m}=\frac{1}{MN}[\bPhi^m]^\top \bPhi^m\}_{m=1}^M$ is a sequence of  independent Hermitian matrices with dimension $n$. It is clear $\lambda_{\min}(Q^{m})\geq 0$ and 
	\begin{align*}
		\lambda_{\max}(Q^{m})=\sup_{\btheta\in S^{n-1}}\btheta^\top Q^{m} \btheta\leq \text{tr}(Q^{m}) \leq \frac{nC_{\max}^2}{M}\, 
	\end{align*} 
for all $1\leq m\leq M$ since $Q^m_{ij}\le \frac{C^2_{\max}}{M}$ for all $1\leq i,j\leq n$ by the uniformly bounded condition on the basis functions. Thus, the matrix Chernoff inequality (Theorem \ref{thm:matChernoff}) and Eq. \eqref{Eq:Coer_Mat_Che} give the left tail probability inequality \eqref{Eq:Low_Min_est_C}, namely
	\begin{align*}
		\P\left\{ \lambda_{\min}(\bar{\bfA}_{n}^{M}) \leq (1-\varepsilon) c_{\bar\calL} \right\} &\leq \P\left\{ \lambda_{\min}(\bar{\bfA}_{n}^{M}) \leq (1-\varepsilon) \lambda_{\min}(\bar{\bfA}_{n}^{\infty}) \right\} \\
		&\leq  n  \sbracket{ \frac{e^{-\varepsilon}}{(1-\varepsilon)^{1-\varepsilon}}}^{\frac{\lambda_{\min}(\bar{\bfA}_{n}^{\infty})  M}{nC_{\max}^2}} \leq n \sbracket{\frac{e^{-\varepsilon}}{(1-\varepsilon)^{1-\varepsilon}}}^{\frac{c_{\bar\calL}  M}{nC_{\max}^2}}\,.
	\end{align*} 
	Letting $\varepsilon=\frac 12$, we get the left tail probability inequality \eqref{Eq:Low_Min_est_C2}.
\end{proof}

\begin{lemma}[{Bernstein} left tail probability of the smallest eigenvalue]\label{thm:min-eigen_B}
	Consider $\bar{\bfA}_n^M$ as defined in \eqref{eq:Ab_A} associated with the basis functions $\{\psi_k\}_{k= 1}^{\infty}$ satisfying Assumption {\rm\ref{assum:bd-eigenfn}}. Then, we have
	\begin{equation}\label{Eq:Low_Min_est_B2}
		\P\left\{ \lambda_{\min}(\bar{\bfA}_{n}^{M}) \leq \frac{c_{\bar\calL}}{4}  \right\}\leq 2n \exp\rbracket{-\frac{9Mc_{\bar\calL}^2/64}{n^2 C_{\max}^4+C_{\max}^2c_{\bar\calL}/4}}=:\calE_{c_{\bar\calL}}^{{\rm Ber}}(n,M)\,.
	\end{equation}
\end{lemma}
The matrix Bernstein inequality is commonly used in the context of learning interaction kernels; see, for example, {\rm\cite{LMT21_JMLR,LMT22,LZTM19pnas}}. The proof of it is included in Appendix \ref{sec:Mat_Concen}.  

\begin{remark}\label{Rmk:Mat_Ber}
A limitation of the Bernstein left tail probability bound in \eqref{Eq:Low_Min_est_B2} is that it requires $\beta>1/2$ to ensure exponential decay as $M\to \infty$ within the minimax framework with $n=M^{\frac{1}{2\beta+1}}$. In contrast, the Chernoff left tail probability bound in \eqref{Eq:Low_Min_est_C} requires only $\beta>0$ to achieve the exponential decay, and it enables the optimal upper minimax rate for hypothesis space $H_{\rho}^{\beta}(L)$ with $\beta\leq 1/2$, which includes discontinuous interaction kernels that are widely used in opinion dynamics.    
\end{remark}


Next, we introduce another left tail probability of the smallest eigenvalue based on PAC-Bayes inequality (Lemma \ref{lemma:PAC_Bay}). The following fourth-moment condition is crucial in the application of PAC-Bayes inequality. It provides an alternative to the Sobolev embedding condition in Assumption \ref{assum:SoboEmb}, particularly in the context of learning convolution kernels in \cite[Section 3.2]{ZhangWangLu2025}.  
It is an extension of the $L^4-L^2$ norm equivalence condition on the distribution of the input random vector for parametric linear regression in {\rm\cite[Eq.~(3)]{Oliveira2016} and \cite[Assumption 3]{Mourtada2022}} and for covariance matrix estimation in \cite{Zhivotovskiy20}. Here, the condition leads to constraints on the distribution of $X$.

\begin{assumption}[Fourth-moment condition]\label{assum:L2L4} 
Let $S^{n-1}$ be the $n$-dimensional sphere and $\phi_{\theta}=\sum_{k=1}^n \theta_k \psi_k$ with $\theta\in S^{n-1}$ and $\{\psi_k\}$ being orthonormal basis functions of $L_{\rho}^{2}$.
Suppose there exists a constant $\kappa>0$, independent of $n$, such that 
\begin{equation}\label{eq:4th-2ndMmt}
\sup_{\theta\in S^{n-1}} \frac{ \E[ \|R_{\phi_{\theta}}[X]\|_{\R^{Nd}}^4]}{(\E[\|R_{\phi_{\theta}}[X]\|_{\R^{Nd}}^2])^2}\leq \kappa <\infty\,. 
\end{equation}
\end{assumption}

\begin{lemma}[{PAC-Bayes} left tail probability of the smallest eigenvalue]\label{thm:min-eigen}
Consider $\bar{\bfA}_n^M$ as defined in \eqref{eq:Ab_A} associated with the basis functions $\{\psi_k\}_{k= 1}^{\infty}$ satisfying Assumption {\rm\ref{assum:bd-eigenfn}} and {the fourth moment condition in Assumption {\rm{\ref{assum:L2L4}}} holds}. 
Then, we have 
 \begin{equation}\label{Eq:Low_Min_est2}
     \P\left\{\lambda_{\min}(\bar{\bfA}_{n}^{M}) \leq \frac{c_{\bar\calL}}{4} \right\} \leq  \exp\rbracket{n\log\left(\frac{5C_{\max}^2}{c_{\bar\calL}}\right)-{\frac{ M c_{\bar\calL}^2}{64 \kappa}}}=:\calE_{c_{\bar\calL}}^{{\rm PAC}}(n,M) 
 \end{equation}
where {$M\geq \frac{256 \kappa }{9 c_{\bar\calL}^2 }\log\left(\frac{5C_{\max}^2}{c_{\bar\calL}}\right) n$} and $n\geq 2$.
\end{lemma}
 
The proof for Lemma \ref{thm:min-eigen} is presented in Appendix \ref{sec:proof_left_tail_prob2}. Our proof adapts the approach outlined in \cite{Mourtada2022}, with simplifications tailored to the distinct assumptions inherent in a nonparametric setting. The primary tool is the PAC-Bayes inequality, which we recall in Lemma \ref{lemma:PAC_Bay}, introduced to address the left tail probability of the smallest eigenvalue in \cite{Oliveira2016} and further customized in \cite{Mourtada2022}.

The bound \eqref{Eq:Low_Min_est2} sharpens the Chernoff left-tail estimate \eqref{Eq:Low_Min_est_C} by eliminating the factor $n$, but it does not change the resulting optimal minimax rate. Since the matrix Chernoff approach is substantially simpler, our main upper minimax results (Theorems \ref{cor:L2rho_upper} and \ref{cor:L2rho_upper_case2}) do not invoke the PAC-Bayesian left-tail bound and hence avoid Assumption \ref{assum:L2L4}. We include this bound to illustrate and compare alternative techniques for deriving left-tail probability bounds, which are of interest in their own right.


\section{Lower minimax rate}\label{sec:lowerBd}
This section is dedicated to the lower minimax rate by the Fano-Tsybakov method \cite[Chapter 2]{tsybakov2008introduction}. The lower rate matches the upper rate in Theorem \ref{thm:L2rho_upper}, confirming the optimality of the rate. 
Recall that $\calC_0^{\beta}(L)$ is the H\"older continuous class defined in \eqref{eq:class-H\"older}, $\rho$ is the exploration measure in Definition {\rm\ref{def:rho}}, and $ \E_{\phi_*}$ is the expectation with respect to the dataset $\{(X^{m},Y^m)\}_{m=1}^M$ generated from model \eqref{eq:model} with $\phi_*$.  

\begin{theorem}[lower minimax rate]\label{thm:L2rho_lower} 
Under Assumption {\rm \ref{assump:noise-X}} and Assumption \ref{assump:noise} \ref{Assump:Density_Noise}, if $\beta>0$, then there exists a constant $c_{\textup{Lower}}>0$ independent of $M$ such that
\begin{equation}\label{ineq:lbd_main}
\liminf_{M\to \infty} \inf_{\widehat{\phi}_M} \sup_{\phi_*\in \calC_0^{\beta}(L)}    \E_{\phi_*}\left[  M^{\frac{2\beta}{2\beta+1}} \| \widehat{\phi}_M-\phi_* \|_{L^2_\rho} ^2 \right]\geq c_{\textup{Lower}}
\end{equation}
where $\inf_{\widehat{\phi}_M}$ is the infimum over all estimators based on the observation model with $M$ i.i.d. samples. Here, $c_{\textup{Lower}}=c_0 c_{\beta,N}$ with $c_0$ independent of $M,N$ and $c_{\beta,N}= N^{-\frac{2\beta}{2\beta+1}}$.
\end{theorem}

For the lower minimax rate, we note that the exploration density $\rho'$ need not be bounded below away from zero. Thus, it can be utilized to match the upper minimax rates for both uniformly and non-uniformly bounded basis functions, as established in Theorems \ref{cor:L2rho_upper} and \ref{cor:L2rho_upper_case2}. 
We follow the general scheme in \cite[Chapter 2 and Theorem 2.11] {tsybakov2008introduction}. This scheme reduces the infimum over all estimators and the supremum over all functions to the bound of the probability of testing error of a finite hypotheses test. 
We summarize it in three steps, as follows. The detailed proof is provided in   Appendix \ref{sec:append2}.
\begin{enumerate}[leftmargin=+.58in]
    \item[Step 1:] Reduce \eqref{ineq:lbd_main} to bounds in probability by Markov inequality and to a finite number of hypotheses $\Theta=\{\phi_{0,M},\cdots,\phi_{K,M}\}\subseteq \calC_0^{\beta}(L)$. We set $\phi_{0,M}\equiv 0$ so that $\P_k \ll \P_0$, where $\P_k$ denotes the measure of the model with $\phi_{k,M}$. 
    \item[Step 2:] Transform to bounds in the average probability of the test error of $2s$-separated hypotheses. The key idea in the transformation is a minimum distance test $\kappa_{\text{test}} = \argmin_{1\leq k\leq K} d(\widehat \phi_M,\phi_{k,M})$ as in \cite[(2.8)]{tsybakov2008introduction}.  
    \item[Step 3:] Bound the average probability of the test error from below by the Kullback-Leibler divergence of the hypotheses. 
\end{enumerate}
  
Our main innovation, which is also the major difficulty, is the construction of the hypotheses $\{\phi_{0,M},\phi_{1,M},\cdots,\phi_{K,M}\}\subseteq \calC_0^{\beta}(L)$ satisfying two conditions: (i) they are $2s$-separated in $L^2_\rho$, namely, $\|\phi_{k,M}-\phi_{k',M}\|_{L^2_\rho}\geq 2s$ and (ii) their average Kullback-Leibler divergence $\KL(\P_{k}, \P_{0})$ has a logarithmic growth in $K$. These two conditions are used in the next lemma to prove Step 3. This lemma follows from a combination of a lower bound based on multiple hypotheses, Fano's lemma, and its corollary, which are in \cite[Theorem 2.5, Lemma 2.10, and Corollary 2.6]{tsybakov2008introduction} respectively, and we omit its proof.  

\begin{remark}[Optimal minimax rate]\label{Rmk:Holder_minimax}
As we discussed in Remark \ref{rmk:Soblev_Holder}, 
$\calC_0^{\beta}(L)\subseteq H_{0,\rho}^{\beta}(L)$ when $\beta$ is an integer. 
Thus, combing the upper bound rate \eqref{ineq:upper_main2} in Theorem \ref{cor:L2rho_upper} and the lower bound rate \eqref{ineq:lbd_main} in Theorem \ref{thm:L2rho_lower} implies the optimal minimax rate:
\begin{align*}
    c_{\textup{Lower}}&\leq \liminf_{M\to \infty} \inf_{\widehat{\phi}_M} \sup_{\phi_*\in \calH^{\beta}}    \E_{\phi_*}[  M^{\frac{2\beta}{2\beta+1}} \| \widehat{\phi}_M-\phi_* \|_{L^2_\rho} ^2 ] \\
	&\leq \limsup_{M\to \infty} \inf_{\widehat{\phi}_M} \sup_{\phi_*\in \calH^{\beta}}    \E_{\phi_*}[  M^{\frac{2\beta}{2\beta+1}} \| \widehat{\phi}_M-\phi_* \|_{L^2_\rho} ^2 ]\leq C_{\textup{Upper}}\,,
\end{align*}
where $\calH^{\beta}$ can be $\calC_0^{\beta}(L)$, $H_{0,\rho}^{\beta}(L)$ or $\calC^{\beta}(L)$, $H_{\rho}^{\beta}(L)$ with (weighted) trigonometric basis functions.

For fractional-valued $\beta$, namely, $\beta\in\bigcup_{l=1}^{\infty}(l, l+1)$, Lemma \ref{Lem:Holder<spectral} establishs 
\begin{align*}
	\mathcal{C}_0^{\beta_{\varepsilon}}(Q)\subseteq H_{0,\rho}^{\beta}(L), 
\end{align*}
where $\beta_{\varepsilon} := \beta + \varepsilon\leq l+1$ with $\varepsilon > 0$ is a sufficiently small positive number and $Q$ may depend on $L$, $\beta$, and $\varepsilon$. The inclusion of $\mathcal{C}_0^{\beta}(L)$ in $H_{0,\rho}^{\beta}(L)$ may not hold since $\mathcal{C}_0^{\beta}(L)$ is not even necessary a subset of the Gagliardo Sobolev space $W^\beta$. Thus, following the ``inclusion'' scheme in \cite{tsybakov2008introduction} as above, we can only obtain a suboptimal rate for $\beta_{\varepsilon} = \beta + \varepsilon$, \begin{align*}
	c_{\textup{Lower}}&\leq \liminf_{M\to \infty} \inf_{\widehat{\phi}_M} \sup_{\phi_*\in \calC_0^{\beta_{\varepsilon}}(L)}    \E_{\phi_*}[  M^{\frac{2\beta_{\varepsilon}}{2\beta_{\varepsilon}+1}} \| \widehat{\phi}_M-\phi_* \|_{L^2_\rho} ^2 ] \\
	&\leq \limsup_{M\to \infty} \inf_{\widehat{\phi}_M} \sup_{\phi_*\in H_{\rho}^{\beta}(L)}    \E_{\phi_*}[  M^{\frac{2\beta}{2\beta+1}} \| \widehat{\phi}_M-\phi_* \|_{L^2_\rho} ^2 ]\leq C_{\textup{Upper}}\,.
\end{align*}
This phenomenon underscores the additional challenges posed by fractional Sobolev classes compared to their integer-order counterparts. In particular, the standard inclusion argument is no longer applicable when one seeks optimal minimax rates.  A potential solution is to derive the lower-bound rate \eqref{ineq:lbd_main} by replacing $\mathcal{C}_0^{\beta}(L)$ with either $H_{0,\rho}^{\beta}(L)$ or $H_{\rho}^{\beta}(L)$. An alternative approach is the binary-coefficient construction in the eigenfunction expansion,  which has successfully achieved this for $\beta>1/2$ (see, e.g., {\rm \cite{CaiYuan2012, SunDuWangMa2018, yuan2010reproducing, ZhangWangLu2025}}). But its validity when $\beta \le 1/2$ and when the exploration density $\rho'$ may vanish, the regime of primary interest here, remains unclear. Exploring binary-coefficient techniques in the fractional Sobolev regime with non-uniformly bounded basis represents a potential direction for future work.
\end{remark}

\begin{lemma}[Lower bound for hypothesis test error ]\label{thm:lb_hypothesis_KL} 
Let $\Theta= \{\theta_k\}_{k=0}^K$ with $K\geq 2$ be a set of $2s$-separated hypotheses, i.e., $d(\theta_k,\theta_{k'})\geq 2s>0$ for all $ 0\leq k<k'\leq K$, for a given metric $d$ on $\Theta$.  Denote $\P_{k}=\P_{\theta_k}$ and suppose they satisfy $\P_{k}\ll \P_{0}$ for each $ k\geq 1$ and 
		\begin{equation}\label{Ineq:Kullback}
			\frac{1}{K+1} \sum\nolimits_{k=1}^K \KL(\P_{k}, \P_{0}) \leq \alpha \log(K)\,, \quad \text{with } 0<\alpha<1/8\,.
		\end{equation}
Then, the average probability of the hypothesis testing error  
has a lower bound:  
   \begin{equation}\label{Main:Lower_bd}
       \       \bar{p}_{e,M}:=\inf_{\kappa_{\textup{test}}} \frac{1}{K+1} \sum\nolimits_{k=0}^K \P_{k}\big(\kappa_{\textup{test}} \neq k  \big) \geq \frac{\log(K+1)-\log(2)}{\log(K)}-\alpha\,, 
   \end{equation}
   where $\inf_{\kappa_{\textup{test}}}$ denotes the infimum over all tests based on the observation model with $M$ i.i.d. samples.  
\end{lemma}

The next lemma constructs the hypothesis functions $\{\phi_{0,M},\phi_{1,M},\cdots,\phi_{K,M}\}$. Its proof is deferred to Appendix \ref{sec:append2}. 

\begin{lemma}\label{lemma:construction}
For each data set $\{(X^{m},Y^m)\}_{m=1}^M$, there exists a set of hypothesis functions $\{\phi_{0,M}\equiv 0,\phi_{1,M},\cdots,\phi_{K,M}\}$ and positive constants $C_0,C_1$ independent of $M$ and $N$, where 
\begin{equation}\label{eq:K}
K\geq 2^{\bar K/8},\quad \text{ with } \bar K= \lceil c_{0,N} M^{\frac{1}{2\beta+1}}\rceil,  \quad  c_{0,N}=C_0 N^{\frac{1}{2\beta+1}}, 
\end{equation}
such that the following conditions hold:  
\begin{enumerate}[label={\rm (C-\arabic*)},leftmargin=\widthof{(C-3)}+\labelsep]
	\item\label{Cond:a} H\"older continuity:  $\phi_{k,M}\in \calC_0^{\beta}(L)$ (defined in Eq. \eqref{eq:class-H\"older}) for each $k=1,\cdots, K$;   
	\item\label{Cond:b} $2s_{N,M}$-separated: $\|\phi_{k,M}-\phi_{k',M}\|_{L^2_\rho}\geq 2s_{N,M}$ with $s_{N,M}= C_1 c_{0,N}^{-\beta} M^{-\frac{\beta}{2\beta+1}}$;   
	\item\label{Cond:c} Kullback-Leibler divergence estimate: $\frac 1{K} \sum_{k=1}^{K} \KL(\bar{\P}_k,\bar{\P}_0)\leq \alpha \log(K)$ with ${\alpha  <1/8}$, 
	where $ \bar{\P}_k(\cdot )={\P}_{\phi_{k,M}}(\cdot \mid X^{1},\ldots, X^{M})$. 
\end{enumerate}
\end{lemma}

\begin{remark}[The exponent in $N$]\label{rmk:rateN2}  The constant $c_{\beta,N}$ in the lower bound in Eq. \eqref{ineq:lbd_main} suggests a potential lower bound rate of $N^{-\frac{2\beta}{2\beta+1}}$ when $N\to \infty$. This rate is determined by the exponent for $N$ in $c_{0,N}=C_0 N^{\frac{1}{2\beta+1}}$ in \eqref{eq:K}, which is the smallest possible. That is, when one replaces the the constant $c_{0,N}=C_0 N^{\frac{1}{2\beta+1}}$ in by $c_{0,N}=C_0 N^{\gamma}$, the exponent $\gamma$ must satisfy $\gamma\geq \frac{1}{2\beta+1}$. Such a constraint arises when we aim for $\alpha<\frac 1 8$ in Appendix \ref{sec:append2} for all $N$. We remark that a $1/\log(N)$ minimax bound is provided in \cite[Proposition 4.6, Theorem 5.1]{belomestny2023semiparametric} under a semiparametric setting.
\end{remark}

\appendix
\section{Technical results and proofs for the upper bound}\label{sec:append}

In subsection \ref{sec:proof_upperRate} and subsection \ref{sec:proof_left_tail_prob2}, we gather the remaining proofs of Lemma \ref{lemma:mmtBd_v2}, and Lemma \ref{thm:min-eigen_B} from Section \ref{sec:upperBd}. Subsection \ref{sec:Mat_Concen} collects the matrix concentration inequalities used throughout the analysis. In subsection \ref{subsec:case2}, we prove the upper minimax rate with a non-uniformly bounded basis.

\subsection{Proofs for the upper minimax rate}\label{sec:proof_upperRate}

This section presents technical proofs in Section \ref{Sec:UpbdMini}. 
First, we present the proof of Lemma \ref{lemma:mmtBd_v2}. Recall the fourth-moment bounds for the normal vectors:
\begin{subnumcases}{\label{Ineq:4MofAb_ap}}
    \Big(\E\big[\| ( \widebar{\bfA}_{n}^{M}-\widebar{\bfA}_{n}^\infty)\btheta^*_n\|_{\R^n}^4\big]\Big)^{\frac 12}  \leq C_A \frac n M \,, \label{Ineq:4MofA_ap} \\ 
		\Big(\E\big[\| \widebar{\bfb}_{n}^{M}-\widebar{\bfb}_{n}^\infty\|_{\R^n}^4\big]\Big)^{\frac 12} \leq C_b \frac n M \,, \label{Ineq:4Mofb_ap}
\end{subnumcases}
for some constants $C_A$ and $C_b $ independent of $n$ and $M$. 
Here and throughout the proof, it is unnecessary to distinguish between $\E$ and $\E_{\phi_*}$. Therefore, we simplify the notation by using $\E$ exclusively.

\smallskip

\begin{proof}[Proof of Lemma \ref{lemma:mmtBd_v2}] We prove these bounds by applying the fourth-moment bounds for mean in Lemma \ref{lemma:4th_mmt_empirical_mean}. To do so, we only need to show that both $\widebar{\bfA}_{n}^{M}\btheta^*_n -\widebar{\bfA}_{n}^\infty\btheta^*_n$ and $\widebar{\bfb}_{n}^{M}-\widebar{\bfb}_{n}^\infty$ are centered empirical means of two random vectors, each of which random vector has bounded fourth-moment.  

{\bf Part I: the normal matrix term.} We start from $\widebar{\bfA}_{n}^{M}\btheta^*_n -\widebar{\bfA}_{n}^\infty\btheta^*_n$. 
Let $\phi_n^*= \sum_{k=1}^n \theta_k^* \psi_k$. Since $R_\phi[X]$ is linear in $\phi$, we have $R_{\phi_n^*}[X] = \sum_{k=1}^n \theta_k^*R_{\psi_k}[X]$ for any $X$. Define an $ \R^n$-valued random vector $Z_A$ to be 
\begin{equation}\label{eq:Za}
	Z_A(l) =  \frac 1 N \innerp{R_{\psi_l}[X], R_{\phi_n^*}[X]}_{\R^{Nd}}, \quad 1\leq l\leq n. 
\end{equation} 
Then, we can write the $ \R^n$-valued random variable $\widebar{\bfA}_{n}^{M}\btheta^*_n$ as 
\begin{align*}
	[\widebar{\bfA}_{n}^{M}]\btheta^*_n(l) &= \sum_{k=1}^n \frac{1}{MN}\sum_{m=1}^M\innerp{R_{\psi_l}[X^{m}], R_{\psi_k}[X^{m}]}_{\R^{Nd}} \theta_k^*  \\
	& =    \frac{1}{M}\sum_{m=1}^M \frac 1 N \innerp{R_{\psi_l}[X^{m}], R_{\phi_n^*}[X^{m}]}_{\R^{Nd}} = \frac{1}{M}\sum_{m=1}^M Z_A^m(l), 
\end{align*}
where $Z_A^m(l)= \frac 1 N \innerp{R_{\psi_l}[X^{m}], R_{\phi_n^*}[X^{m}]}_{\R^{Nd}} $ is a sample of $Z_A(l)$ for each $m$. Also, $\widebar{\bfA}_{n}^\infty\btheta^*_n (l) = \E[Z_A(l)]$ for each $l$ by the definition of $\widebar{\bfA}_{n}^\infty$. 

Meanwhile, note that by the definitions of $Z_A$ and $R_\phi[X]$, we have 
\begin{align}
	| Z_A(l)|& = \Big| \frac 1 N  \sum_{i=1}^N \frac{1}{(N-1)^2} \sum_{j\neq i} \sum_{j'\neq i}\phi_n^*(r_{ij})\psi_l(r_{ij'}) \innerp{\br_{ij}, \br_{ij'}}_{\R^d}  \Big| \nonumber \\
	&\leq \frac{1}{N(N-1)^2} \sum_{i=1}^N \sum_{j\neq i} \sum_{j'\neq i} |\phi_n^*(r_{ij})\psi_l(r_{ij'})|\,. \label{Ineq:Za}
\end{align}
We have shown that $\widebar{\bfA}_{n}^{M}\btheta^*_n -\widebar{\bfA}_{n}^\infty\btheta^*_n$ is the centered empirical mean of i.i.d.~samples of a random vector $Z_A$.  As a result, applying Lemma \ref{lemma:4th_mmt_empirical_mean} below, we obtain
\begin{align}\label{Ineq:Z_A_L^4}
	\E[ \| (\widebar{\bfA}_{n}^{M}-\widebar{\bfA}_{n}^\infty)\btheta^*_n \|_{\R^n}^4 ]&\leq \frac{96 n }{M^2}  \sum_{l=1}^n \E [|Z_A(l)|^4] \,.
\end{align}

To bound $\E [|Z_A(l)|^4]$, we consider two cases for different ranges of $\beta$.
\begin{enumerate}
\item[(i)] When $\beta>1/2$ and uniformly bounded basis condition in Assumption \ref{assum:bd-eigenfn} holds. Then, one can bound \eqref{Ineq:Za} as
	\begin{align*}
		\frac{1}{N(N-1)^2} \sum_{i=1}^N \sum_{j\neq i} \sum_{j'\neq i} |\phi_n^*(r_{ij})\psi_l(r_{ij'})|\leq \|\phi_n^*\|_\infty \sup_{l\geq 1}\|\psi_l\|_\infty \leq C_{\max} \|\phi_n^*\|_\infty\,.
	\end{align*}
	Also, note that
	\begin{align*}
		\|\phi_n^*\|_\infty &\leq C_{\max} \sum_{k=1}^n|\theta_k| \leq C_{\max} \left(\sum_{k=1}^n k^{2\beta}|\theta_k|^2 \right)^{1/2} \left(\sum_{k=1}^nk^{-2\beta}\right)^{1/2} \leq C_{\max}  C_{\beta}L{/\pi^\beta}
	\end{align*}
	with $C_{\beta}=(\sum\nolimits_{k=1}^{\infty}k^{-2\beta})^{1/2}\leq \frac{2\beta}{2\beta-1}$. Thus, we get $| Z_A(l)| \leq C_{\max}^2  C_{\beta}L{/\pi^\beta}$. Plugging it back to \eqref{Ineq:Z_A_L^4}, we have
	\begin{align*}
		\E[ \| (\widebar{\bfA}_{n}^{M}-\widebar{\bfA}_{n}^\infty)\btheta^*_n \|_{\R^n}^4 ]\leq \frac{96 n^2}{M^2}    C_{\beta}^4 C_{\max}^8 L^4{/\pi^{4\beta}}\,.
	\end{align*}
	Taking the square root,  we obtain \eqref{Ineq:4MofA_ap} with $C_A=\sqrt{96} C_{\beta}^2C_{\max}^4 L^2{/\pi^{2\beta}}$.
	
\item[(ii)] When $\beta\geq 1/4$ and uniformly bounded basis condition in Assumption \ref{assum:bd-eigenfn} and the Sobolev embedding condition in Assumption \ref{assum:SoboEmb} hold. We have	
\begin{align*}
	| Z_A(l)|& = \Big| \frac 1 N  \sum_{i=1}^N \frac{1}{(N-1)^2} \sum_{j\neq i} \sum_{j'\neq i}\phi_n^*(r_{ij})\psi_l(r_{ij'}) \innerp{\br_{ij}, \br_{ij'}}_{\R^d}  \Big| \\
	 &\leq   \frac{1}{N(N-1)}  \sum_{i=1}^N  \sum_{j\neq i}\Big| \phi_n^*(r_{ij})   \Big| \sup_{l\geq 1}\|\psi_l\|_\infty \leq \frac{C_{\max}}{N(N-1)}  \sum_{i=1}^N  \sum_{j\neq i} | \phi_n^*(r_{ij})|\,.
\end{align*}
Applying Jensen's inequality and the Sobolev embedding condition in Assumption \ref{assum:SoboEmb}, we get
\begin{align*}
	\E [\|Z_A\|_{\ell^4}^4]&=\sum_{l=1}^n \E [|Z_A(l)|^4] \leq \sum_{l=1}^n \E\bigg[\Big| \frac{C_{\max}}{N(N-1)}  \sum_{i=1}^N  \sum_{j\neq i} | \phi_n^*(r_{ij})|\Big|^4 \bigg] \\
	&\leq \frac{C_{\max}^4 n}{N(N-1)} \sum_{i=1}^N  \sum_{j\neq i} \E[| \phi_n^*(r_{ij})|^4] \leq C_{\max}^4 n \|\phi_n^*\|_{L_{\rho}^4}^4 \leq \calK_{\beta}C_{\max}^4  nL^4/{\pi^{4\beta}}\,.
\end{align*}
Thus, Eq. \eqref{Ineq:Z_A_L^4} implies
\begin{align*}
	\E[ \| (\widebar{\bfA}_{n}^{M}-\widebar{\bfA}_{n}^\infty)\btheta^*_n \|_{\R^n}^4 ]\leq  \frac{96 n^2}{M^2} \calK_{\beta} C_{\max}^4 L^4  /{\pi^{4\beta}}\,.
\end{align*}
Taking the square root,  we obtain \eqref{Ineq:4MofA_ap} with $C_A=\sqrt{96\calK_{\beta}}C_{\max}^2 L^2 /{\pi^{2\beta}}$.
\end{enumerate}

{\bf Part II: the normal vector term. }
The proof for the bound in \eqref{Ineq:4Mofb_ap} is similar. 
By definition, the normal vector $\widebar{\bfb}_{n}^{M} =\frac{1}{M}\sum_{m=1}^M \bfb_n^m$ is the average of $M$ samples $\{\bfb_n^m\}_{m=1}^M$ of the $\R^n$-value random vector $\bfb_{n}$ with entries 
\[
\bfb_{n}(l)= \frac{1}{N} \innerp{R_{\psi_l}[X], R_{\phi_{*}}[X] + \boldeta}_{\R^{Nd}  }, \quad 1\leq l\leq n. 
\]
To show that $\bfb_n$ has a bounded fourth-moment, we decompose it into a bounded part and an unbounded part, $\bfb_{n} = \xi  + \tilde\eta$, where  
\begin{align*}	
\xi(l) = \frac{1}{N} \innerp{R_{\psi_l}[X], R_{\phi_{*}}[X]}_{\R^{Nd}}, \quad 
\widetilde{\eta} (l) = \frac{1}{N} \innerp{R_{\psi_l}[X], \boldeta}_{\R^{Nd}}. 
\end{align*}  
To bound the noise term, we use the Cauchy-Schwarz inequality,
\begin{align*}
\E[ | \widetilde{\eta} (l)|^4 ]
&= 
\frac{1}{N^4}\E[|\innerp{R_{\psi_l}[X], \boldeta}_{\R^{Nd}}|^4]
\leq \frac{1}{N^4}\E[\|R_{\psi_l}[X]\|^4_{\R^{Nd}} \|\boldeta\|^4_{\R^{Nd}}] 
\le \frac{1}{N^2}C_{\max}^4C_{\eta},
\end{align*}
where the first inequality follows from the Cauchy-Schwarz inequality, and the last inequality follows from the assumption that the fourth moment of $\boldeta$ is bounded by some constant $C_{\eta}>0$ and that $\|R_{\psi_l}[X]\|^4_{\R^{Nd}}\le N^2 C_{\max}^4$ for all $X$.

We obtain the bounds for the fourth moment of the random vector $\xi$ by the boundedness of the basis functions in Assumption \ref{assum:bd-eigenfn} and the Sobolev embedding condition in Assumption \ref{assum:SoboEmb}. The argument is the same as for $Z_A(l)$ if we replace $\phi_n^*$ in Eq. \eqref{eq:Za} by $\phi^*$. Again, we split the rest of the proof into two parts with regard to the range of $\beta$.
\begin{enumerate}
	\item[(i)] If $\beta>1/2$, assume that the uniformly bounded basis condition in Assumption \ref{assum:bd-eigenfn} holds. Then, one has
	\begin{align*}	
	| \xi(l)| & \leq \|\phi^*\|_\infty  \sup_{k\geq 1}\|\psi_k\|_\infty^2 \leq  C_{\max} \|\phi^*\|_\infty \,. 
	\end{align*}
Also, note that $\|\phi_n^*\|_\infty \leq C_{\max}  C_{\beta}L/{\pi^{\beta}}$. Thus, we have $| \xi(l)|\leq C_{\beta}C_{\max}^2L/{\pi^{\beta}}$. Combining these bounds, we have, for $1\leq l\leq n$, 
\begin{align*}
	\E[|\bfb_n(l)|^4] &\leq \E\Big[|\xi(l) +  \widetilde{\eta} (l)|^4\Big]\leq 2^3 \E\Big[ |\xi(l)|^4  + |\widetilde{\eta} (l)|^4\Big] \\ 
	&\leq 2^3  \Big[ C_{\beta}^4 C_{\max}^8 L^4 /{\pi^{4\beta}}  + \frac{1}{N^2}C_{\max}^4C_{\eta}\Big] \leq 2^3 C_{\max}^4 \Big[C_{\beta}^4C_{\max}^4 L^4/{\pi^{4\beta}} + \frac{1}{N^2}C_{\eta} \Big]. 
\end{align*}
Consequently, applying the following Lemma \ref{lemma:4th_mmt_empirical_mean} with $Z_m(k) = \frac{1}{N}\innerp{R_{\psi_k}[X^{m}], Y^{m}}_{\R^{Nd}}$, we obtain
\begin{align*}
	\E[\|\widebar{\bfb}_{n}^{M}-\widebar{\bfb}_{n}^\infty\|_{\R^n}^4]&\leq \frac{768 n^2 }{M^2} C_{\max}^4 \left[C_{\beta}^4 C_{\max}^4 L^4/{\pi^{4\beta}}  + \frac{1}{N^2}C_{\eta} \right] 
	\leq C_b^2\frac{n^2}{M^2} 
\end{align*}
with $C_b = \sqrt{768} C_{\max}^2(C_{\beta}^4C_{\max}^4L^4/{\pi^{4\beta}} +\frac{1}{N^2}C_{\eta} )^{1/2}$ as shown in \eqref{Ineq:4Mofb_ap}.
\item[(ii)] If $\beta\geq 1/4$, assume that the uniformly bounded basis condition in Assumption \ref{assum:bd-eigenfn} and the Sobolev embedding condition in Assumption \ref{assum:SoboEmb} hold.  
	Employing the condition \eqref{Assu:suff_L4}, we have the following estimates:
\begin{align*}
	| \xi(l)|\leq \frac{C_{\max}}{N(N-1)}  \sum_{i=1}^N  \sum_{j\neq i} | \phi^*(r_{ij})\,, \text{ and }\quad \E[| \xi(l)|^4] \leq \calK_{\beta}C_{\max}^4 L^4/{\pi^{4\beta}}\,.
\end{align*}
Combining these bounds, we have, for $1\leq l\leq n$, 
\begin{align*}
	\E[|\bfb_n(l)|^4] &\leq \E[|\xi(l) +  \widetilde{\eta} (l)|^4]\leq 2^3 \E\Big[ |\xi(l)|^4  + |\widetilde{\eta} (l)|^4\Big] \\ 
	&\leq 2^3 \Big[\calK_{\beta}L^4 C_{\max}^4  + \frac{1}{N^2}C_{\max}^4C_{\eta}\Big] \leq 2^3 C_{\max}^4 \left[\calK_{\beta}L^4/{\pi^{4\beta}}  + \frac{1}{N^2}C_{\eta} \right]. 
\end{align*}
Consequently, applying the following Lemma \ref{lemma:4th_mmt_empirical_mean} with $Z_m(l) = \frac{1}{N}\innerp{R_{\psi_l}[X^{m}], Y^{m}}_{\R^{Nd}}$, we obtain
\begin{align*}
	\E[\|\widebar{\bfb}_{n}^{M}-\widebar{\bfb}_{n}^\infty\|_{\R^n}^4]&\leq \frac{768 n^2}{M^2} C_{\max}^4 \left[{\calK_{\beta}}L^4/{\pi^{4\beta}}   + \frac{1}{N^2}C_{\eta} \right] 
	\leq C_b^2\frac{n^2}{M^2} 
\end{align*}
with $C_b = \sqrt{768} C_{\max}^2({\calK_{\beta}}L^4/{\pi^{4\beta}} +\frac{1}{N^2}C_{\eta} )^{1/2}$ as shown in \eqref{Ineq:4Mofb_ap}.
\end{enumerate}

In conclusion, we complete the proof of Lemma \ref{lemma:mmtBd_v2}.
\end{proof}
\smallskip

The next lemma provides bounds for the fourth moment of the empirical mean of i.i.d. samples. The proof follows from applying the independence between the samples and the direct expansion of the fourth power of the sum.   
\begin{lemma}[Fourth-moment bounds of empirical mean]\label{lemma:4th_mmt_empirical_mean}
Let $\{Z_{m}\}_{m=1}^M$ be i.i.d. samples of the $R^n$-valued random variable $Z = ( Z(1),\ldots,Z(n))$  with bounded fourth moment.  Then we have  
\[
\E\bigg[\Big\| \frac{1}{M}\sum_{m=1}^M (Z_{m} - \E[Z]) \Big\|^4_{\R^{n}} \bigg] \leq \frac{6n}{M^2} \sum_{k=1}^n \E[|Z(k)-\E[Z(k)]|^4]\leq \frac{96 n}{M^2}   \sum_{k=1}^n \E [|Z(k)|^4] \,. 
\]	
\end{lemma}
\begin{proof}
The second inequality follows directly from 
\[
 \E [| Z(k) - \E[Z(k) ]|^4] \leq 2^3 \left( \E [|Z(k) |^4] + \E [|\E[Z(k) ]|^4]\right) \leq 2^4 \E [|Z(k) |^4]
\]
for each $1\leq k\leq n$ since $\E [|\E[Z(k) ]|^4]\leq \E[|Z(k) |^4]$ by Jensen's inequality. 
 
To prove the first inequality, it suffices to consider $\E[Z]=0$ and prove  
 \begin{equation}\label{eq:4th_mmt_mean0}
     \E\left[\bigg\| \frac{1}{M}\sum_{m=1}^M Z_{m} \bigg\|^4_{\R^{n}}\right] \leq \frac{6n}{M^2}  \sum_{k=1}^n \E[|Z(k)|^4]. 
 \end{equation}
 
We first prove the case with $n=1$, then extend it to the case with $n>1$. 

{\bf Case $n=1$: $Z$ is a 1-dimensional random variable.} Note that 
  	\begin{align}
	\Big|\sum_{m=1}^M Z_m\Big|^4&=\sum_{m_1,\cdots,m_4=1}^M \prod_{i=1}^4 Z_{m_i}\nonumber \\
	&=\sum_{m=1}^M Z_{m}^4+4\sum_{\substack{m_1,m_2=1 \\ m_1\neq m_2} }^M Z_{m_1} Z_{m_2}^3 +6\sum_{\substack{m_1,m_2=1 \\ m_1\neq m_2} }^M Z_{m_1}^2 Z_{m_2}^2 \nonumber \\
	&+6\sum_{\substack{m_1,m_2,m_3=1 \\ m_1\neq m_2\neq m_3} }^M Z_{m_1}^2 Z_{m_2} Z_{m_3} 
	+\sum_{\substack{m_1,m_2,m_3,m_4=1 \\ m_1\neq m_2\neq m_3\neq m_4} }^M Z_{m_1} Z_{m_2} Z_{m_3} Z_{m_4}\,. \nonumber
\end{align}
Meanwhile, the independence between these mean zero samples implies that $\E[Z_{m_1} Z_{m_2}^3]=0$, $\E [Z_{m_1}^2 Z_{m_2} Z_{m_3}]=0$, and $\E[Z_{m_1} Z_{m_2} Z_{m_3} Z_{m_4}]=0$, for any mutually different indices $1\leq m_1,m_2,m_3,m_4\leq M$. Then, the desired inequality in \eqref{eq:4th_mmt_mean0} with $n=1$ follows from
\begin{align*}
\E\left[ \Big|\sum_{m=1}^M Z_m\Big|^4\right] & = \E \left[\sum_{m=1}^M Z_{m}^4 + 6\sum_{\substack{m_1,m_2=1 \\ m_1\neq m_2} }^M Z_{m_1}^2 Z_{m_2}^2 \right] \\
& = M \E[|Z|^4] + 6M(M-1) (\E[|Z^2|])^2 	\leq  6M^2  \E[|Z|^4], 
\end{align*} 
where the second equality follows from that $\{Z_m\}_{m=1}^M$ being i.i.d. samples of $Z$, and the last inequality follows from $ (\E[|Z^2|])^2\leq \E[|Z|^4]$ by Jensen's inequality.  

{\bf Case $n>1$: $Z$ is a random vector.} We prove it by applying the above bound to each component of the vector. Note that 
\[
\Big\|\frac{1}{M}\sum_{m=1}^M Z_m\Big\|^4_{\R^{n}} = \left( \sum_{k=1}^n\Big|\frac{1}{M}\sum_{m=1}^M  Z_m(k)\Big|^2\right)^2 \leq n  \sum_{k=1}^n\Big|\frac{1}{M}\sum_{m=1}^M  Z_m(k)\Big|^4.
\]
Meanwhile, applying the result in Case $n=1$ to each component $Z_m(k)$, we have 
\[
  \E\left[ \bigg| \frac{1}{M}\sum_{m=1}^M Z_{m}(k) \bigg|^4 \right] \leq \frac{6}{M^2}  \E[|Z(k)|^4], \quad \forall 1\leq k\leq n. 
\]
Combining the two inequalities, we obtain the inequality \eqref{eq:4th_mmt_mean0}.
\end{proof}

\subsection{Matrix concentration inequalities}\label{sec:Mat_Concen}
We recall two concentration inequalities in the theory of random matrix and Weyl's inequality, which can be found in \cite{Vershynin2018,Tropp2015}. They are used to prove the left tail probability estimates in Section {\rm \ref{sec:upperBd}}.  

\begin{theorem}[Matrix Chernoff inequality]\label{thm:matChernoff}
	Let $\{X_i\}_{i=1}^M\subset \R^{n\times n}$ be independent random Hermitian matrices. Assume that $0\leq \lambda_{\min}(X_i)$ and $\lambda_{\max}(X_i)\leq K$ almost surely for all $i$. Define the minimum eigenvalue $\mu_{\min}=\lambda_{\min}\Big(\sum_{i=1}^M \E[X_i] \Big)$.	Then,  we have 
	\[
	\P\left(\lambda_{\min}\Big(\sum_{i=1}^M X_i\Big)\leq (1-t)\mu_{\min} \right) \leq n \bigg[ \frac{e^{-t}}{(1-t)^{1-t}} \bigg]^{\mu_{\min}/K}, 
	\]
	for every $t\in[0,1)$. 
\end{theorem}

\begin{theorem}[Matrix Bernstein inequality]\label{thm:matBerstein}
	Let $\{X_i\}_{i=1}^M\subset \R^{n\times n}$ be independent mean zero symmetric random matrices such that $\|X_i\|_{\textup{op}}\leq K$ almost surely for all $i$. Then, for every $t\geq 0$, we have 
	\[
	\P\left(\Big\|\sum_{i=1}^M X_i\Big\|_{\textup{op}}\geq t \right) \leq 2n \exp\left( -\frac{t^2/2}{\sigma^2+Kt/3}\right), 
	\]
	where $\sigma^2= \|\sum_{i=1}^M \E[X_i^2]\|_{\textup{op}}$ is the matrix variance norm of the sum and $\|\cdot\|_{\textup{op}}$ represents the operator norm of a matrix. 
\end{theorem}

\begin{theorem}[Weyl's inequality]\label{thm:weyl}
For any symmetric matrices $S\in \R^{M\times M}$ and $T\in \R^{M\times M}$, we have 
\[\max_{1\leq i\leq M} |\lambda_i(S) - \lambda_i(T)|\leq \|S-T\|_{\textup{op}},  
\] 
where $\lambda_i(S)$ is the $i$-th eigenvalue of $S$ in descending order. 
\end{theorem}

Next, we include the proof of Bernstein left tail probability of the smallest eigenvalue. 

\begin{proof}[Proof of Lemma \ref{thm:min-eigen_B}]
	The proof follows from the matrix Bernstein inequality \cite{Vershynin2018,Tropp2015}, which we recall in Theorem \ref{thm:matBerstein}. Note that Lemma \ref{lemma:projEst} implies
	\begin{equation}\label{Eq:Coer_Mat}
		\lambda_{\min}(\bar{\bfA}_{n}^{\infty})=\inf_{\btheta\in S^{n-1}}\btheta^\top \bar{\bfA}_{n}^{\infty}\btheta= \inf_{\btheta\in S^{n-1}} \frac{1}{N}\E[\|R_{\phi_{\btheta}}[X]\|_{\R^{Nd}}^2]\geq c_{\bar \calL}>0\,.
	\end{equation} 	
	We denote $\bPhi^m =(R_{\psi_1}[X^{m}],\ldots,R_{\psi_n}[X^{m}])$ for each sample $X^m$ and thus $\bar{\bfA}_{n}^{M}=\frac{1}{MN}\sum_{m=1}^M [\bPhi^m]^\top \bPhi^m$. Also, we define
	\begin{align*}
		\widebar{Q}_{M,N}=\bar{\bfA}_{n}^{M}-\bar{\bfA}_{n}^{\infty}=\frac{1}{MN}\sum\nolimits_{m=1}^M \Big[ [\bPhi^m]^\top \bPhi^m-\E[[\bPhi^m]^\top \bPhi^m] \Big] \,, 
	\end{align*}
	where $\{Q^{m}=\frac 1N[\bPhi^m]^\top \bPhi^m-\frac 1N\E[[\bPhi^m]^\top \bPhi^m]\}_{m=1}^M$ form a sequence of mean zero independent matrices. 
 Note that $\lambda_{\max}(Q^m)=\|Q^m\|_{\textup{op}}\leq 2nC_{\max}^2$ and $\sigma^2=\|\sum_{m=1}^M \E[(Q^m)^2]\|_{\textup{op}}\leq 2(n C_{\max}^2)^2$. Then the matrix Bernstein inequality (Theorem \ref{thm:matBerstein}) gives that
	\begin{align*}
		\P\{\norm{\widebar{Q}_{M,N}}_{\textup{op}}\geq t  \}\leq 2n\exp\rbracket{-\frac{M t^2/4}{(n C_{\max}^2)^2+nC_{\max}^2t /3}}\,,
	\end{align*}
	for any $t \leq c_{\bar\calL}$. So, by \eqref{Eq:Coer_Mat} and then the Weyl's  inequality (Theorem \ref{thm:weyl}) we have
	\begin{align*}
		\P\left\{\lambda_{\min}(\bar{\bfA}_{n}^{M}) \leq c_{\bar\calL}-\varepsilon c_{\bar\calL} \right\} &\leq \P\left\{|\lambda_{\min}(\bar{\bfA}_{n}^{M})- \lambda_{\min}(\bar{\bfA}_{n}^{\infty})|\geq \varepsilon c_{\bar\calL}\right\} \leq \P\left\{\norm{\widebar{Q}_{M,N}}_{\textup{op}}\geq \varepsilon c_{\bar\calL}\right\}\,.
	\end{align*}
	Thus, we have 
	\begin{equation}\label{Eq:Low_Min_est_B}
		\P\left\{ \lambda_{\min}(\bar{\bfA}_{n}^{M}) \leq (1-\varepsilon) c_{\bar\calL} \right\}\leq 2n \exp\rbracket{-\frac{M\varepsilon^2c_{\bar\calL}^2/4}{(n C_{\max}^2)^2+nC_{\max}^2\varepsilon c_{\bar\calL}/3}}\,,
	\end{equation}
	for any $\varepsilon\in (0,1)$. The inequality \eqref{Eq:Low_Min_est_B2} follows by taking $\varepsilon=\frac{3}{4}$ in Eq. \eqref{Eq:Low_Min_est_B}.
\end{proof}

\smallskip

Compared to the Chernoff left-tail bound, the Bernstein left-tail bound is weaker for our purposes, as it requires controlling $\sigma^2=\left\| \sum_{m=1}^M \mathbb{E}[(Q^m)^2] \right\|_{\textup{op}}$ rather than simply bounding $\lambda_{\max}(Q^m) = \|Q^m\|_{\textup{op}}$.
\subsection{Proof for the PAC-Bayesian left tail probability}\label{sec:proof_left_tail_prob2}
This section includes the technical proofs involved with the PAC-Bayesian left tail probability in Section \ref{sec:prob_eig_bd}. We present a technical proof of the Lemma \ref{thm:min-eigen} and some auxiliary lemmas under uniform boundedness Assumption \ref{assum:bd-eigenfn} on the basis functions and the fourth-moment Assumption \ref{assum:L2L4}.

Specifically, we want to show that the left tail probabilities of the smallest eigenvalue: for any $\varepsilon\in(0,1)$
 \begin{equation}\label{Eq:Low_Min_est_ap}
     \P\left\{\lambda_{\min}(\bar{\bfA}_{n}^{M}) \leq \frac{1-\varepsilon}{2}c_{\bar\calL} \right\} \leq  \exp\rbracket{n\log\left(\frac{5C_{\max}^2}{c_{\bar\calL}}\right)-{\frac{\varepsilon^2 M c_{\bar\calL}^2}{16 \kappa N^2}}} \,,
 \end{equation}
 where {$M\geq \frac{16 \kappa N^2}{c_{\bar\calL}^2}\log\left(\frac{5C_{\max}^2}{c_{\bar\calL}}\right) \cdot \frac{n}{\varepsilon^2}$} and $n\geq 2$. In particular, letting $\varepsilon=\frac{1}{2}$, we have
 \begin{equation}\label{Eq:Low_Min_est2_ap}
     \P\left\{\lambda_{\min}(\bar{\bfA}_{n}^{M}) \leq \frac{c_{\bar\calL}}{4} \right\} \leq  \exp\rbracket{n\log\left(\frac{5C_{\max}^2}{c_{\bar\calL}}\right)-{\frac{ M c_{\bar\calL}^2}{64 \kappa N^2}}} \,.
 \end{equation}

We split the proof into three steps:
\begin{enumerate}[leftmargin=1.5cm]
	\item[Step 1:] Construct from $\btheta^\top \bar{\bfA}_{n}^{M}\btheta
=\frac{1}{MN}\sum_{m=1}^M \|R_{\phi_{\btheta}}[X^m]\|_{\R^{Nd}}^2$ an empirical process with uniformly bounded moment generating function, and apply the PAC-Bayes inequality that we recall in Lemma \ref{lemma:PAC_Bay}.  
	\item[Step 2:] Obtain a parametric lower bound for $\lambda_{\min}(\bar{\bfA}_{n}^{M})$ via controls of the approximation and entropy terms in the PAC-Bayes inequality. 
	\item[Step 3:] Select the parameter properly to achieve the desired bound for the probability of the smallest eigenvalue being below the threshold. 
\end{enumerate}

We begin by introducing our primary tool to prove Lemma \ref{thm:min-eigen}, the PAC-Bayes inequality (see, e.g., \cite{Zhivotovskiy20,Mourtada2022,Oliveira2016}).
\begin{lemma}[PAC-Bayes inequality]\label{lemma:PAC_Bay}
	Let $\Theta$ be a measurable space, and $\{Z(\theta):\theta\in \Theta\}$ be a real-valued measurable process. Assume that 
	\begin{equation}\label{Eq:PAC_Assu}
		\E[\exp(Z(\theta))]\leq 1\,,\quad \text{for every } \theta\in\Theta\,.
	\end{equation}
	Let $\pi$ be a (prior) probability measure on $\Theta$. Then, for $t>0$
	\begin{align}\label{Eq:PAC_Ineq}
		\P\left\{\forall \mu, \int_{\Theta} Z(\theta) \mu(\theta)\leq \textup{KL}(\mu,\pi)+t\right\}\geq 1-e^{-t}\,,
	\end{align}
	where $\mu$ spans all (posterior) probability measures on $\Theta$, and $\textup{KL}(\mu,\pi)$ is the Kullback-Leibler divergence between $\mu$ and $\pi$:
	\begin{align*}
		\textup{KL}(\mu,\pi):=\begin{cases}
			\int_{\Theta} \log\Big[\frac{d\mu}{d\pi}\Big] d\mu & \text{if }\mu \ll \pi \,; \\
			\infty & \text{otherwise}\,.
		\end{cases}
	\end{align*}
\end{lemma}

The next lemma, from Section 2.3 in the supplement of \cite{Mourtada2022}, controls the approximate term in the application of PAC-inequality. Here, we present an alternative constructive proof.
\begin{lemma}\label{Lem:F_Sigma} 
	For every $\gamma\in(0,1/2]$, $v\in S^{n-1}$, i.e., $\|v\|_{\R^n}=1$, define 
	\begin{equation}\label{eq:prob_measures}
		\Theta_{v,\gamma}:=\{\theta\in S^{n-1}: \|\theta-v\|_{\R^n} \leq \gamma \}, \quad \text{  and } 	\pi_{v,\gamma}(d\theta):=\frac{\bf1_{\Theta_{v,\gamma}}(\theta)}{\pi(\Theta_{v,\gamma})}\pi(d\theta),
	\end{equation}
	where $\pi$ is a uniform measure on the sphere. That is, $\Theta_{v,\gamma}$ is a ``spherical cap'' or ``contact lens'' in $n$-th dimension space, and $\pi_{v,\gamma}$ is a uniform surface measure on the spherical cap. Then, 
\begin{align}\label{Eq:F_Sigma}
	F_{v,\gamma}(\Sigma):= \int_{S^{n-1}}\innerp{\Sigma \theta, \theta}_{\R^n} \pi_{v,\gamma}(d\theta)=[1-\gphi(\gamma)]\innerp{\Sigma v, v}_{\R^n}+\gphi(\gamma) \frac{\textup{Tr}(\Sigma)}{n}, 
\end{align} 
for any symmetric matrix $\Sigma\in \R^{n\times n}$, where 
\begin{equation}\label{Eq:phi_gam}
	\gphi(\gamma)= \frac{n}{n-1}\int_{S^{n-1}}[1-\innerp{\theta,v}_{\R^n}^2]\pi_{v,\gamma}(d\theta) \in \bigg[0,\frac{n \gamma^2}{(n-1)}\bigg]\,.
\end{equation}
\end{lemma}

\begin{proof}[Proof of Lemma \ref{Lem:F_Sigma}]
	Note that
\begin{align*}
	F_{v,\gamma}(\Sigma)&= \int_{\Theta}\innerp{\Sigma \theta, \theta}_{\R^n} \pi_{v,\gamma}(d\theta) = \int_{\Theta} \textup{Tr} [\theta^\top \Sigma\theta] \pi_{v,\gamma}(d\theta) \\
	&=\textup{Tr}[\Sigma A_{v,\gamma}]:=\textup{Tr}\sbracket{\Sigma \int_{\Theta} \theta \theta^\top \pi_{v,\gamma}(d\theta)}\,.
\end{align*}
To conclude \eqref{Eq:F_Sigma}, we proceed to show that
\begin{align}\label{Eq:A_{v,gam}}
	A_{v,\gamma} = [1-\gphi(\gamma)]vv^\top + \gphi(\gamma)\frac{I_n}{n}\,.
\end{align}

By isometric invariance, we set without loss of generality 
\begin{align*}
	v=e_1=(1,0,\cdots,0)\in\R^n\,.
\end{align*}
So we get a $\R^n$--``spherical cap'' $\Theta_\gamma=\Theta(e_1,\gamma)$ centered at $v=e_1$. The notations $A_{e_1,\gamma}$ and $\pi_{e_1,\gamma}(d\theta)$ are abbreviated as $A_{\gamma}$ and $\pi_{\gamma}(d\theta)$. Thus, for $\theta=(\theta_1,\theta_2,\cdots,\theta_n)\in \Theta_\gamma\subseteq S^{n-1}$ we have
\begin{align*}
	A_{\gamma}&=\int_{\Theta_\gamma} \theta \theta^\top \pi_{\gamma}(d\theta)=\int_{\Theta_\gamma} \begin{bmatrix}
		\theta_1^2 & \theta_1 \theta_2 &\cdots &\theta_1\theta_n \\
		\theta_2 \theta_1 & \theta_2^2 &\cdots &\theta_2\theta_n \\
		\vdots & \vdots & \vdots & \vdots \\
		\theta_n \theta_1 & \theta_n \theta_2 &\cdots &\theta_n^2
	\end{bmatrix}
	\pi_{\gamma}(d\theta)\\
	&=\text{diag}\left[\int_{\Theta_\gamma} \theta_1^2\pi_{\gamma}(d\theta), \int_{\Theta_\gamma} \theta_2^2\pi_{\gamma}(d\theta),\cdots,\int_{\Theta_\gamma} \theta_2^2\pi_{\gamma}(d\theta) \right]
\end{align*}
since $\int_{\Theta_\gamma}\theta_2^2\pi_{\gamma}(d\theta)=\cdots=\int_{\Theta_\gamma}\theta_n^2\pi_{\gamma}(d\theta)$ and $\int_{\Theta_\gamma} \theta_i \theta_j \pi_{\gamma}(d\theta)=0$ if $i\neq j$. Moreover, it is readily seen that
\begin{align*}
	1=&\int_{\Theta_\gamma} \|\theta\|^2 \pi_{\gamma}(d\theta)=\int_{\Theta_\gamma} [\theta_1^2+\theta_2^2+\cdots+\theta_n^2] \pi_{\gamma}(d\theta) \\
	=&\int_{\Theta_\gamma} \theta_1^2 \pi_{\gamma}(d\theta) +(n-1)\int_{\Theta_\gamma} \theta_2^2 \pi_{\gamma}(d\theta)\,,
\end{align*}
and consequently
\begin{align*}
	\int_{\Theta_\gamma}\theta_2^2\pi_{\gamma}(d\theta)=\cdots=\int_{\Theta_\gamma}\theta_n^2\pi_{\gamma}(d\theta)=\frac{1}{n-1}\sbracket{1-\int_{\Theta_\gamma} \theta_1^2 \pi_{\gamma}(d\theta)}=\frac{\gphi(\gamma)}{n}\,.
\end{align*}
Hence, we have
\begin{align}\label{Eq:A_gam}
	A_{\gamma}=\text{diag}\left[\int_{\Theta_\gamma} \theta_1^2 \pi_{\gamma}(d\theta), \frac{\gphi(\gamma)}{n},\cdots,\frac{\gphi(\gamma)}{n}\right]\,.
\end{align}
Noticing that $\gphi(\gamma)=\frac{n}{n-1}\sbracket{1-\int_{\Theta_\gamma} \theta_1^2 \pi_{\gamma}(d\theta)}$ and   $(1-\gphi(\gamma))+\frac{\gphi(\gamma)}{n}=\int_{\Theta_\gamma} \theta_1^2 \pi_{\gamma}(d\theta)$, the right-hand side of \eqref{Eq:A_{v,gam}} can be written as 
\begin{align*}
	(1-\gphi(\gamma))vv^\top + \gphi(\gamma)\frac{I_n}{n}=\text{diag}\left[\int_{\Theta_\gamma} \theta_1^2 \pi_{\gamma}(d\theta), \frac{\gphi(\gamma)}{n},\cdots,\frac{\gphi(\gamma)}{n}\right]\,,
\end{align*}
which matches \eqref{Eq:A_gam}. 

The bound of $\gphi(\gamma)$ in \eqref{Eq:phi_gam} can follow the same argument as in Section 2.3 in the supplement of \cite{Mourtada2022}. This completes the proof.
\end{proof}

\bigskip
We introduce an inequality in \cite[Lemma A.1]{Oliveira2016} to control the generating moment function before the proof of Lemma \ref{thm:min-eigen}.
\begin{lemma}\label{Lem:Oliveira}
	Let $X$ be a nonnegative random variable with a finite second moment. Then for all $\lambda\geq 0$
	\begin{align*}
		\E[e^{-\lambda X}]\leq e^{-\lambda \E[X]+\frac{\lambda^2}{2}\E[X^2]}\,.
	\end{align*}
\end{lemma}
\begin{proof}
    We include the proof for completeness. It is clear that
    \begin{align*}
        \E[e^{-\lambda X}]\leq 1-\lambda \E[X]+\frac{\lambda^2}{2}\E[X^2]\leq e^{-\lambda \E[X]+\frac{\lambda^2}{2}\E[X^2]}
    \end{align*}
    by using $1+y\leq e^y$ in the second inequality.
\end{proof}

\smallskip

\begin{proof}[Proof of Lemma \ref{thm:min-eigen}]
\noindent\textbf{Step 1:} 
For every $\theta \in S^{n-1}$ and $\lambda>0$, the bound for the moment generating function can be derived by Lemma \ref{Lem:Oliveira}:
\begin{align*}
	\E\left[\exp\left(-\lambda \frac{1}{N} \|R_{\phi_{\theta}}[X^m]\|^2_{\R^{Nd}}\right)\right]&\leq \exp\left(-\lambda \frac{1}{N}\E[\|R_{\phi_{\theta}}[X^m]\|^2_{\R^{Nd}}]+{\frac{\lambda^2}{2N^2}\E[\|R_{\phi_{\theta}}[X^m]\|^4_{\R^{Nd}}]}\right) \\
 &\leq \exp\left(-\lambda c_{\bar\calL}+{\frac{\lambda^2}{2N^2}\E[\|R_{\phi_{\theta}}[X^m]\|^4_{\R^{Nd}}]}\right)\,.
\end{align*}
By fourth-moment Assumption \ref{assum:L2L4} and Jensen's inequality
\begin{align*}
	\E[\|R_{\phi_{\theta}}[X^m]\|^4_{\R^{Nd}}]&\leq \kappa\cdot  \left(\E[\|R_{\phi_{\theta}}[X^m]\|^2_{\R^{Nd}}]\right)^2 \leq \kappa\cdot  \left(\E\sbracket{\sum_{i=1}^N \norm{R_{\phi_{\theta}}[X^m]_i}^2_{\R^{d}}} \right)^2\\
	&= \kappa\cdot  \left(\E\sbracket{\sum_{i=1}^N \norm{\frac{1}{N-1}\sum_{j\neq i}\phi_{\theta}(r^m_{ij})\bfr^m_{ij}}^2_{\R^{d}}} \right)^2 \\
	&\leq \kappa\cdot \left(\sum_{i=1}^N \frac{1}{N-1}\sum_{j\neq i}\E\sbracket{\normv{\phi_{\theta}(r^m_{ij})}^2} \right)^2\,.
\end{align*}
Recall that $\phi_{\theta}(r^m_{ij})=\sum_{k=1}^n \theta_k \psi_{k}(r^m_{ij})$, the distribution of random variable $r^m_{ij}$ is $\rho$, and $\{\psi_k\}_{k=1}^{\infty}$ are orthonormal basis functions in $L^2_{\rho}$. Then, we can proceed to get
\begin{align*}
	\frac{1}{N-1}\sum_{j\neq i}\E\sbracket{\normv{\phi_{\theta}(r^m_{ij})}^2}=\frac{1}{N-1}\sum_{j\neq i}\E\sbracket{\normv{\sum_{k=1}^n \theta_k \psi_{k}(r^m_{ij})}^2}=\sum_{k=1}^n \theta_k^2=1\,.
\end{align*}
Therefore, we have 
\begin{equation}\label{Ineq:R^4}
	{\E[\|R_{\phi_{\theta}}[X^m]\|^4_{\R^{Nd}}]\leq \kappa N^2\,.}
\end{equation}
Combining \eqref{Ineq:R^4} and the fact that $\frac{1}{N}\E[\|R_{\phi_{\theta}}[X^m]\|^2_{\R^{Nd}}]\geq c_{\bar\calL}$,  we obtain
\begin{equation}\label{Eq:Exp<1}
	\E\left[\exp\left(- \frac{\lambda}{N} \|R_{\phi_{\theta}}[X^m]\|^2_{\R^{Nd}}+ \lambda c_{\bar\calL}- {\frac{\lambda^2 \kappa}{2}}  \right)\right]\leq 1, \quad \forall \theta \in S^{n-1}, \lambda>0.
\end{equation}

Thus, by the independence of samples, we obtain
\begin{equation*}
    \sup_{\theta \in S^{n-1}}\E\left[\exp\left(-\frac{\lambda}{N} \sum_{m=1}^M \|R_{\phi_{\theta}}[X^m]\|^2_{\R^{Nd}}+ \lambda M c_{\bar\calL}-\frac{\lambda^2}{2} {\kappa M }\right)\right]\leq 1, \quad \forall \lambda>0.
\end{equation*}
In other words, the process 
$$
Z_\lambda(\theta):= -\frac{\lambda}{N} \sum_{m=1}^M \|R_{\phi_{\theta}}[X^m]\|^2_{\R^{Nd}}+ \lambda M c_{\bar\calL}-\frac{\lambda^2}{2} {\kappa M }
$$ 
with $\theta\in S^{n-1}$ has a uniformly bounded moment generating function. Then, applying the PAC-Bayes inequality in Lemma \ref{lemma:PAC_Bay} with $\Theta = S^{n-1}$, we obtain
	\begin{align}\label{Eq:PAC_Ineq2}
		\P\left\{ {\forall~ \mu\in \mathcal{P}}\,, \int_{\Theta} Z_\lambda(\theta) \mu(\theta)\leq \textup{KL}(\mu,\pi)+t\right\}\geq 1-e^{-t},\quad \forall t>0, 
	\end{align}
where $\pi, \mu \in \mathcal{P}$ with $\mathcal{P}$ denoting the set of all probability measures on $\Theta$. In the next step, we will select a specific $\pi$ and a subset of $\mathcal{P}$ in \eqref{Eq:PAC_Ineq2} to obtain a $\lambda$-dependent bound $\P\left\{\lambda_{\min}(\bfA_n^M)<  \frac 1 8 c_{\bar\calL}\right\}$, and we remove the dependence on $\lambda$ in Step 3.

\smallskip 
\noindent\textbf{Step 2:} Obtain a lower bound for $\lambda_{\min}(\bfA_n^M)$ through constructing probability measures $\pi$ and $\mu$ in \eqref{Eq:PAC_Ineq2} to control $\int_{\Theta} Z_\lambda(\theta)\mu(d\theta)$. This lower bound depends on $\lambda$, which will be selected in Step 3 to achieve the desired bound in \eqref{Eq:Low_Min_est_ap}.

Let $\pi$ be a uniform probability measure on $S^{n-1}$. For each $v\in S^{n-1}$ and $\gamma\in(0,1/2]$, define $\Theta_{v,\gamma}$ and probability measures $\pi_{v,\gamma}$ as in \eqref{eq:prob_measures}. 
Then,  the PAC-Bayes inequality \eqref{Eq:PAC_Ineq2} with $\mu(d\theta)=\pi_{v,\gamma}(d\theta)$ implies that 
	\begin{align*}
		\P\left\{\sup_{v\in S^{n-1},\gamma\in(0,1/2] } \int_{\Theta} Z_\lambda(\theta) \pi_{v,\gamma} (\theta)- \textup{KL}(\pi_{v,\gamma},\pi)\leq t\right\}\geq 1-e^{-t}\,. 
	\end{align*}
Meanwhile, note that
 \begin{align*}
    \frac{1}{M}\int_{\Theta} Z_\lambda(\theta) \pi_{v,\gamma} (d\theta) &=\bigg[ -\lambda \int_{\Theta}\innerp{\bar{\bfA}_{n}^{M} \theta,\theta}\pi_{v,\gamma}(d\theta)+ \lambda c_{\bar\calL}-\frac{\lambda^2 \kappa}{2}  \bigg]\\
	&=:[-\lambda F_{v,\gamma}(\bar{\bfA}_{n}^{M})+ \lambda c_{\bar\calL}-\frac{\lambda^2 \kappa}{2}]\,.
\end{align*}  
Hence, the above inequality implies that, with at least probability $1-e^{-Mu}$ with $u=\frac t M$ for simplicity,  
\begin{align}
	& \sup_{v\in S^{n-1},\gamma\in(0,1/2] } \bigg[-\lambda F_{v,\gamma}(\bar{\bfA}_{n}^{M}) \nonumber \\
	&\qquad\qquad\qquad\quad + \lambda c_{\bar\calL}-\frac{\lambda^2}{2} {\kappa N^2}-\frac{1}{M}\textup{KL}(\pi_{v,\gamma},\pi) \bigg] \nonumber  
	 \leq \frac t M \,. \nonumber \\
	\Leftrightarrow\quad& \inf_{v\in S^{n-1},\gamma\in(0,1/2] }  \lambda F_{v,\gamma}(\bar{\bfA}_{n}^{M}) +\frac{1}{M} \textup{KL}(\pi_{v,\gamma},\pi) \geq  \lambda c_{\bar\calL}-\frac{\lambda^2 \kappa}{2}   - u \,.\label{Ineq:Key_F+KL}
\end{align}
Follow the conventions in \cite{Mourtada2022}, we refer $F_{v,\gamma}(\bar{\bfA}_{n}^{M})$ to be the approximation term and $\textup{KL}(\pi_{v,\gamma},\pi)$ the entropy term. The controls of these terms follow from the above selection of measure $\pi_{v,\gamma}$ and $\pi$.

The control of approximate term follows from applying Lemma \ref{Lem:F_Sigma} with $\Sigma = \bar{\bfA}_{n}^{M}$: 
\[
F_{v,\gamma}(\bar{\bfA}_{n}^{M})=[1-\gphi(\gamma)]\innerp{\bar{\bfA}_{n}^{M} v, v}+\gphi(\gamma) \frac{\textup{Tr}(\bar{\bfA}_{n}^{M})}{n} 
\]
with $\gphi(\gamma)$ in \eqref{Eq:phi_gam}. The control of the entropy term is from Section 2.4 in the supplement of \cite{Mourtada2022}). Specifically, we have for every $v\in S^{n-1}$ and $\gamma>0$, 
\begin{align}
	\textup{KL}(\pi_{v,\gamma},\pi) &= \int_{\Theta} \log\rbracket{\frac{d\pi_{v,\gamma}}{d\pi} (\theta)} \pi_{v,\gamma}(d\theta)= \int_{\Theta} \log\sbracket{\frac{1}{\pi(\Theta_{v,\gamma})}}\pi_{v,\gamma}(d\theta) \nonumber \\
	&= \log\sbracket{\frac{1}{\pi(\Theta_{v,\gamma})}}\leq n\log(1+2/\gamma)\,, \nonumber 
\end{align}
where the bound for surface area $\pi(\Theta_{v,\gamma})$ is from \cite[Lemma 4.2.13]{Vershynin2018}. 

Plugging these two estimates into \eqref{Ineq:Key_F+KL}, we obtain with at least probability $1-e^{-Mu}$, for all $v\in S^{n-1}, \gamma\in(0,1/2]$, 
\begin{align*}
	& \lambda(1-\gphi(\gamma))\innerp{\bar{\bfA}_{n}^{M} v, v} +  \lambda\gphi(\gamma) \frac{\text{Tr}(\bar{\bfA}_{n}^{M})}{n}+  \frac{n}{M}\log(1+2/\gamma) \geq  \lambda c_{\bar\calL}-\frac{\lambda^2 \kappa }{2}  - u, \,
\end{align*} 
which amounts to
\begin{align}
	\innerp{\bar{\bfA}_{n}^{M} v, v}
	&\geq \frac{1}{\lambda(1-\gphi(\gamma))}\bigg[\lambda c_{\bar\calL}-\frac{\lambda^2 \kappa}{2}\nonumber \\
	&\quad - \frac{n}{M}\log(1+2/\gamma)-u\bigg]-\frac{\gphi(\gamma)}{1-\gphi(\gamma)} \frac{\text{Tr}(\bar{\bfA}_{n}^{M})}{n}\,. \label{Ineq:LowBD_Abar}
\end{align}
Also, the uniform boundedness of $\{\psi_k\}$ in Assumption \ref{assum:bd-eigenfn} ($C_{\max}=\sup_{k\geq 1}\|\psi_k\|_{\infty}<\infty$) implies: 
\begin{align*}
	\frac{\text{Tr}(\bar{\bfA}_{n}^{M})}{n}&=\frac{1}{nMN} \sum_{m=1}^M\sum_{k=1}^n  \|R_{\psi_k}(X^m) \|_{\R^{Nd}}^2\\
	&\leq \frac{1}{nMN} \sum_{m=1}^M\sum_{k=1}^n \sum_{i=1}^N \bigg\| \frac{1}{N}\sum_{j\neq i}\psi_k(r_{ij})\bfr_{ij} \bigg\|_{\R^d}^2 \\ 
	&\leq C_{\max}^2<\infty. 
\end{align*}
 Moreover, when $\gamma\in(0,1/2]$ we have by \eqref{Eq:phi_gam}
\begin{align*}
	\frac{\gphi(\gamma)}{1-\gphi(\gamma)}\leq \frac{2\gamma^2}{1-\gphi(\gamma)}  \quad \text{ and } \quad \log(1+2/\gamma) \leq \log\left(\frac{5}{4\gamma^2}\right)\,.
\end{align*}
{Letting $c_{\gamma}:=\frac{1}{1-\gphi(\gamma)}$, one can note that $1\leq c_{\gamma}\leq 2$ and $c_{\gamma}\to 1$ as $\gamma \to 0$.} Therefore, we have for every  $\gamma\in(0,1/2]$ and $u>0$
\begin{align}
	&\inf_{v\in S^{n-1}} \innerp{\bar{\bfA}_{n}^{M} v, v} \nonumber \\
	&\geq \frac{c_{\gamma}}{\lambda}\bigg[\lambda c_{\bar\calL}-\frac{\lambda^2 \kappa}{2}- \frac{n}{M}\log\left(\frac{5}{4\gamma^2}\right)-u\bigg]-2c_{\gamma}C_{\max}^2 \gamma^2  \nonumber \\
	&=c_{\gamma}\left[ c_{\bar\calL}-{\frac{\lambda  \kappa }{2}}-\frac{n}{\lambda  M}\log\left(\frac{5}{4\gamma^2}\right)-\frac{u}{\lambda}-2C_{\max}^2 \gamma^2 \right]=:G_{u}^M(\gamma,\lambda) \label{Ineq:LowBD_Abar2} 
\end{align}
holds with probability at least $1-e^{-Mu}$.

\smallskip 
\noindent\textbf{Step 3:} Select $\lambda, \gamma$ properly to obtain the probability bound for $\lambda_{\min}(\bar{\bfA}_{n}^{M})=\inf_{v\in S^{n-1}} \innerp{\bar{\bfA}_{n}^{M} v, v}$.  
Based on \eqref{Ineq:LowBD_Abar2}, we have 
\begin{align}\label{Ineq:Prob_lam_Low}
\P\left\{\lambda_{\min}(\bar{\bfA}_{n}^{M}) \leq \sup_{\gamma, \lambda} G_{u}^M(\gamma,\lambda)\right\} \leq e^{-Mu}	\,.
\end{align}
Choosing {$\gamma^2=\frac{c_{\bar\calL}}{4C_{\max}^2}\leq \frac 14$ and $\lambda =\frac{c_{\bar\calL}}{2c_{\gamma}\kappa }$} in \eqref{Ineq:LowBD_Abar2}, then writing  $C_{0,n/M}=\frac{n}{M}\log\left(\frac{5C_{\max}^2}{c_{\bar\calL}}\right)$ in short, we have 
\begin{align*}
    G_{u}^M(\gamma,\lambda)&=c_{\gamma}\cdot\left[{\frac{c_{\bar\calL}}2}-\frac{\lambda  \kappa }{2}-\frac{C_{0,n/M}}{\lambda }-\frac{u}{\lambda}\right]\\
    &= c_{\gamma}\cdot\left[{\frac{c_{\bar\calL}}2}-2\sqrt{\frac{\kappa}{2}(C_{0,n/M}+u)}\ \right]\,,
\end{align*}
with the choice of $\lambda=\sqrt{\frac{C_{0,n/M}+u}{\frac{\kappa}{2}}}$. 

Letting {$G_{u}^M(\gamma,\lambda)=\frac{1}{2}(1-\varepsilon)c_{\bar\calL}$ for any $\varepsilon>0$}, namely
\begin{align*}
	u=\frac{c_{\bar\calL}^2}{8 c_{\gamma}^2 \kappa}[c_{\gamma}-1+\varepsilon]^2 -C_{0,n/M}\,,
\end{align*}
we have by PAC Bayesian inequality \eqref{Ineq:Prob_lam_Low} that 
\begin{equation}\label{Ineq:Prob_lam_Low_2}
	\begin{aligned}
		&\P\left\{\lambda_{\min}(\bar{\bfA}_{n}^{M}) \leq \frac{1}{2}(1-\varepsilon)c_{\bar\calL} \right\} \\
		&\leq e^{-Mu}=\exp\rbracket{M C_{0,n/M}-{\frac{Mc_{\bar\calL}^2}{8c_{\gamma}^2 \kappa }}\left[c_{\gamma}-1+\varepsilon\right]^2 } \\
	&=\exp\rbracket{n\log\left(\frac{5C_{\max}^2}{c_{\bar\calL}}\right)-{\frac{\calC_0^2 M c_{\bar\calL}^2}{8 \kappa N^2}}}\,,
	\end{aligned}
\end{equation}
where we denote $\calC_0=\frac{1}{c_{\gamma}}\left(c_{\gamma}-1+\varepsilon\right)$.  Then notice that 
\begin{align*}
     \calC_0=\frac{c_{\gamma}-1+\varepsilon}{c_{\gamma}}\geq \frac{\varepsilon}{2} \,,
\end{align*}
by $1\leq c_{\gamma}\leq 2$. 
The inequality \eqref{Ineq:Prob_lam_Low_2} and the range of $\calC_0$ imply that
\begin{align*}
    \P\left\{\lambda_{\min}(\bar{\bfA}_{n}^{M}) \leq \frac{1}{2}(1-\varepsilon)c_{\bar\calL} \right\} &\leq  \exp\rbracket{n\log\left(\frac{5C_{\max}^2}{c_{\bar\calL}}\right)-{\frac{\varepsilon^2 M c_{\bar\calL}^2}{16 \kappa}}}
\end{align*}
which is an exponential decay tail in $M$. We conclude the proof.
\end{proof}

\subsection{Proof of Theorem \ref{cor:L2rho_upper_case2} under non-uniformly bounded basis functions}\label{subsec:case2} 
In this section, we present the proof of Theorem \ref{cor:L2rho_upper_case2} under Assumption {\rm\ref{assum:L4-eigenfn}}. 
Similar to the {uniformly bounded basis} case, we need a fourth-moment bound lemma for the {non-uniformly bounded basis} case.
\begin{lemma}[Fourth-moment bounds: non-uniformly bounded basis]\label{lemma:mmtBd_case2}
Let $\widebar{\bfA}_{n}^\infty = \E[\widebar{\bfA}_{n}^{M}] $ and $\widebar{\bfb}_{n}^\infty = \E[\widebar{\bfb}_{n}^{M}]$, where $\widebar{\bfA}_{n}^{M}$ and $\widebar{\bfb}_{n}^{M}$ are defined in Eq. {\rm \eqref{eq:Ab}}. Let $\btheta_n^*=(\theta_1^*,\ldots,\theta_n^*)$ be the first $n$ coefficients of the true function $\phi_*$. 
{Under Assumption {\rm\ref{assum:L4-eigenfn}} with $\delta\geq 0$ and assuming $\beta>\delta+1/2$}, we have
\begin{subnumcases}{\label{Ineq:4MofAb_case2}}
    \Big(\E\big[\| ( \widebar{\bfA}_{n}^{M}-\widebar{\bfA}_{n}^\infty)\btheta^*_n\|_{\R^n}^4\big]\Big)^{\frac 12}  \leq C_A \frac n M \,; \label{Ineq:4MofA_case2}\\ 
		\Big(\E\big[\| \widebar{\bfb}_{n}^{M}-\widebar{\bfb}_{n}^\infty\|_{\R^n}^4\big]\Big)^{\frac 12} \leq C_b \frac n M \,,\label{Ineq:4Mofb_case2}
\end{subnumcases}
where the constants $C_A = \sqrt{96}C_{\delta}^2C_{\delta,\beta}^2 \widetilde{L}^2 L^2/{\pi^{2\beta}}$ and $C_b = \sqrt{798} \widetilde{L}^2 (C_{\delta}^4  C_{\delta,\beta}^4L^4/{\pi^{4\beta}} +\frac{1}{N^2}C_{\eta} )^{1/2}$ with $C_{\delta} >0$ in Eq. \eqref{cond:sqrt_n}, $C_{\delta,\beta}=(\sum_{k=1}^\infty k^{2(\delta-\beta)})^{1/2}$ are independent of $n$ and $M$, and $c_\eta$ is a constant related to the noise. 
\end{lemma}
\begin{proof}
	We follow the proof of Lemma \ref{lemma:mmtBd_v2} in Appendix \ref{sec:proof_upperRate} above. The main task is to replace the use of the uniform bound $C_{\max}=\sup_{k} \|\psi_k\|_{\infty}$ by the $n^\delta$-growth condition.
	
	We only need to adopt the estimates of $Z_A(l)$ and $\xi(l)$ after the inequality \eqref{Ineq:Za} in the proof of the Lemma \ref{lemma:mmtBd_v2} presented in Appendix \ref{sec:proof_upperRate}. Let Assumption {\rm\ref{assum:L4-eigenfn}} hold with $\delta\geq 0$ and $\beta>\delta+1/2$. Note that for $\phi_n^*= \sum_{k=1}^n \theta_k^* \psi_k\in H^{\beta}_{\rho}(L)$, we have 
	\begin{align*}
		\|\phi_n^*\|_\infty &\leq C_{\delta} \sum_{k=1}^n|\theta_k^*|k^{\delta} \leq C_{\delta}  \left(\sum_{k=1}^n k^{2(\delta-\beta)}\right)^{1/2} \left(\sum_{k=1}^n k^{2\beta}|\theta_k^*|^2 \right)^{1/2} \leq C_{\delta}  C_{\delta,\beta}L/{\pi^{\beta}}\,,
	\end{align*}
 where $C_{\delta,\beta}=(\sum_{k=1}^\infty k^{2(\delta-\beta)})^{1/2}\leq \frac{2(\beta-\delta)}{2(\beta-\delta)-1} <\infty $ since $\beta>\delta+1/2$. 
	Using Jensen's inequality and the uniform $L^4_{\rho}$-bound for the basis in \eqref{cond:L4-eigenfn} , we get
	\begin{align*}
	 \E[|Z_A(l)|^4]& =\E\bigg[ \Big| \frac 1 N  \sum_{i=1}^N \frac{1}{(N-1)^2} \sum_{j\neq i} \sum_{j'\neq i}\phi_n^*(r_{ij})\psi_l(r_{ij'}) \innerp{\br_{ij}, \br_{ij'}}_{\R^d}  \Big|^4 \bigg] \\
	 &\leq  \frac{1}{N(N-1)}  \sum_{i=1}^N  \sum_{j\neq i}\E\Big[| \psi_l(r_{ij})|^4   \Big] \|\phi_n^*\|_\infty^4 \leq C_{\delta}^4  C_{\delta,\beta}^4 \| \psi_l\|_{L_{\rho}^4}^4 L^4/{\pi^{4\beta}} \leq C_{\delta}^4  C_{\delta,\beta}^4 \widetilde{L}^4 L^4/{\pi^{4\beta}}\,.
	\end{align*}
	Thus, we obtain $\E [\|Z_A\|_{\ell^4}^4]=\sum_{l=1}^n \E [|Z_A(l)|^4] \leq  C_{\delta}  C_{\delta,\beta} n \widetilde{L}^4 L^4/{\pi^{4\beta}}$ and then 
	\begin{align*}
		\E[ \| (\widebar{\bfA}_{n}^{M}-\widebar{\bfA}_{n}^\infty)\btheta^*_n \|_{\R^n}^4 ]\leq  \frac{96 n^2}{M^2}  C_{\delta}^4 C_{\delta,\beta}^4  \widetilde{L}^4 L^4/{\pi^{4\beta}}\,.
	\end{align*}
	Taking square root,  we obtain \eqref{Ineq:4MofA_case2} with $C_A = \sqrt{96}C_{\delta}^2C_{\delta,\beta}^2 \widetilde{L}^2 L^2/{\pi^{2\beta}}$. 
	
	Similarly, we can get 
	\begin{align*}
		|\xi(l)|\leq \frac{C_{\delta}  C_{\delta,\beta}L}{N(N-1) {\pi^{\beta}}}  \sum_{i=1}^N  \sum_{j\neq i}  | \psi_l(r_{ij})|\,, \text{ and } \E[| \xi(l)|^4] \leq C_{\delta}^4  C_{\delta,\beta}^4 \widetilde{L}^4 L^4/{\pi^{4\beta}}\,.
	\end{align*}
	We use the Cauchy-Schwarz inequality to bound the noise term above as follows.
	\begin{align*}
	\E[ | \widetilde{\eta} (l)|^4 ]&= 
	\frac{1}{N^4}\E[|\innerp{R_{\psi_l}[X], \boldeta}_{\R^{Nd}}|^4] \\
	&\leq \frac{1}{N^4}\E[\|R_{\psi_l}[X]\|^4_{\R^{Nd}} \|\boldeta\|^4_{\R^{Nd}}] \\
	&\leq \frac{1}{N^2} C_{\eta} \sup_{k\geq 1} \|\psi_k\|_{L^4_{\rho}}^4\leq \frac{1}{N^2} C_{\eta}\widetilde{L}^4 \,,
	\end{align*}
	Combining these bounds, we have, for $1\leq l\leq n$, 
	\begin{align*}
	\E[|\bfb_n(l)|^4] &\leq \E[|\xi(l) +  \widetilde{\eta} (l)|^4]\leq 2^3 \E\Big[ |\xi(l)|^4  + |\widetilde{\eta} (l)|^4\Big] \\ 
	&\leq 2^3 \Big[C_{\delta}^4  C_{\delta,\beta}^4 \widetilde{L}^4 L^4/{\pi^{4\beta}}  + \frac{1}{N^2} C_{\eta}\widetilde{L}^4\Big] \leq 2^3 \widetilde{L}^4[C_{\delta}^4  C_{\delta,\beta}^4L^4/{\pi^{4\beta}}  + \frac{1}{N^2}C_{\eta} ]. 
	\end{align*}
	Consequently, applying the following Lemma \ref{lemma:4th_mmt_empirical_mean}, we obtain \eqref{Ineq:4Mofb_case2}
\begin{align*}
	\E[\|\widebar{\bfb}_{n}^{M}-\widebar{\bfb}_{n}^\infty\|_{\R^n}^4]&\leq \frac{798 n^2}{M^2} \widetilde{L}^4[C_{\delta}^4  C_{\delta,\beta}^4L^4/{\pi^{4\beta}}  + \frac{1}{N^2}C_{\eta} ] 
	\leq C_b^2\frac{n^2}{M^2} 
\end{align*}
with $C_b = \sqrt{798} \widetilde{L}^2 (C_{\delta}^4  C_{\delta,\beta}^4L^4/{\pi^{4\beta}} +\frac{1}{N^2}C_{\eta} )^{1/2}$. 
\end{proof}

Next, to prove the bound for the variance, we need a bound for the left tail probability for $\{\lambda_{\min}(\bar{\bfA}_{n}^{M}) \leq (1-\varepsilon) c_{\bar\calL}\}$ as in Lemma \ref{thm:min-eigen_C} or Lemma \ref{thm:min-eigen}. Under Assumption {\rm\ref{assum:L4-eigenfn}} on the basis functions, we only need to change the $C_{\max}$'s in left tail probabilities to be $C_{\delta} n^{\delta}$. We omit the proof since the main change is to replace the use of $C_{\max}$ by the $\delta$-power growth condition \eqref{cond:sqrt_n} in the proof. 

\begin{lemma}[Left tail probabilities: non-uniformly bounded basis]\label{thm:min-eigen_B_case2}
	Consider the norm matrix $\bar{\bfA}_n^M$ as defined in Eq. \eqref{eq:Ab_A}  associated with the basis functions $\{\psi_k\}_{k= 1}^{\infty}$.
	\begin{enumerate}
		\item[(1)] If Assumption {\rm\ref{assum:L4-eigenfn}} on the basis functions holds, then we have
	\begin{equation}\label{Eq:Low_Min_est_C_case2}
		\P\left\{ \lambda_{\min}(\bar{\bfA}_{n}^{M}) \leq (1-\varepsilon) c_{\bar\calL} \right\}\leq n  \sbracket{ \frac{e^{-\varepsilon}}{(1-\varepsilon)^{1-\varepsilon}}}^{\frac{c_{\bar\calL}  M}{C_{\delta}^2 n^{1+2\delta}}}\,, 
	\end{equation}
	for any $\varepsilon\in (0,1)$.
	\item[(2)] In addition, if Assumption {\rm \ref{assum:L2L4}} also holds, then we have, for any $\varepsilon\in(0,1)$, 
 \begin{equation}\label{Eq:Low_Min_est_case2}
     \P\left\{\lambda_{\min}(\bar{\bfA}_{n}^{M}) \leq \frac{1-\varepsilon}{2}c_{\bar\calL} \right\} \leq  \exp\rbracket{n\log\left(\frac{5C_{\delta}^2n^{2\delta}}{c_{\bar\calL}}\right)-{\frac{\varepsilon^2 M c_{\bar\calL}^2}{16 \kappa}}} \,, 
 \end{equation}
 where {$M\geq \frac{16 \kappa}{c_{\bar\calL}^2 \varepsilon^2}n\log\left(\frac{5C_{\delta}^2n^{2\delta}}{c_{\bar\calL}}\right)$} and $n\geq 2$.
	\end{enumerate}
\end{lemma}

\begin{remark}
	Analogously to Lemma \ref{thm:min-eigen_B}, under the Assumption {\rm\ref{assum:L4-eigenfn}}, the matrix Bernstein inequality implies that
	\begin{equation}\label{Eq:Low_Min_est_B_case2}
		\P\left\{ \lambda_{\min}(\bar{\bfA}_{n}^{M}) \leq (1-\varepsilon) c_{\bar\calL} \right\}\leq 2n \exp\rbracket{-\frac{M\varepsilon^2c_{\bar\calL}^2/4}{C_{\delta}^4 n^{2+4\delta} +C_{\delta}^2 n^{1+2\delta}\varepsilon c_{\bar\calL}/3}}\,,
	\end{equation}
	for any $\varepsilon\in (0,1)$.
\end{remark}

Notice that as the sample size increases, the first left tail probability \eqref{Eq:Low_Min_est_C_case2}, obtained by matrix Chernoff inequality, leads to an exponentially decaying bound when $M\geq C_{\delta}^2 n^{1+2\delta}/c_{\bar\calL}$. On the other hand, the second left tail probability \eqref{Eq:Low_Min_est_case2}, obtained by PAC-Bayes inequality, leads to an exponentially decaying bound when $M\geq C_{\delta, \bar\calL} n\log(n)$, for some constant $C_{\delta, \bar\calL}>0$. Thus, when taking $n\approx M^{\frac{1}{2\beta+1}}$ in the bias-variance tradeoff, we need $\beta>\delta$ and $\beta>0$ to have an exponentially decaying bound for the left tail probabilities \eqref{Eq:Low_Min_est_C_case2} and \eqref{Eq:Low_Min_est_case2}, respectively.  Both of the conditions $\beta>\delta$ and $\beta>0$ are weaker than the condition $\beta>\delta+1/2$ required in the fourth-moment bound in Lemma \ref{lemma:mmtBd_case2} (and therefore in Theorem \ref{cor:L2rho_upper_case2}). It is interesting to explore whether the fourth-moment bounds in \eqref{Ineq:4MofAb_case2} hold only when $\beta>\delta$.

Next, we have the following bound for the variance, similar to Lemma \ref{Lem:Ab_conv}. Again, we omit the proof since we only need to modify the term $G_{L,c_{\bar\calL}}(n,M)$ and the constant $C_0$ accordingly.

\begin{lemma}[Bound for variance: non-uniformly bounded basis]\label{Lem:Ab_conv_case2}
Under Assumption {\rm\ref{assum:L4-eigenfn}} on the basis functions with $\beta>\delta+1/2$, the following bound for the tamed LSE in Definition {\rm   \ref{def:tlse}} satisfies  
\begin{equation}\label{eq:errorL2E_case2}
	\E_{\phi_*}\Big[\Big\|\widehat \btheta_{n,M}- \btheta_{n}^*\Big\|^2_{\R^n}\Big] \leq C_0  c_{\bar\calL}^{-2} \frac n M + 2\epsilon_n^* +  G_{L,c_{\bar\calL}}(n,M) \,,
\end{equation}
where {$C_0= \sqrt{798}C_{\delta}^2C_{\delta,\beta}^2C_{\beta}^4\tilde{L}^2(\frac{C_\eta^{1/2}\pi^{2\beta}}{C_{\delta} C_{\delta,\beta} C_{\beta}^2L^2N} + 1)$}, $\epsilon_n^*= c_{\bar\calL}^{-2}  \sum_{l= n+1}^\infty |\theta_l^*|^2$ and 
\begin{align}\label{Def:G_Lc_case2}
	G_{L,c_{\bar\calL}}(n,M)= \frac{L^2n}{\pi^{2\beta}}  \sbracket{ \frac{e}{2}}^{\frac{c_{\bar\calL}  M}{2C_{\delta}^2 n^{1+2\delta}}}\,.
\end{align}
Moreover, if the fourth-moment Assumption \ref{assum:L2L4} also satisfies,  the bound in \eqref{eq:errorL2E_case2} holds with 
\begin{align}\label{Def:G_Lc2_case2}
	G_{L,c_{\bar\calL}}(n,M)=  \frac{L^2}{\pi^{2\beta}} \exp\rbracket{{n}\log\left(\frac{5C_{\delta}^2n^{2\delta}}{c_{\bar\calL}}\right)-{\frac{ M c_{\bar\calL}^2}{64 \kappa N^2}}}\,.
\end{align}
\end{lemma}

With Lemma \ref{lemma:mmtBd_case2}--\ref{Lem:Ab_conv_case2} in hand, we can prove Theorem \ref{cor:L2rho_upper_case2} now.

\begin{proof}[Proof of Theorem \ref{cor:L2rho_upper_case2}]
	Suppose Assumption {\rm 2.1},  Assumption {\rm (B1)} and Assumption {\rm 2.16} on the model hold. Let {Assumption {\rm\ref{assum:L4-eigenfn}} on the basis hold with $\delta\geq 0$}. As in the {uniformly bounded basis} case, we only need to focus on proving
\begin{equation}\label{ineq:upper_main_case2}
\limsup_{M\to\infty} \sup_{\phi_*\in H_{\rho}^{\beta}(L)}   \E_{\phi_*}\left[ M^{\frac{2\beta}{2\beta+1}} \| \widehat{\phi}_{n_M,M}-\phi_* \|_{L^2_\rho} ^2 \right] \leq C_{\textup{upper}}=2C_{\beta,L}  (C_0 c_{\bar\calL}^{-2})^\frac{2\beta}{2\beta+1}\,,
\end{equation}
with $n_M=\floor{M^{\frac{1}{2\beta+1}}}$. Applying the same bias-variance-concentration tradeoff in Proof of Proposition \ref{thm:L2rho_upper}, we have 
\begin{align}
  \E_{\phi_*}[ \| \widehat{\phi}_{n,M}-\phi_* \|_{L^2_\rho} ^2 ]
	&\leq (L+ 2 c_{\bar\calL}^{-2} ) n^{-2\beta}+  C_0  c_{\bar\calL}^{-2} \frac n M+ G_{L,c_{\bar\calL}}(n,M)\nonumber \\
	&=: g(n)+ G_{L,c_{\bar\calL}}(n,M) \,.\label{Eq:Trade-off_case2}
\end{align}
Notice that $g(n_M) \leq C M^{-\frac{2\beta}{2\beta+1}} $ with some positive constant $C$ independent of $M$. 
Under Assumption {\rm\ref{assum:L4-eigenfn}} on the basis functions with $\beta>\delta+1/2$, the term $G_{L,c_{\bar\calL}}(n_M,M)$ can be expressed as in \eqref{Def:G_Lc_case2}, which decays to zero exponentially. Thus, we prove Eq. \eqref{ineq:upper_main_case2}.
\end{proof}


\section{Fractional Sobolev spaces, Sobolev classes and H\"older class}\label{sec:append_Sob}
In this section, we recall some related fractional Sobolev spaces and their properties. We refer \cite{BSV2015DCDS,DiscretizationsFL2018, Hitchhiker2012,VariationalMethods2016} for more details. We only consider the case when $\rho(dx)=\rho'(x)dx$ and the density $\rho'$ satisfying $0<c\leq \rho'(x)\leq C<\infty$ for all $x\in[0,1]$.  Moreover, we define the Gagliardo (or classical) fractional Sobolev class $W_{\rho}^{\beta}(L)$ and the spectral fractional Sobolev class $H_{0,\rho}^{\beta}(L)$ by truncating the corresponding fractional Sobolev spaces. Namely, $W_{\rho}^{\beta}(L)$ and $H_{0,\rho}^{\beta}(L)$ are the balls with radius $L$ in the respective fractional Sobolev spaces.

\subsection{Classical fractional Sobolev spaces and fractional Sobolev classes}
When $\beta$ is an integer, we define the weighted Sobolev space $W_\rho^{\beta}$ as 
\begin{align}
	W_{\rho}^{\beta}=W_{\rho}^{\beta}([0,1])&:=\Big\{f\in L^2_\rho([0,1]): f^{(\beta-1)} \text{ is a.c.}, \|f^{(\beta)}|\|_{L^2_{\rho}} < \infty \Big\}. \label{Def:Sobolev}
\end{align}
where a.c. means absolutely continuous. And $W_{0,\rho}^{\beta}$ denote the closure of $C^{\infty}_0([0,1])$, the spaces of smooth functions with compact support on $(0,1)$, under the norm $\|\cdot\|_{W^{\beta}_{\rho}}$. So, inspired by the conventional weighted Sobolev classes  \cite[Definition 1.11]{tsybakov2008introduction}, we define (see also Remark {\rm \ref{rmk:Soblev_Holder}})
\begin{align*}
	W_{\rho}^{\beta}(L):=&\Big\{f\in L^2_\rho([0,1]): f^{(\beta-1)} \text{ is a.c.}, \int_0^1 |f^{(\beta)}(x)|^2\rho(dx) \leq L^2\Big\}\,,\\
	W_{0,\rho}^{\beta}(L):=&\Big\{ f\in W_{\rho}^{\beta}(L): f^{(j)}(0)=f^{(j)}(1)=0 \text{ for } j=0,\cdots,\beta-1 \Big\}. 
\end{align*}
This can be viewed as truncation versions of $W_{\rho}^{\beta}, W_{0,\rho}^{\beta}$. When $\rho$ is a uniform measure on $[0,1]$, we omit it and write $W_\rho^\beta$ as $W^\beta$. We use this abbreviation for all the functional spaces or classes. Recalling $0<c\leq \rho'(x)\leq C<\infty$ for all $x\in[0,1]$, it is not hard to see that $W_\rho^\beta=W^\beta$.

When $\beta$ is not an integer, for example, $\beta \in(0, 1)$, following the definition of classical fractional Sobolev space (see, e.g., \cite{Hitchhiker2012}) based on Gagliardo (semi)norm, we can define one dimensional fractional Sobolev space $W_{\rho}^{\beta}=W_{\rho}^{\beta}([0,1])$ as follows
\begin{equation}\label{Def:wSobolev}
	W_{\rho}^{\beta}:=W_{\rho}^{\beta}([0,1])= \Big\{ f\in L_{\rho}^2([0,1]): [f]_{W_{\rho}^{\beta}}< \infty \Big\}\,
\end{equation}
with a norm $\norm{f}_{W_{\rho}^{\beta}} :=\norm{f}_{L_{\rho}^2}+[f]_{W_{\rho}^{\beta}}$,  where the term 
\begin{align}\label{Eq:GagliardoNorm}
	[f]_{W_{\rho}^{\beta}}:=\rbracket{\int_0^1 \int_0^1 \frac{|f(x)-f(y)|^2}{|x-y|^{1+2\beta}}\rho(dx) \rho(dy)}^{\frac 12}
\end{align}
is a weighted Gagliardo (semi)norm of $f$. We shall call $W_{\rho}^{\beta}, W^{\beta}$ the Gagliardo fractional Sobolev spaces or classical fractional Sobolev spaces. 
Similarly, we can define the Gagliardo (or classical) fractional Sobolev class $W_{\rho}^{\beta}(L)$ as a truncation of $W_{\rho}^{\beta}$:
\begin{align}\label{Def:Gagliardo_FSC}
	W_{\rho}^{\beta}(L):=\Big\{ f\in L_{\rho}^2([0,1]): [f]_{W_{\rho}^{\beta}}^2= \int_0^1 \int_0^1 \frac{|f(x)-f(y)|^2}{|x-y|^{1+2\beta}}\rho(dx) \rho(dy) \leq L^2 \Big\}\,.
\end{align} 
Let $W^{\beta}_{0,\rho}=W^{\beta}_{0,\rho}([0,1])$ denote the closure of $C^{\infty}_0([0,1])$ in the norm $\|\cdot\|_{W^{\beta}_{\rho}}$ defined in Eq.  \eqref{Def:wSobolev}. We define the Gagliardo (or classical) fractional Sobolev class $W_{0,\rho}^{\beta}(L)$ when $\beta>1/2$ as:
\begin{equation}\label{Def:Gagliardo_FSC0}
	\begin{aligned}
	W_{0,\rho}^{\beta}(L):=\Big\{ f\in L_{\rho}^2([0,1]):& f \text{ is continuous on $[0,1]$ and } f(0)=f(1)=0 \,, \\
	& [f]_{W_{\rho}^{\beta}}^2= \int_0^1 \int_0^1 \frac{|f(x)-f(y)|^2}{|x-y|^{1+2\beta}}\rho(dx) \rho(dy) \leq L^2\Big\},
\end{aligned}
\end{equation}
Under the assumption of $\rho'$, it is clear that the weighted Sobolev norm and weighted Gagliardo semi-norm are equivalent to the unweighted ones. Then we abbreviate $W^{\beta}_{\rho}$ and $W^{\beta}_{0,\rho}$ to be $W^{\beta}$ and $W^{\beta}_{0}$ as before. 
It is worthwhile to mention that $W^{\beta}\neq W^{\beta}_0$ in general, see \cite[page 528]{Hitchhiker2012}. But as mentioned in \cite[page 5754]{BSV2015DCDS}, the Gagliardo fractional Sobolev spaces $W^{\beta}$ have no trace when $0<\beta\leq 1/2$, therefore, we have $W^{\beta}= W^{\beta}_0$.

The definition in Eq. \eqref{Eq:GagliardoNorm} cannot be plainly extended to the case $\beta\geq 1$. It is strictly related to the following identity (\cite[Eq.~(2.8)]{Hitchhiker2012}):
\begin{align*}
	\lim_{\beta\to 1^-} (1-\beta)\int_0^1 \int_0^1 \frac{|f(x)-f(y)|^2}{|x-y|^{1+2\beta}} dx dy= C\int_0^1 |\nabla f(x)|^2 dx\,, \forall ~ f\in W^1([0,1])
\end{align*}
with some positive constant $C$ independent of $\beta$. 
When $\beta>1$ and is not an integer, we define the fractional Sobolev spaces $W^{\beta}_{\rho}, W^{\beta}_{0,\rho}$ by constraints on the derivatives $f^{(\floor{\beta})}$, similar as in \cite[page 527]{Hitchhiker2012}. We omit the details.

\subsection{Fractional Laplacian, fractional Sobolev spaces and fractional Sobolev classes}
There are also other approaches to define the fractional Sobolev space $H_0^{\beta}([0,1])$ based on the power $\beta\geq 0$ of the Laplacian operator. See for example,  \cite[Section 2 and Section 3]{BSV2015DCDS}, \cite[Section 2]{DiscretizationsFL2018} and \cite[Chapter 1 and Chapter 5]{VariationalMethods2016}. Let $u(t,x):= e^{t\Delta}\phi(x)$ denote the solution of the heat equation with zero Dirichlet boundary condition:
\begin{align*}
	\begin{cases}
		\frac{\partial}{\partial t}u(t,x)=\Delta u(t,x)\,, &t> 0, x\in (0,1) \\
		\text{BC: } u(t,0)=u(t,1)=0 \,, & t\geq 0,\\
		\text{IC: } u(0,x)=\phi(x) & x\in [0,1]\,.
	\end{cases}
\end{align*}
Thus, by separation of variables, the function $u$ can be written as the following series:
\begin{align*}
	u(t,x)=\sum_{k=1}^\infty e^{-\lambda_k t}\theta_k e_k(x)\,, \theta_k=\innerp{\phi, e_k}_{L^2} \,, \quad t\geq 0, x\in [0,1]\,,
\end{align*}
where $\{(\lambda_k,e_k)\}_{k\geq 1}$ are the eigenpairs obtained from the Sturm-Liouville Eigenvalue Problem. In particular, 
\begin{align}\label{Spec:Eigen_pair}
	\lambda_k=\rbracket{k\pi}^2\,, e_k(x)=\sqrt{2} \sin(\sqrt{\lambda_k}x)=\sqrt{2}\sin\rbracket{k\pi x}\,.
\end{align}
The spectral definition of the fractional power of $-\Delta$ relies on the following formulas:
\begin{align}\label{Def:Spec_FLap}
	(-\Delta)^\beta \phi(x)=\sum_{k=1}^\infty \lambda_k^\beta \theta_k \psi_k(x) =\frac{1}{\Gamma(-\beta)}\int_0^\infty [e^{t\Delta}-Id]\phi(x)\frac{dt}{t^{1+\beta}}\,.
\end{align}
with $\Gamma(x)$ being the Gamma function.
\begin{remark}
	The fractional Laplacian $(-\Delta)^\beta$ obtained by using the spectral decomposition of the Laplacian is called (spectral) fractional Laplacian, see {\rm \cite{BSV2015DCDS}} for example. There is also much interest in the restricted fractional Laplacian ({\rm \cite[Section 2]{BSV2015DCDS}}). 
	The eigen-pairs $\{\lambda_k,e_k\}$ may differ when considering the spectral fractional Laplacian or restricted fractional Laplacian. Especially, we have eigen-pair {\rm \eqref{Spec:Eigen_pair}} for the spectral fractional Laplacian. 
\end{remark}
We define the spectral fractional Sobolev space $H_0^{\beta}=H_0^{\beta}([0,1])$ as
\begin{align}\label{Def:Spec_FSob}
	H_0^{\beta}([0,1]):=\Big\{\phi=\sum_{k=1}^\infty \theta_k e_k\in L^2([0,1]): \norm{(-\Delta)^{\beta/2}\phi}_{L^2}^2=\sum_{k=1}^\infty \lambda_k^{\beta}\theta_k^2 <\infty \Big\}
\end{align}
with a norm 
$$
\norm{\phi}_{H_0^\beta}:=\norm{(-\Delta)^{\beta/2}\phi}_{L^2}=\bigg[\sum_{k=1}^\infty \lambda_k^{\beta}\theta_k^2\bigg]^{1/2}=\pi^{\beta}\bigg[\sum_{k=1}^\infty k^{2\beta}\theta_k^2\bigg]^{1/2}\,.
$$
Moreover, $H_0^\beta([0,1])=W^{\beta}_0([0,1])$ when $\beta\in(0, 1/2)\cup (1/2,1)$. We refer \cite[Section 3.1]{BSV2015DCDS} and \cite[Chapter 5]{VariationalMethods2016} for more details. We summarize the relationship between fractional Sobolev spaces $H_0^{\beta}([0,1])$ and classical Sobolev spaces $W^{\beta}([0,1]), W^{\beta}_0([0,1])$ in the following:
\begin{equation}\label{Eq:Sob_relas}
\begin{aligned}
	H_0^{\beta}([0,1])=
	\begin{cases}
		W^{\beta}_0([0,1])=W^{\beta}([0,1])\,, &\text{when } 0<\beta< \frac 12\,;\\
		W^{\beta}_0([0,1])\subseteq W^{\beta}([0,1])\,, &\text{when } \frac 12<\beta< 1\,.
	\end{cases}
\end{aligned}
\end{equation}
When $\beta=1/2$, $H_0^{\beta}([0,1])$ can be identified as the Lions-Magenes space denoted by $W_{00}^{\frac 12}([0,1])$ characterized by an interpolation norm \cite[Eq.(116)]{BSV2015DCDS}.

If the weight $\rho$ admits a density $\rho'$ that is bounded below and bounded above, we can define the weighted fraction Sobolev space $H^\beta_{0,\rho}$ simply by replacing $e_k$ by $\psi_k=e_k/\sqrt{\rho'}$. It is clear $H^\beta_{0,\rho}=H_0^\beta$ since every function $f\in H^\beta_{0,\rho}$ can be identified as $f=\phi/\sqrt{\rho'}$ for some $\phi\in H_0^\beta$ and vice versa. 
We then define a truncation version of (spectral) fractional Sobolev space $H_0^\beta([0,1])$ as 
\begin{align}\label{Def:Spectral_FSC}
	H_0^{\beta}(L)=\left\{\phi = \sum_{k=1}^\infty \theta_k e_k \in L^2([0,1]): \theta_k=\innerp{\phi,e_k}_{L^2}, \forall k\geq 1 \text{ and }\sum_{k=1}^\infty \lambda_k^{\beta}\theta_k^2 \leq L^2 \right\}
\end{align}
with eigen-pair $\{(\lambda_k,e_k)\}_{k\geq 1}$ defined in \eqref{Spec:Eigen_pair}. 
The weighted fractional Sobolev class $H^\beta_{0,\rho}(L)$ in Definition \ref{def:wSobolev} can be viewed as a truncation of the space $H^\beta_{0,\rho}$. 

Following the spirit of relationship \eqref{Eq:Sob_relas}, we should expect $H_0^{\beta}(L)\subseteq W^{\beta}(Q)$ when $\beta\in(0,1)$ and $H_0^{\beta}(L)\subseteq  W^{\beta}_0(Q)$ when $\beta\in (1/2,1)$ for some positive constant $Q$ may depend on $\beta,L$. The reverse inclusions, i.e.,  $W^{\beta}_0(Q)\subseteq H_0^{\beta}(L)$ and $W^{\beta}(Q) \subseteq H_0^{\beta}(L)$ present greater challenges especially at the transition point of $\beta=1/2$. Future investigations will delve into this direction.
\begin{lemma}\label{Lem:spectral<Gagliardo}  
	Let $\beta\in(0,1)$, $L>0$ and $H_0^{\beta}(L)$ be the spectral fractional Sobolev class defined in \eqref{Def:Spectral_FSC}. There exist $Q=Q_{\beta,L}$ depending on $\beta,L$ and a Gagliardo fractional Sobolev class $W^\beta(Q)$ defined in \eqref{Def:Gagliardo_FSC} such that 
	\begin{align*}
    H_0^{\beta}(L) \subseteq 	W^{\beta}(Q)\,.
	\end{align*}
	Moreover, when $\beta\in(1/2,1)$ we have 
	\begin{align*}
	H_0^{\beta}(L) \subseteq 	W^{\beta}_0(Q)
	\end{align*}
	where $W^{\beta}_0(Q)$ denotes the Gagliardo fractional Sobolev class with zero boundary conditions defined in \eqref{Def:Gagliardo_FSC0}.
\end{lemma}
\begin{proof} 
	Let $\phi(x)=\sum_{k=1}^{\infty} \theta_k e_k(x)\in H_0^{\beta}(L)$. So we have $\sum_{k=1}^{\infty} \lambda_k^{\beta}\theta_k^2\leq L^2$. We proceed to show that 
	\begin{align*}
		\int_0^1\int_0^1 \frac{|\phi(x)-\phi(y)|^2}{|x-y|^{1+2\beta}}dxdy \leq Q^2
	\end{align*}
	with $Q=Q_{\beta,L}$ to be determined. 	
	Changing of variables implies that 
\begin{align*}
	\int_0^1\int_0^1 \frac{|\phi(x)-\phi(y)|^2}{|x-y|^{1+2\beta}}dxdy
	&=\int_0^1\int_{x-1}^{x} \frac{|\phi(x)-\phi(x-h)|^2}{|h|^{1+2\beta}}dhdx \\
	&\leq \int_0^1\int_{-1}^{1} \frac{|\phi(x)-\phi(x-h)|^2}{|h|^{1+2\beta}}dhdx \\
	&=\int_0^1\int_{0}^{1} \frac{|\phi(x)-\phi(x-h)|^2}{h^{1+2\beta}}dhdx \\
	&+\int_0^1\int_{0}^{1} \frac{|\phi(x)-\phi(x+h)|^2}{h^{1+2\beta}}dhdx \,.
\end{align*}
It suffices to prove the following two inequalities:
\begin{align}\label{Ineqs:Gagliardo}
	\int_0^1\int_{0}^{1} \frac{|\phi(x)-\phi(x-h)|^2}{h^{1+2\beta}}dhdx \leq \frac{Q^2}{2}\,, \text{ and } \int_0^1\int_{0}^{1} \frac{|\phi(x)-\phi(x+h)|^2}{h^{1+2\beta}}dhdx \leq \frac{Q^2}{2}\,.
\end{align}
It is clear that we only need to show the first inequality. The second one can be done similarly. By Parseval's identity, we have
\begin{align*}
	\int_0^1 |\phi(x)-\phi(x-h)|^2 dx = \sum_{k=1}^{\infty} \theta_k^2 |e^{2\pi h k}-1|^2=2\sum_{k=1}^{\infty} \theta_k^2 [1-\cos(2\pi hk)]\,.
\end{align*}
Therefore we get
\begin{align*}
	\int_0^1\int_{0}^{1} \frac{|\phi(x)-\phi(x-h)|^2}{h^{1+2\beta}}dhdx &= 2\sum_{k=1}^{\infty}\theta_k^2 \int_0^1 \frac{1-\cos(2\pi hk)}{h^{1+2\beta}}dh \\
	&\leq 2\sum_{k=1}^{\infty}\theta_k^2 (2\pi k)^{2\beta} \int_0^{\infty} \frac{1-\cos(h)}{h^{1+2\beta}}dh \\
	&= 2C_{\beta} \sum_{k=1}^{\infty} \lambda_k^{\beta} \theta_k^2 \leq 2C_{\beta} L^2\,,
\end{align*}
where $C_{\beta}=2^{2\beta} \int_0^{\infty} \frac{1-\cos(h)}{h^{1+2\beta}}dh$ is a positive constant. Letting $Q^2=4C_{\beta} L^2$, we have the first inequality in Eq. \eqref{Ineqs:Gagliardo} holds. Thus, we get $H_0^{\beta}(L) \subseteq 	W^{\beta}(Q)$.

By the H\"older regularity result of fractional Sobolev space (e.g. \cite[Theorem 8.2]{Hitchhiker2012}), we know that $\phi\in H^{\beta}(L) \subseteq W^{\beta}(Q)\subseteq \calC^{\beta-\frac 12}([0,1])$ when $\beta\in(1/2,1)$. Thus, we need to show more that if $\phi\in H^{\beta}(L)$ then $\phi(0)=\phi(1)=0$. For any $\phi=\sum_{k=1}^{\infty} \theta_k e_k \in H^{\beta}(L)$, let $\phi_{N}=\sum_{k=1}^{N}\theta_k e_k$. It is easy to see that $\phi_N(0)=\phi_N(1)=0$. All we need to do is to prove $\phi_N \to \phi$ uniformly on $[0,1]$. 
By Cauchy-Schwarz inequality, it is clear that for any $x\in[0,1]$
\begin{align*}
	|\phi(x)-\phi_{N}(x)|&=|\sum_{k=N+1}^{\infty} \theta_k e_k(x)| \\
	&\leq \bigg( \sum_{k=N+1}^{\infty} \theta_k^2 \lambda_k^{\beta} \bigg)^{\frac 12} \bigg( \sum_{k=N+1}^{\infty} \lambda_k^{-\beta} |e_k(x)|^2 \bigg)^{\frac 12} \\
	&\leq  \frac{\sqrt{2} Q}{\pi^{\beta}} \bigg( \sum_{k=N+1}^{\infty} \frac{1}{k^{2\beta}} \bigg)^{\frac 12} \to 0
\end{align*}
as $N\to \infty$. Thus, we can conclude that $H_0^{\beta}(L) \subseteq 	W^{\beta}_0(Q)$.
\end{proof}

\subsection{Fractional Sobolev classes and H\"older class}
Let us recall the H\"older class in Definition {\rm \ref{def:H\"older}}. For $\beta,L>0$, the H\"older class $\calC^{\beta}(L)$ on $[0,1]$ is 
\begin{equation}\label{Def:Holder-class}
\begin{aligned}
	\calC_0^{\beta}(L)=\Big\{\phi:[0,1] \to \R :& | \phi^{(l)}(x)-\phi^{(l)}(y)|\leq L| x -y |^{\beta-l}, \forall x,y \in [0,1] \,, \\
	& \phi^{(j)}(0)=\phi^{(j)}(1)=0, j=0,\cdots,l-1 \Big\},
\end{aligned}
\end{equation}
 where $\phi^{(j)}$  denotes the $j$-th order derivative of functions $f$ and $l =\lfloor\beta\rfloor$. If $\beta$ is an integer and $\rho(dx)=\rho'(x)dx$ with density $\rho'$ is bounded below and bounded above, the weighted Sobolev class $H_{0,\rho}^{\beta}(L)$ is equivalent to $W_{0,\rho}^{\beta}(L)$ by the same proof for {\rm\cite[Proposition 1.14]{tsybakov2008introduction}} when the basis functions are the weighted trigonometric functions. Combining these two facts, we obtain that $\calC_0^{\beta}(L)\subseteq H_{0,\rho}^{\beta}(L)$.
If $\beta$ is not an integer, for example $\beta\in(0,1)$, one can observe that $\phi\in \calC_0^{\beta}(L)$ is not guaranteed to belong to $W_{\rho}^\beta$ defined in \eqref{Def:wSobolev}. Therefore $\phi$ is not necessary in $H_{0,\rho}^{\beta}(L)$ which is a subset of $W_{\rho}^\beta$. Instead, an improvement on the regularity of the H\"older class is needed to ensure inclusion in a fractional Sobolev class. In particular, we have the following result.
\begin{lemma}\label{Lem:Holder<spectral}   
	Let $\beta\in\bigcup_{l=0}^{\infty} (l,l+1)$, $L>0$ and $H_0^{\beta}(L)$ be the spectral fractional Sobolev class defined in Eq. \eqref{Def:Spectral_FSC}. Then for any $\alpha\in(\beta,1)$, there exist $Q=Q_{\alpha, \beta,L}$ depending on $\alpha, \beta,L$ and H\"older class $\calC_0^{\alpha}(Q)$ defined in \eqref{Def:Holder-class} such that 
	\begin{align*}
    \calC_0^{\alpha}(Q) \subseteq H_0^{\beta}(L)\,.
	\end{align*}
\end{lemma}

\begin{proof}
Without loss of generality, we only need to discuss $\beta\in(0,1)$. Let $L>0$, $H_0^{\beta}(L)$ and $\phi\in \calC_0^\alpha(Q),$ with $\beta<\alpha<1$ and $Q=Q_{\alpha, \beta,L}$ to be determined. This is, we assume
\begin{align}\label{Ineq:phi_C_alpha}
	\sup_{x,y\in [0,1]}\frac{|\phi(x)-\phi(y)|}{|x-y|^{\alpha}}\leq Q\,, \text{ and } \phi(0)=\phi(1)=0\,.
\end{align}
We need to show that
\begin{align*}
	\norm{(-\Delta)^{\beta/2} \phi}_{L^2}=\int_0^1 \bigg| \frac{1}{\Gamma(-\beta/2)}\int_0^\infty [e^{t\Delta}-Id]\phi(x)\frac{dt}{t^{1+\frac{\beta}{2}}} \bigg|^2 dx \leq L^2\,.
\end{align*}
Notice that
\begin{align*}
	\frac{1}{\Gamma(-\beta/2)}\int_0^\infty [e^{t\Delta}-Id]\phi(x)\frac{dt}{t^{1+\frac{\beta}{2}}} = \frac{\Gamma(\beta/2)}{\Gamma(1-\beta/2)}\int_0^\infty \bigg[\phi(x)-\int_0^1 g_{t}(x,y)\phi(y)dy \bigg] \frac{dt}{t^{1+\frac{\beta}{2}}}\,,
\end{align*}
where 
\begin{align*}
	g_{t}(x,y) = \frac{1}{\sqrt{4\pi t}} \sum_{n=-\infty}^\infty \bigg[ \exp\bigg(-\frac{(x-y-2n)^2}{4t}\bigg)-\exp\bigg(-\frac{(x+y-2n)^2}{4t}\bigg) \bigg]
\end{align*}
is the one-dimensional Green function of the Heat equation on $[0,1]$. Since $g_{t}(x,y)$ is good kernel and \eqref{Ineq:phi_C_alpha}, we have
\begin{align*}
	\norm{(-\Delta)^{\beta/2} \phi}_{L^2}^2=&\int_0^1 \bigg| \frac{\Gamma(\beta/2)}{\Gamma(1-\beta/2)}\int_0^\infty \int_0^1 g_{t}(x,y)\big[\phi(x)-\phi(y)\big]dy  \frac{dt}{t^{1+\frac{\beta}{2}}} \bigg|^2 dx \\
	\leq& \int_0^1 \bigg| \frac{\Gamma(\beta/2)}{\Gamma(1-\beta/2)}\int_0^\infty \int_0^1 |g_{t}(x,y)|\big|\phi(x)-\phi(y)\big|dy  \frac{dt}{t^{1+\frac{\beta}{2}}} \bigg|^2 dx \\
	\leq& Q^2 \int_0^1 \bigg| \frac{\Gamma(\beta/2)}{\Gamma(1-\beta/2)}\int_0^\infty \int_0^1 |g_{t}(x,y)|\cdot |x-y|^{\alpha} dy  \frac{dt}{t^{1+\frac{\beta}{2}}} \bigg|^2 dx\,.
\end{align*} 
Then, employing the Gaussian upper bound for Heat kernel $g_t(x,y)$:
\begin{align*}
	|g_t(x,y)|\leq \frac{C_1}{\sqrt{t}}\exp\bigg(-\frac{C_2(x-y)^2}{t}\bigg)
\end{align*}
with some positive universal constants $C_1,C_2$ for every $t>0$ and $x,y\in[0,1]$. We proceed to get
\begin{align*}
	\int_0^1 |g_{t}(x,y)||x-y|^{\alpha} dx \leq& \min\bigg\{ \int_{\R} \frac{C_1}{\sqrt{t}}\exp\bigg(-\frac{C_2(x-y)^2}{t}\bigg) dx, \\
	&\qquad\quad \int_{\R} \frac{C_1}{\sqrt{t}}\exp\bigg(-\frac{C_2(x-y)^2}{t}\bigg) |x-y|^{\alpha} dx \bigg\} \\
	=&\min\bigg\{\int_{\R} \frac{C_1}{\sqrt{t}}\exp\bigg(-\frac{C_2 x^2}{t}\bigg) dx , C_1 \int_{\R} \exp\bigg(-\frac{2C_2 x^2}{t}\bigg)|x|^{\alpha} \frac{dx}{\sqrt{t}} \bigg\}\\
	\leq& C_{\alpha} (1\wedge t^{\frac{\alpha}{2}})
\end{align*}
with a universal positive constant $C_{\alpha}$ only depends on $\alpha$ and $C_1,C_2$. We then get
\begin{align*}
	\norm{(-\Delta)^{\beta/2} \phi}_{L^2}^2 \leq& C_{\alpha}^2 L^2 \bigg| \frac{\Gamma(\beta/2)}{\Gamma(1-\beta/2)}\bigg|^2 \int_0^1 \int_0^\infty (1\wedge t^{\frac{\alpha}{2}})\frac{dt}{t^{1+\frac{\beta}{2}}} dx \\
	\leq&  C_{\alpha,\beta}^2 Q^2
\end{align*}
for some positive constant $C_{\alpha,\beta}$ when $\alpha>\beta$. Letting $Q=L/C_{\alpha,\beta}$, we finish the proof.
\end{proof}

We summarize all the notations of functional spaces and classes in the following Table \ref{table:notationsSob}.
\begin{table}[h!]
\caption{Notations for Sobolev, fractional Sobolev, and H\"older spaces and classes.}
\label{table:notationsSob}
\centering
\begin{tabular}{|c||c|c|c|c|}
    \hline 
      & \textbf{Integer Sobolev} & \textbf{Gagliardo Sobolev} & \textbf{Spectral Sobolev} & \textbf{H\"older} \\ 
    \hline \hline
    \textbf{Spaces}    & $W_{\rho}^\beta, W_{0,\rho}^\beta$ ($\beta\in\mathbb{N}$)    & $W_{\rho}^\beta, W_{0,\rho}^\beta$ ($\beta\notin\mathbb{N}$)     & $H_{\rho}^\beta, H_{0,\rho}^\beta$    & $\calC^{\beta}, \calC^{\beta}_0$     \\ 
    \hline
    \textbf{Classes}   & $W_{\rho}^\beta(L), W_{0,\rho}^\beta(L)$ ($\beta\in\mathbb{N}$)     & $W_{\rho}^\beta(L), W_{0,\rho}^\beta(L)$ ($\beta\notin\mathbb{N}$)     & $H_{\rho}^\beta(L), H_{0,\rho}^\beta(L)$     & $\calC^{\beta}(L), \calC^{\beta}_0(L)$     \\ 
    \hline
\end{tabular}
\end{table}
In the specific case of a uniform measure $\rho$ on the interval [0, 1], we adopt a simplified notation for Sobolev spaces and classes by omitting the explicit reference to $\rho$. For instance, $W_{\rho}^{\beta}$ is simplified to be $W^{\beta}$, and $H_{\rho}^{\beta}(L)$ is represented as $H^{\beta}(L)$.

\section{Technical results and proofs for the lower bound}\label{sec:append2}
In this section, we present the remaining proofs of Lemma \ref{lemma:construction} and the main result Theorem \ref{thm:L2rho_lower}. 

\subsection{Constructions of the hypotheses for the lower bound and Proof of Lemma \ref{lemma:construction}}\label{Sec:Constructions}
We prove Lemma \ref{lemma:construction} by directly constructing the hypothesis functions $\{\phi_{0,M}, \cdots, \phi_{K,M}\}$ satisfying Conditions \ref{Cond:a}--\ref{Cond:c} there, that is, they are H\"older-continuous, $2s$-separated in $L^2_\rho$, and they induce hypotheses satisfying a Kullback-Leibler divergence upper bound.    

The construction consists of two steps:
\begin{itemize}[leftmargin=+.58in]
\item[Step 1:] construct $\bar K$ disjoint equidistance intervals with a proper length in support of the exploration measure $\rho$,  
\item[Step 2:] define the hypothesis functions as a linear combination of $\bar K$ functions supported in these disjoint intervals with binary coefficients, and prove that these hypothesis functions satisfy Conditions \ref{Cond:a}--\ref{Cond:c} in Lemma \ref{lemma:construction}.  
\end{itemize}
The second step largely follows the proof in \cite[page 103]{tsybakov2008introduction}, particularly, the Varshamov-Gilbert bound leads to an upper bound for the Kullback-Leibler divergence of the hypothesis. Our main innovation is the first step, constructing disjoint equidistance intervals in support of the measure $\rho$. Note that the lower bound case does not require the density function to be uniformly bounded below by a positive number.

We let $\psi\in \calC^\beta_0 (\frac12)\bigcap \calC^{\infty}(\R)$ be a bounded nonnegative smooth function: 
\begin{equation} \label{def:smooth-fn}
\psi(u ) = e\phi_0(2u), \quad  \phi_0(u) = e^{-\frac{1}{1-u^2}} \mathbf{1}_{|u| < 1}. 
\end{equation}
Note that $\psi({u})>0$ if and only if  ${u}\in(-1/2,1/2)$, and $\|\psi\|_{\infty} =\max_{x}\psi(x) =e \phi_0 (0) =  1$. 

We recall the Varshamov-Gilbert bound in \cite[Lemma 2.9]{tsybakov2008introduction} that we state here without proof. 
\begin{lemma}[Varshamov-Gilbert bound]\label{lemma:VG}
    Let $\bar K\geq 8$. Then there exists a subset $\{\omega^{(0)},\cdots, \omega^{(K)}\}$ of $~\Omega:=\{0,1\}^{\bar{K}}$ such that $\omega^{(0)}=(0,\cdots,0)$ and
    \begin{equation}
        K\geq 2^{\bar K/8}\,, \quad \text{and} \quad \rho_{H}(\omega^{(j)},\omega^{(k)})\geq \frac {\bar K}{8}\,, \forall ~ 0\leq j<k\leq K\,,
    \end{equation}
    where $\rho_{H}(\omega,\omega')=\sum_{l=1}^{\bar K}\mathbf{1}(\omega_l \neq \omega'_l)$ is called the \emph{Hamming distance} between two binary sequences $\omega=(\omega_1,\cdots,\omega_{\bar K})$ and $\omega'=(\omega'_1,\cdots,\omega'_{\bar K})$. 
\end{lemma}

\begin{proof}[Proof of Lemma \ref{lemma:construction}] 
The proof consists of two steps. 

{\bf Step 1:} we construct $\bar K=\lceil c_{0,N} M^{\frac{1}{2\beta+1}}\rceil$ disjoint equidistance intervals 
\begin{equation}\label{eq:h_M}
\{\Delta_\ell = (r_\ell-h_M, r_\ell+h_M)\}_{\ell=1}^{\bar K}, \quad \text{with } h_{M}=\frac{L_0}{8n_0\bar K}\,, 
\end{equation}
 where the numbers $\{r_\ell\}_{\ell=1}^{\bar K}$, $n_0$, and $L_0$ are to be specified next according to $\rho$ so that $\{r_\ell\}_{\ell=1}^{\bar K}\subseteq \supp(\rho) \bigcap [0,1]$ and $n_0\geq 1$. Here the constant $c_{0,N}=C_0 N^{\frac{1}{2\beta+1}}$ is defined in Eq. \eqref{eq:K}.
 
Note that if $\rho$ has a density function that is bounded from below by $\underline{a_0}>0$, we can simply use the uniform partition of $\supp(\rho)$ to obtain the desired $\{\Delta_l\}$. That is, we set $n_0=1$, $L_0= 4$, and $r_\ell=(2 \ell -1) h_M$. Since the density $\rho'$ may not be bounded below by a positive constant in general, we use the continuity of the density function as follows.

By Lemma \ref{lemma:rho}, the exploration measure $\rho$ has a density function $\rho'$ that is continuous on the interval $[0,1]$. Then, the quantity $a_0 = \sup_{r\in [0,1]} \rho'(r)$ is finite. Let $\underline{a_0} < a_0\wedge 1$ be a fixed constant.   

We proceed to construct the intervals in Eq. \eqref{eq:h_M} satisfying $\bigcup_{\ell=1}^{\bar K} \Delta_\ell \subseteq A_0:=\{r\in [0,1]:\rho'(r)>\underline{a_0}\}$. Let $ L_0:= \frac{1-\underline{a_0}}{a_0-\underline{a_0}}$.  Note that $\text{Leb}(A_0)\geq L_0$ since 
\begin{align*}
	1&=\int_0^1 \rho'(r)dr=\int_{A_0} \rho'(r)dr+\int_{A_0^c} \rho'(r)dr \\
	&\leq  a_0 \text{Leb}(A_0)+\underline{a_0}[1-\text{Leb}(A_0)]\,.
\end{align*}

Also, note that the set $A_0$ is open by the continuity of $\rho'$. Thus, there exist disjoint intervals $(a_j,b_j)$ such that $A_0= \bigcup_{j=1}^{\infty}(a_j,b_j)$. Without loss of generality, we assume that these intervals are ordered in ascending order according to their length $b_j-a_j$. Let 
\begin{align*}
	n_0 = \min\left\{n:\sum_{j=1}^{n}(b_j-a_j) > \frac{L_0}{2} \right\}\,.
\end{align*}
It is clear that $n_0\geq 1$. We begin by constructing disjoint intervals $\{\Delta_\ell = (r_\ell-h_M, r_\ell+h_M)\}_{\ell =1}^{n_1}\subseteq (a_1,b_1)$ such that $r_\ell = a_1+\ell h_M$ and $n_1= \floor{(b_1-a_1)/(2h_M)}$.  If $n_1\geq \bar K$, we stop. Otherwise, we continue by constructing additional disjoint intervals $\{\Delta_\ell = (r_\ell-h_M, r_\ell+h_M)\}_{\ell =n_1+1}^{n_1+n_2}\subseteq (a_2,b_2)$ similarly, and continue to $(a_j,b_j)$ until obtain $\bar K$ such intervals.  It remains to show that the cardinality of the constructed collection $\{\Delta_\ell\}$ can exceed $\bar K$. Let $K_*=n_1+n_2+\cdots$ denote the total number of intervals $\{\Delta_\ell\}_{\ell=1}^{K_*}$ to traverse all $\{(a_j, b_j)\}_{j=1}^{n_0}$ without stopping. It suffices to show that $K_*\geq \bar K$. 
Since the Lebesgue measure of $(a_j,b_j)\backslash \bigcup_{\ell=1}^{K_*} \Delta_\ell$ is less than $2h_M$ for each $j$, the Lebesgue measure of the uncovered parts $\bigcup_{j=1}^{n_0}(a_j,b_j) \backslash \big( \bigcup_{\ell=1}^{K_*} \Delta_\ell \big)$ is at most $2n_0h_M $. Thus, the intervals $\{\Delta_\ell\}_{\ell=1}^{K_*}$ must have a total length no less than $\frac{L_0}{2} -2n_0h_M $. Consequently, the total number must satisfy $K_*\geq (\frac{L_0}{2} -2n_0h_M  ) /(2h_M) = 2 \bar  K n_0- n_0 \geq \bar K $.

{\bf Step 2}: construct hypothesis functions satisfying Conditions \ref{Cond:a}--\ref{Cond:c} in Lemma \ref{lemma:construction}. We first define $2^{\bar K}$ functions, from which we will select a subset of 2s-separated hypothesis functions,   
\[ \phi_{\omega}(r)=\sum_{\ell=1}^{\bar K} \omega_\ell \psi_{\ell,M}(r), \quad \omega=(\omega_1,\cdots,\omega_{\bar K})\in \Omega=\{0,1\}^{\bar K}, \quad r\in[0,1]
\] 
where the basis functions are  
\begin{equation}
	\psi_{\ell,M}(r)
	:=L h_M^\beta \psi\bigg(\frac{r-r_\ell}{h_M} \bigg)\,, \quad \ell=1,\cdots, \bar K, \quad r\in[0,1]
\end{equation}
with $\psi(u)=e^{1-\frac{1}{1-(2u)^2}}\textbf{1}_{|u|< 1/2}$ as in Eq. \eqref{def:smooth-fn}.  Note that the support of $\psi_{\ell,M}(r)$ is $\Delta_\ell$, and $\|\psi_{\ell,M}\|_{L^2}=(\int_{\Delta_\ell}|\psi_{\ell,M}(r)|^2 dr)^{1/2} = L h_M^{\beta+\frac 12}\|\psi\|_{L^2}$. Recall that we use the notation $\|\cdot\|_{L^2}:=\|\cdot\|_{L^2([0, 1])}$ for brevity.
By definition, these hypothesis functions satisfy Condition \ref{Cond:a}, i.e., they are H\"older continuous.

Next, we select a subset of  $2s_{N,M}$-separated functions $\{\phi_{k,M}:=\phi_{\omega^{(k)}}\}_{k=1}^K$ satisfying Condition \ref{Cond:b}, i.e., $ \|\phi_{\omega^{(k)}}-\phi_{\omega^{(k')}}\|_{L^2_\rho} \geq 2s_{N,M}$ for any $k\neq k'\in \{1,\ldots,K\}$. Here $s_{N,M}= C_1 c_{0,N}^{-\beta} M^{-\frac{\beta}{2\beta+1}}$ with $C_1$ being a positive constant to be determined below. Since $\{\Delta_\ell = \supp(\psi_{\ell,M})\}_{\ell=1}^{\bar K}$ are disjoint, we have 
\begin{align*}
    \|\phi_{\omega}-\phi_{\omega'}\|_{L^2_\rho} &= \bigg(\int_0^1 \bigg|\sum_{\ell=1}^{\bar K} (\omega_\ell-\omega'_\ell) \psi_{\ell,M}(r) \bigg|^2 \rho'(r)dr\bigg)^{\frac 12} \\
        &= \bigg(\sum_{\ell=1}^{\bar K} (\omega_\ell-\omega'_\ell)^2 \int_{\Delta_\ell}|\psi_{\ell,M}(r)|^2 \rho'(r)dr\bigg)^{\frac 12}\,.
\end{align*}
Since $\rho'(r)\geq \underline{a_0}$ over each $\Delta_\ell$, we have 
\[ 
\int_{\Delta_\ell}|\psi_{\ell,M}(r)|^2 \rho'(r)dx\geq \underline{a_0}  \int_{\Delta_\ell}|\psi_{\ell,M}(r)|^2 dr = \underline{a_0}L^2 h_M^{2\beta+1}\|\psi\|_{L^2}^2.
\] 
Meanwhile, applying the Vashamov-Gilbert bound (\cite[Lemma 2.9]{tsybakov2008introduction}, see also Lemma \ref{lemma:VG}), one can obtain a subset $\{\omega^{(k)}\}_{k=1}^K$ with $K\geq 2^{\bar K/8}$ such that 
$ \sum_\ell^{\bar{K}} (\omega_\ell^{(k)}- \omega_\ell^{(k')})^2 \geq \frac{\bar{K}}{8} $ for any $k\neq k'\in \{1,\ldots,K\}$. Thus, 
\begin{align*}
     \|\phi_{\omega}-\phi_{\omega'}\|_{L^2_\rho} 
	&\geq \sqrt{\underline{a_0}}L h_M^{\beta+\frac 12}\|\psi\|_{L^2([0,1])} \bigg(\sum_{l=1}^{\bar K} (\omega_\ell-\omega'_\ell)^2 \bigg)^{\frac 12} \\
	&\geq \sqrt{\underline{a_0}}L h_M^{\beta+\frac 12}\sqrt{\bar K/8}= 2 C_N M^{-\frac{\beta}{2\beta+1}} = 2 s_{N,M}
\end{align*}
with 
 $C_N=\frac{1}{16} \sqrt{\underline{a_0}}L \left(\frac{L_0}{8 n_0}\right)^{\beta}  c_{0,N}^{-\beta}\,$
by recalling that $\bar K=\lceil c_{0,N} M^{\frac{1}{2\beta+1}}\rceil$ and $h_{M}=\frac{L_0}{8n_0\bar K}$ in Eq. \eqref{eq:h_M}.

\smallskip

To verify Condition \ref{Cond:c} for each fixed dataset $X^1,\cdots,X^M$, we first compute the Kullback divergence.  Recall that $r_{ij}^m=\|X^m_j-X^m_i\|_{\R^d}$ and $\br_{ij}^m=\frac{X^m_j-X^m_i}{r_{ij}^m}$, then  
$R_\phi[X^m]_i =\frac{1}{N-1} \sum_{j\neq i} \phi(r_{ij}^m) \br_{ij}^m$. 
By Assumption \ref{assump:noise-X} on the noise $\eta$ (i.e., being i.i.d. with a distribution $p_\eta$ satisfying $\int p_{\eta}(u) \log \frac{p_{\eta}(u)}
{p_{\eta}(u + v)}du \leq c_\eta \Vert v\Vert ^2$ for all $\Vert v\Vert\leq v_0$), we obtain 
\begin{align*}
	\KL(\bar{\P}_k,\bar{\P}_0) &= \int \cdots \int \log \prod_{m=1}^M \frac{p_\eta(u^m-R_{\phi_{k,M}}[X^m])}{p_{\eta}(u^m)} \prod_{m=1}^M[p_\eta(u^m-R_{\phi_{k,M}}[X^m])du^m] \\
	&= \sum_{m=1}^M \int \log \frac{p_{\eta}(v)}{p_\eta(v+R_{\phi_{k,M}}[X^m])} p_{\eta}(v)dv \\
    &\leq c_\eta  \sum_{m=1}^M \|R_{\phi_{k,M}}[X^m]\|^2_{\R^{Nd}}\,. 
\end{align*}
Employing Jensen's inequality, we have
\begin{align*}
	\|R_{\phi_{k,M}}[X^m]\|^2_{\R^{Nd}}=  \sum_{i=1}^N \Big\|\frac{1}{N-1} \sum_{j\neq i} \phi_{k,M}(r_{ij}^m) \br_{ij}^m \Big\|_{\R^d}^2 \leq   \sum_{i=1}^N \frac{1}{N-1} \sum_{j\neq i} |\phi_{k,M}(r_{ij}^m)|^2  \,.
\end{align*}
Recalling that $\phi_{k,M}(r_{ij}^m)=\sum_{\ell=1}^{\bar K} \omega^{(k)}_\ell \psi_{\ell,M}(r_{ij}^m)$, where $\text{supp} (\psi_{\ell,M})\subseteq \Delta_\ell$ are disjoint and $| \psi_{\ell,M}(r_{ij}^m) | =L h_M^\beta \psi\bigg(\frac{r_{ij}^m-r_\ell}{h_M} \bigg) \leq L h_M^\beta \|\psi\|_{\infty}\mathbf{1}_{\{r_{ij}^m \in \Delta _\ell\}} $, we have 
\[
|\phi_{k,M}(r_{ij}^m)|^2  	= \Big| \sum_{\ell=1}^{\bar K} \omega^{(k)}_\ell \psi_{\ell,M}(r_{ij}^m) \Big|^2  = \sum_{\ell=1}^{\bar K} \omega^{(k)}_\ell \big|\psi_{\ell,M}(r_{ij}^m)\big|^2
\leq L^2 h_M^{2\beta} \|\psi\|_{\infty}^2 \sum_{\ell=1}^{\bar K}  \mathbf{1}_{\{r_{ij}^m \in \Delta _\ell\}}, 
\]
where we have used the fact that $0\leq  \omega^{(k)}_\ell\leq 1$.

Combining the above three inequalities, we obtain 
\begin{align*}
	 \KL(\bar{\P}_k,\bar{\P}_0)   &\leq \frac { c_\eta L^2 h_M^{2\beta}}{N}   \psi_{\text{max}}^{2} \sum_{i,j=1;i\neq j}^N \sum_{m=1}^M \sum_{\ell=1}^{\bar K}  \mathbf{1}_{\{r_{ij}^m \in \Delta _\ell\}} \\ 
	 &\leq  c_\eta\psi_{\text{max}}^{2}L^2 N  M h_M^{2\beta}, 
\end{align*}
where the second ineqaulty follows from that $\sum_{m=1}^M \sum_{\ell=1}^{\bar K}  \mathbf{1}_{\{r_{ij}^m \in \Delta _\ell\}}\leq M $ since the intervals $\{\Delta_\ell\}$ are disjoint. Hence, recalling that $h_M=L_0/(8n_0\bar K) $ in Eq. \eqref{eq:h_M}, $K\geq 2^{\bar K/8}$, and $\bar K=\lceil c_{0,N} M^{\frac{1}{2\beta+1}}\rceil$, we obtain  
\begin{align*}
	\frac 1{K} \sum_{k=1}^{K} \KL(\bar{\P}_k,\bar{\P}_0)&\leq  c_\eta\psi_{\text{max}}^{2}L^2 N  M h_M^{2\beta}= c_\eta \psi_{\text{max}}^{2}L^2 N  M \bigg(\frac{L_0}{8n_0\bar K} \bigg)^{2\beta}\\ 
	&\leq c_\eta \psi_{\text{max}}^{2}L^2 N  (L_0/8n_0 )^{2\beta} c_{0,N}^{-2\beta -1}\bar K \leq \alpha \log (K)
\end{align*}
with 
\begin{equation}\label{Ineq:alpha_N}
 \alpha = 8 c_\eta \psi_{\text{max}}^{2}L^2 N  (L_0/8n_0 )^{2\beta} c_{0,N}^{-2\beta -1}. 
\end{equation}
To ensure $\alpha <\frac{1}{8}$ for all $N$, we need
\[
c_{0,N} > ( 64 c_\eta\psi_{\text{max}}^{2}L^2 (L_0/8n_0 )^{2\beta})^{\frac{1}{2\beta +1}} N^{\frac{1}{2\beta+1}} .
\] 
Setting $c_{0,N} = C_0 N^{\frac{1}{2\beta+1}}$ with $C_0= 2( 64 c_\eta \psi_{\text{max}}^{2}L^2 (L_0/8n_0 )^{2\beta})^{\frac{1}{2\beta +1}} $, we obtain the desired bound in  Condition \ref{Cond:c}.  
\end{proof}

\subsection{Proof of Theorem \ref{thm:L2rho_lower}}
We present the proof of the main Theorem of the lower bound: there exists a constant $c_{\textup{Lower}}>0$ independent of $M$ such that
\begin{equation}\label{ineq:lbd_main_ap}
 \liminf_{M\to \infty} \inf_{\widehat{\phi}_M\in L_{\rho}^2} \sup_{\phi_*\in \calC_0^{\beta}(L)}    \E_{\phi_*}[  M^{\frac{2\beta}{2\beta+1}} \| \widehat{\phi}_M-\phi_* \|_{L^2_\rho} ^2 ]\geq c_{\textup{Lower}}
\end{equation}
for $\beta>0$.

\begin{proof}[Proof of Theorem \ref{thm:L2rho_lower}]
The proof consists of three steps. We will briefly write $C_N=C_1 c_{0,N}^{-\beta}$ with $c_{0,N} = C_0 N^{\frac{1}{2\beta+1}}$,  $C_0= 2( 64 c_\eta \psi_{\text{max}}^{2}L^2 (L_0/8n_0 )^{2\beta})^{\frac{1}{2\beta +1}} $ and $C_1=\frac{1}{16} \sqrt{\underline{a_0}}L \left(\frac{L_0}{8 n_0}\right)^{\beta}$. Then $s_{N,M}$ in Condition \ref{Cond:b} can be written as \[s_{N,M}= C_N M^{-\frac{\beta}{2\beta+1}}\,.\]

{\bf Step 1}: Reduction to bounds in probability for a finite number of hypothesis functions. Leveraging the hypothesis functions $\{\phi_{k,M}\}_{k=0}^M$ and the Markov inequality, we have
\begin{align*}
    \sup_{\phi_*\in \calC_0^{\beta}(L)}  \E_{\phi_*} [M^{\frac{2\beta}{2\beta+1}}\|\widehat{\phi}_M-\phi_* \|_{L^2_\rho}^2] 
	&\geq \max_{\phi\in\{\phi_{0,M},\cdots,\phi_{K,M}\} }   \E_{\phi} [M^{\frac{2\beta}{2\beta+1}}\|\widehat{\phi}_M-\phi \|_{L^2_\rho}^2] \notag	\\
	&\geq C_N^2 \max_{\phi\in\{\phi_{0,M},\cdots,\phi_{K,M}\} } \P_{\phi} \bigg( \|\widehat{\phi}_M-\phi \|_{L^2_\rho} \geq 
	s_{N,M}	\bigg) 
\end{align*}
Recall that the equalities $\P(A) = \E[\mathbf{1}_A] = \E[\E[ \mathbf{1}_A\mid Z] ]= \E[\P(A|Z)]$ hold for any measurable set $A$. Then, we can reduce Eq. \eqref{ineq:lbd_main_ap} to bounds in probability by 
\begin{align}
    &\sup_{\phi_*\in \calC_0^{\beta}(L)}  \E_{\phi_*} [M^{\frac{2\beta}{2\beta+1}}\|\widehat{\phi}_M-\phi_* \|_{L^2_\rho}^2] \nonumber \\
	&\geq C_N^2 \frac{1}{K+1}\sum_{k=0}^K \E_{X^1,\cdots,X^M} \Big[ \P_{k}\big( \|\widehat{\phi}_M-{\phi_{k,M}} \|_{L^2_\rho} \geq s_{N,M} | X^1,\cdots,X^M \big) \Big] \notag \\
	&= C_N^2 \E_{X^1,\cdots,X^M}\bigg[ \frac{1}{K+1} \sum_{k=0}^K \P_{k}\big( \|\widehat{\phi}_M-{\phi_{k,M}} \|_{L^2_\rho} \geq s_{N,M} | X^1,\cdots,X^M \big)\bigg]\,. \label{eq:step1}
\end{align} 
We remark that $\{X^1,\cdots,X^M\}$ inside the expectation are fixed and can be treated as deterministic values in the conditional probability. 

\smallskip

{\bf Step 2}: Transform to bounds in the average probability of testing error of the $2s_{N,M}$-separated hypotheses. 
Define $\kappa_{\text{test}}:\Omega\to \{0,1,\ldots,M\}$ the minimum distance test 
\[ 
\kappa_{\text{test}} = \argmin_{0\leq k\leq K} \| \widehat \phi_M-\phi_{k,M}\|_{L^2_\rho}.
\] 
Then,  if $\kappa_{\rm{test}} \neq k$, we have $\|\widehat{\phi}_M-\phi_{\kappa_{\text{test}},M}\|_{L^2_\rho} \leq  \|\widehat{\phi}_M-\phi_{k,M}\|_{L^2_\rho}$. 
Together with Condition \ref{Cond:b} in Lemma \ref{lemma:construction} (i.e., the functions in $\Theta$ are $2s_{N,M}$-separated) and  the triangle inequality, we obtain for $\kappa_{\rm{test}} \neq k$
\begin{equation}\label{KeyCond1}
\begin{aligned}
	2s_{N,M} \leq & \|{\phi_{k,M}}-\phi_{\kappa_{\text{test}},M} \|_{L^2_\rho} \\
	\leq & \|\widehat{\phi}_M-\phi_{\kappa_{\text{test}},M}\|_{L^2_\rho} +  \|\widehat{\phi}_M-\phi_{k,M}\|_{L^2_\rho}
\leq 2 \|\widehat{\phi}_M-\phi_{k,M}\|_{L^2_\rho}. 
\end{aligned}
\end{equation} 
 That is,  $\kappa_{\rm{test}} \neq k$ implies 
$ \|\widehat{\phi}_M-{\phi_{k,M}} \|_{L^2_\rho} \geq s_{N,M}$, and hence, $\P_k(\|\widehat{\phi}_M-{\phi_{k,M}} \|_{L^2_\rho} \geq s_{N,M}\mid  X^1,\cdots,X^M ) \geq \P(\kappa_{\rm{test}} \neq k \mid  X^1,\cdots,X^M)$. 
Consequently, we have
\begin{align}\label{eq:step2}
    &\frac{1}{K+1} \sum_{k=0}^K \P_{k}\big( \|\widehat{\phi}_M-{\phi_{k,M}} \|_{L^2_\rho} \geq s_{N,M} | X^1,\cdots,X^M \big) \nonumber \\
    \geq& \inf_{\kappa_{\text{test}}} \frac{1}{K+1} \sum_{k=0}^K \P_{k}\big(\kappa_{\text{test}} \neq k | X^1,\cdots,X^M \big)\nonumber \\
    =&\inf_{\kappa_{\text{test}}} \frac{1}{K+1} \sum_{k=0}^K \bar{\P}_{k}\big(\kappa_{\text{test}} \neq k \big)=:\bar{p}_{e,M},
\end{align}
where $ \bar{\P}_k(\cdot )={\P}_{\phi_{k,M}}(\cdot \mid X^{1},\ldots, X^{M})$. 
We call $\bar{p}_{e,M}$ the average probability of testing error.

\smallskip

{\bf Step 3}: Bound $\bar{p}_{e,M}$ from below.  Conditional on each data $\{X^{m}\}_{m=1}^M$, the Kullback divergence estimate \ref{Cond:c}: $\frac 1{K} \sum_{k=1}^{K} \KL(\bar{\P}_k,\bar{\P}_0)\leq \alpha \log(K)$ holds   with $0<\alpha<1/8$, and hence by Lemma \ref{thm:lb_hypothesis_KL} and the fact that $K=2^{\lceil c_{0,N} M^{\frac{1}{2\beta+1}}\rceil} $ in Eq. \eqref{eq:K} of Lemma \ref{lemma:construction} increases exponentially in $M$, we have  
\[
\bar{p}_{e,M}\geq \frac{\log(K+1)-\log(2)}{\log(K)}-\alpha \geq  \frac{1}{2}
\]
if $M$ is large. 
Note that the above lower bound of $\bar{p}_{e,M}$ is independent of the dataset $\{X^m\}_{m=1}^M$. Together with Eq. \eqref{eq:step1} in {\bf Step 1} and Eq. \eqref{eq:step2} in {\bf Step 2}, we obtain with $c_0=\frac 12 [C_1 C_0^{-\beta}]^2$
\begin{equation*}
	\sup_{\phi_*\in \calC_0^{\beta}(L)}  \E_{\phi_*} [\|\widehat{\phi}_M-\phi_* \|_{L^2_\rho}^2] \geq  \frac{C_N^2}{2}	M^{-\frac{2\beta}{2\beta+1}} =c_0  (NM)^{-\frac{2\beta}{2\beta+1}}
\end{equation*}
for any estimator. Hence, the lower bound \eqref{ineq:lbd_main_ap} follows. 
\end{proof}


\section{Example: the normal operator for the uniform distribution}\label{sec:unif-X}
We show in this section that the coercivity constant of the normal operator $\Lbar$ is exactly $c_{\Lbar}=\frac{N-1}{N^2}$ when $X$ has i.i.d.~components uniformly distributed on $[0,1]$. 

Recall that the normal operator can be written as  $\Lbar=\frac{(N-1)(N-2)}{N^2}  \calL_G + \frac{N-1}{N^2} I$, where the operator $\calL_G: L^2_{\rho}([0,1]) \to L^2_{\rho}([0,1])$ is defined by 
 \begin{equation}\label{eq:LG_ap}
 	\innerp{\calL_G\phi, \psi}_{L^2_\rho} = \E[\phi(r_{12})\psi(r_{13})\innerp{\br_{12},\br_{13}}_{\R}]\,, \forall \phi, \psi \in L^2_{\rho}\,.
 \end{equation} 
Thus, to prove that the coercivity constant is $c_{\Lbar}=\frac{N-1}{N^2}$, it suffices to verify that $\calL_G$ is compact. 

By the exchangeability of $X$, the exploration measure $\rho$ is 
\begin{align*}
	\rho(A) &  = \frac{1}{N(N-1)} \sum\nolimits_{i\neq j} \P{(|X_i-X_j|\in A)} =\P{(|X_1-X_2|\in A)}  
\end{align*}
for any measurable set $A\subseteq [0,1]$. Then, it is easy to see that $\rho$ has a density 
\begin{equation}\label{Eq:dens_Unif}
	\rho'(r)=(2-2r)\b1_{\{0\leq r\leq 1\}}\,.
\end{equation}

\begin{proposition}
Let $X= (X_1,X_2,X_3)$ with $X_i \overset{iid}{\sim} U([0,1])$. Then, the operator $\LGbar$ defined in \eqref{eq:LG_ap} is a compact integral operator with integral kernel
\begin{equation}\label{eq:G-kernel}
\begin{aligned}
G(r,s)&=\frac{\widetilde{G}(r,s)}{\rho'(r)\rho'(s)} \,,  \textup{ with }
	\widetilde{G}(r,s)=[2-(|r-s|+|r+s|)]-[2-2|r+s|]\b1_{\{r+s\leq 1\}}\,.
\end{aligned}
\end{equation} 
Consequently, the infimum of the eigenvalues of $\Lbar=\frac{(N-1)(N-2)}{N^2}  \calL_G + \frac{N-1}{N^2} I$ is $c_{\Lbar}=\frac{N-1}{N^2}$. 
\end{proposition}

\begin{proof}
We first prove that $\LGbar$ is an integral operator with $G$ as an integral kernel. Then, we show that it is a compact operator by showing that $G(r,s)\in L^2_\rho$.

Let us recall the notations  $r_{ij}=|X_i-X_j|$ and $\br_{ij}=\frac{X_i-X_j}{r_{ij}}$. We write $\Phi(X_i-X_j)=\phi(r_{ij})\br_{ij}$ and $\Psi(X_i-X_j)=\psi(r_{ij})\br_{ij}$ and then have
\begin{align}
	\innerp{\calL_G\phi, \psi}_{L^2_\rho} &= \E[\phi(r_{12})\br_{12} \psi(r_{13})\br_{13}] \nonumber \\
	&= \E[\Phi(X_1-X_2)\Psi(X_1-X_3)] \nonumber \\
	&=\int_{[0,1]^3} \Phi(x_1-x_2)\Psi(x_1-x_3) \prod_{i=1}^3 dx_i \,. \label{Eq:LG_Unif}
\end{align}

We introduce a change of variables: 

\begin{minipage}{.4\textwidth}
\begin{align*}[left = \empheqlbrace\,]
	x&=x_1-x_2\,; \\
	y&=x_1-x_3\,; \\
	z&=x_2+x_3\,;
\end{align*}
\end{minipage}
\begin{minipage}{.2\textwidth}
	\begin{align*}
		& \\
		\text{which is equivalent to} &  \\
		&
	\end{align*}
\end{minipage}	
\begin{minipage}{.4\textwidth}
\begin{align*}[left = \empheqlbrace\,]
	x_1&=\frac 12 (x+y+z)\,; \\
	x_2&=\frac 12 (-x+y+z)\,; \\
	x_3&=\frac 12 (x-y+z)\,.
\end{align*}
\end{minipage}
Thus, Eq. \eqref{Eq:LG_Unif} becomes
\begin{align*}
	\int_{[0,1]^3} \Phi(x_1-x_2)\Psi(x_1-x_3) \prod_{i=1}^3 dx_i=\frac 12 \int_{D} \Phi(x)\Psi(y) dx dy dz,
\end{align*}
where the cube $[0,1]^3$ is transformed to a region $D$ under the change of variables:
\begin{align*}
	D=\bigcup_{j=1}^4 D_j=\bigcup_{j=1}^4 \left\{(x,y,z):(x,y)\in B_j \right\}
\end{align*}
and the projected disjoint regions $\{B_j\}_{j=1}^4$ on $(x,y)$-plane are defined as follows
\begin{align*}
	B_1=\{(x,y): x\in[0,1], y\in[0,1]\}\,, &\quad B_2=\{(x,y): x\in[0,1], -1+x\leq y\leq 0\}\,,\\
	B_3=\{(x,y): x\in[-1,0], y\in[-1,0]\}\,, &\quad B_4=\{(x,y): x\in[-1,0], 0\leq y\leq 1+x\}\,. 
\end{align*}

The computations for $D_1$ and $D_3$ are similar, so are  $D_2$ and $D_4$. 

To compute $D_1$, we consider the decomposition 
\begin{align*}
	B_1=B_{11}\cup B_{12}:=\{(x,y)\in[0,1]^2: x>y\}\bigcup \{(x,y)\in[0,1]^2: x\leq y\}\,.
\end{align*}
Thus, let $D_1 = D_{11} \cup D_{12}$, where the projection of $D_{11}$ to $(x,y)$-plane corresponds to $B_{11}$ and the projection of $D_{12}$ to $(x,y)$-plane corresponds to $B_{12}$. Thus, we have 
\begin{align*}
	\int_{D_1} \Phi(x)\Psi(y) dx dy dz&=\int_{B_{11}}\int_{x-y}^{2-(x+y)}dz\cdot \Phi(x)\Psi(y) dx dy \\
	&+ \int_{B_{12}}\int_{y-x}^{2-(x+y)}dz\cdot \Phi(x)\Psi(y) dx dy \\
	&= \int_{[0,1]^2} \Phi(x)\Psi(y)\cdot 2[(1-x)\b1_{\{x>y\}}+(1-y)\b1_{\{x\leq y\}}] dxdy \,.
\end{align*}
Note that $\Phi(x)=\phi(x)\frac{x}{|x|}=\phi(x)$ and $\Psi(y)=\psi(y)\frac{y}{|y|}=\psi(y)$ on $\{(x,y)\in [0,1]\times[0,1]\}$. So, with the change of variables $r=x$ and $s=y$ we have
\begin{align*}
	\int_{D_1} \Phi(x)\Psi(y) dx dy dz &= \int_{[0,1]^2} \phi(r)\psi(s)\cdot 2[1-r\b1_{\{r>s\}}-s\b1_{\{r\leq s\}}] drds \\
	&=\int_{[0,1]^2} \phi(r)\psi(s)\frac{[2-(|r-s|+|r+s|)]}{\rho'(r)\rho'(s)}\rho(dr)\rho(ds),
\end{align*}
where $\rho'(r)=(2-2r)\b1_{\{0\leq r\leq 1\}}$. 

On $D_3$, we get the same formula similarly.

On the other hand, $D_2$ is a region with the projected domain on $(x,y)$-plane corresponds to $B_{2}$. Then, we get
\begin{align*}
	\int_{D_2} \Phi(x)\Psi(y) dx dy dz&=\int_{B_{2}}\int_{x-y}^{2+(y-x)}dz\cdot \Phi(x)\Psi(y) dx dy \\
	&=\int_0^1 \int_{-1}^0 \Phi(x)\Psi(y)\cdot 2(1+y-x)\b1_{\{-1\leq y-x\}}dydx \\
	&=-\int_{[0,1]^2} \phi(r)\psi(s) \frac{[2-2|r+s|]}{\rho'(r)\rho'(s)}\b1_{\{r+s\leq 1\}}\rho(dr)\rho(ds)\,.
\end{align*}
On $D_4$, we obtain the same formula in a similar manner. 

In conclusion, we get
\begin{align*}
	\E[\Phi(X_1-X_2)&\Psi(X_1-X_3)] \\
	&=\int_{[0,1]^3} \Phi(x_1-x_2)\Psi(x_1-x_3) \prod_{i=1}^3 dx_i=\frac 12 \int_{D} \Phi(x)\Psi(y) dx dy dz \\
	&=\int_{[0,1]^2} \phi(r)\psi(s)\frac{[2-(|r-s|+|r+s|)]}{\rho'(r)\rho'(s)}\rho(dr)\rho(ds) \\
	&\qquad -\int_{[0,1]^2} \phi(r)\psi(s) \frac{[2-2|r+s|]}{\rho'(r)\rho'(s)}\b1_{\{r+s\leq 1\}}\rho(dr)\rho(ds) \\
	&=\int_{[0,1]^2} \phi(r)\psi(s) G(r,s)\rho(dr)\rho(ds)
\end{align*}
with $G$ defined in Eq. \eqref{eq:G-kernel}.

At last, to show that $\LGbar$ is a compact operator, it suffices to show that $G(r,s)\in L^2_\rho$, which implies that $\LGbar$ is a Hilbert-Schmidt integral operator and therefore compact. Note that 
\begin{align*}
	\int_{[0,1]\times[0,1]}|G(r,s)|^2 \rho(dr)\rho(ds)=&\int_{[0,1]\times[0,1]} \frac{|\widetilde{G}(r,s)|^2}{4(1-r)(1-s)} drds \\
	=&2\int_{0\leq r< s\leq 1} \frac{|[1-s]-[1-|r+s|]\b1_{\{r+s\leq 1\}}|^2}{(1-r)(1-s)}drds \\
	\leq & 2\int_{0\leq r< s\leq 1}\frac{(1-s)}{(1-r)}drds=2\int_{0\leq s< r\leq 1} \frac{s}{r}drds \\
	=& \int_{0\leq r \leq 1}rdr=\frac{1}{2}\,.
\end{align*}

Note that the eigenvalues of $\Lbar$ are simply the eigenvalues of the compact operator $\LGbar$ added by $\frac{N-1}{N^2}$. Since the infimum of the eigenvalues of a compact operator is zero,  we obtain that the infimum of the eigenvalues of the normal operator $\Lbar$ is 
$c_{\Lbar}=\frac{N-1}{N^2}$.
\end{proof}

\smallskip
\section*{Acknowledgments}
The work of XW and FL is partially supported by AFOSR FA9550-20-1-0288. FL is partially funded by the Johns Hopkins Catalyst Award, NSF DMS-2238486, and AFOSR FA9550-21-1-0317. The work of IS is partially supported by the Israel Science Foundation grant 1793/20.  We would like to thank Professor Yaozhong Hu at the University of Alberta, Professor Junxi Zhang at Concordia University, and Professor Maruo Maggioni and Sichong Zhang at Johns Hopkins University for helpful discussions.

\bibliographystyle{plain}
\bibliography{ref_IPS_learning2310,ref_IPS_stochastic,ref_FeiLU2023_10,ref_lb_minimax,ref_regularization23_09}

\end{document}